\documentclass[english]{article}
\usepackage{preamble}

\begin{document}

\title{Chromatic Cyclotomic Extensions}
\maketitle

\begin{abstract}
    We construct Galois extensions of the $T(n)$-local sphere, lifting all finite abelian Galois extensions of the $K(n)$-local sphere. This is achieved by realizing them as higher semiadditive analogues of cyclotomic extensions. Combining this with a general form of Kummer theory, we lift certain elements from the $K(n)$-local Picard group to the $T(n)$-local Picard group. 
\end{abstract}

\begin{figure}[H]
    \centering{}
    \includegraphics[scale=0.55]{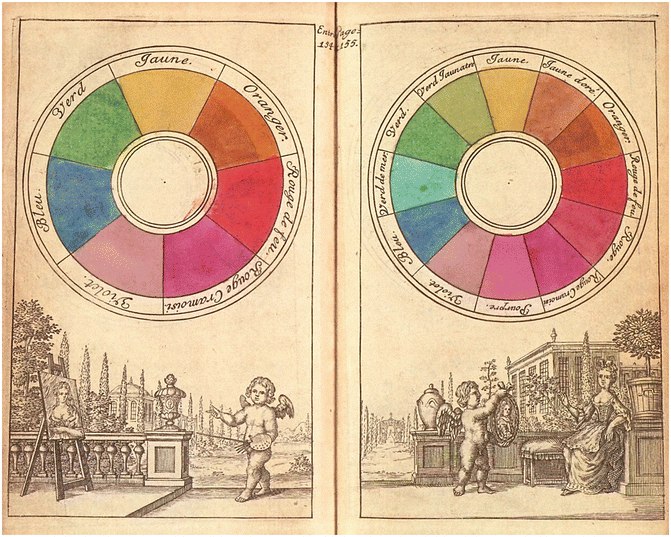}
    \caption*{Color circles from the enlarged 1708 edition of the Treatise on Miniature Painting, Claude Boutet.}
\end{figure}

\tableofcontents{}

\section{Introduction}

\subsection{Overview}
\subsubsection{Background \& main results}
Chromatic homotopy theory, as a general approach, proposes to study the $\infty$-category $\Sp_{(p)}$, of $p$-local spectra, via the ``chromatic height filtration''. In practice, there are two prominent candidates for the ``monochromatic layers'' for such a filtration.
The first, are the $K(n)$-localizations 
$L_{K(n)}\colon\Sp_{(p)} \to \Sp_{K(n)},$ 
where $K(n)$ is the Morava $K$-theory spectrum of height $n$.
The second, are the telescopic localizations $L_{T(n)}\colon\Sp_{(p)} \to \Sp_{T(n)},$ 
where $T(n)$ is obtained by inverting a $v_n$-self map of a finite spectrum of type $n$.
The two candidates are related by the inclusion $\Sp_{K(n)}\sseq \Sp_{T(n)}$, which is known to be an equivalence for $n=0,1$ by the works of Miller \cite{miller1981telescope} and Mahowald \cite{mahowald1981bo}. Whether this inclusion is an equivalence for all $n$ is an open question and is the subject of the celebrated ``telescope conjecture'' of Ravenel. 
On the one hand, the $\infty$-categories $\Sp_{T(n)}$ are fundamental from a structural standpoint, as they arise via the thick subcategory theorem \cite[Theorem 7]{nilp2}. They also admit a close connection to unstable homotopy theory \cite{heuts2021lie}, and figure in the ``redshift'' phenomena for algebraic K-theory \cite{ausoni2008chromatic,hahn2022redshift,land2020purity,clausen2020descent}. However, they are hard to access computationally. 
On the other hand, the $\infty$-categories $\Sp_{K(n)}$, which a priori might contain somewhat less information, still exert a large control over $\Sp_{(p)}$, due to the nilpotence theorem of Hopkins-Devinatz-Smith \cite[Corollary 5]{nilp2} and the chromatic convergence theorem of Hopkins-Ravenel (see \cite[Theorem 7.5.7]{ravenel1992nilpotence}). Moreover, they possess deep connections to the algebraic geometry of formal groups, and are consequently much more amenable to computations. 

One of the key instances of the relationship between the theory of formal groups and $\Sp_{K(n)}$, is the construction of the  Lubin-Tate $\mathbb{E}_\infty$-ring spectrum $E_n$\footnote{In this paper, we use the version of $E_n$ whose coefficients satisfy $\pi_0E_n \simeq W(\cl{\FF}_p)[[u_1,\dots,u_{n-1}]]$.} (see \cite{goerss2004moduli}, or alternatively \cite[Construction 5.1.1]{Lurie_Ell2}). Simply put, $E_n$ provides a faithful and relatively computable (highly structured) multiplicative (co)homology theory for $K(n)$-local spectra. Moreover, the cohomology operations of $E_n$ can be understood in terms of the Morava stabilizer group  
$\Morex_n = \widehat{\ZZ}\ltimes \aut(\cl{\Gamma}),$
where $\Gamma$ is a formal group law of height $n$ over $\FF_p$, and $\cl{\Gamma}$ is its base-change to $\cl{\FF}_p$.
From a more conceptual perspective, by the work of \cite{DH,RognesGal,BakerRichterGalois,AkhilGalois}, $E_n$ can be viewed as an ``algebraic closure''  of the $K(n)$-local sphere   $\Sph_{K(n)}$ in $\Sp_{K(n)}$, with $\Morex_n$ as its Galois group. Hence, as in ordinary commutative algebra, one can apply ``Galois decent'' to study the $\infty$-category $\Sp_{K(n)}$ in terms of the, far more tractable, $\infty$-category of $K(n)$-local $E_n$-modules. 

In light of that, it seems beneficial to study Galois extensions of $\Sph_{T(n)}$ in $\Sp_{T(n)}$ as well. In this regard, we have the following result.
\begin{theorem}[\Cref{Tele_Gal}]\label{Thm_A}
    Let $G$ be a finite abelian group. For every $G$-Galois extension $R$ of $\Sph_{K(n)}$ in $\Sp_{K(n)}$, there exists a $G$-Galois extension $R^f$ of $\Sph_{T(n)}$ in $\Sp_{T(n)}$, such that $L_{K(n)}R^f \simeq R$.
\end{theorem}
In particular, all of the Galois extensions of $\Sph_{K(n)}$, that are classified by finite quotients of the determinant map $\det\colon \Morex_n\to \ZZ_p^\times$, can be lifted to Galois extensions of $\Sph_{T(n)}$ in $\Sp_{T(n)}$. 
In fact, the lifting of the various abelian Galois extensions can be done in a compatible way. In the language of \cite{AkhilGalois}, the localization functor 
$L_{T(n)}\colon \Sp_{T(n)} \to \Sp_{K(n)}$
induces a map on the weak Galois groups (in the opposite direction), and we show that after abelianization this map admits a retract.  
The proof of \Cref{Thm_A} relies on the  $\infty$-semiadditivity of the $\infty$-categories $\Sp_{T(n)}$ (\cite[Theorem A]{TeleAmbi}), and the theory of ``higher cyclotomic extensions'', which we develop in this paper. The latter builds on the theory of semiadditive height and semisimplicity (\cite[Theorem D]{AmbiHeight}). 

A related structural invariant, which is better understood for $\Sp_{K(n)}$ than for $\Sp_{T(n)}$, is the \textit{Picard group}. Recall that for a symmetric monoidal $\infty$-category $\cC$, the Picard group $\Pic(\cC)$ is the abelian group of isomorphism classes of invertible objects in $\cC$ under tensor product. 
While $\Pic(\Sp_{K(n)})$ was intensively studied (e.g. \cite{hopkins1994constructions,lader2013resolution,goerss2015Picard,heard2015morava}), very little is known about $\Pic(\Sp_{T(n)})$. Our second main result concerns the construction of non-trivial elements in $\Pic(\Sp_{T(n)})$.
\begin{theorem}[\Cref{Tele_Pic_Odd}]\label{Thm_B}
    For an odd prime $p$, the group $\Pic(\Sp_{T(n)})$ admits a subgroup isomorphic to $\ZZ/(p-1)$. 
\end{theorem}
Moreover, under $K(n)$-localization, this subgroup is mapped isomorphically onto the subgroup of $\Pic(\Sp_{K(n)})$, consisting of objects which are $(p-1)$-torsion and of symmetric monoidal dimension (a.k.a. Euler characteristic) 1. We also construct some non-trivial elements in $\Pic(\Sp_{T(n)})$ for $p=2$, and describe their image in the algebraic Picard group (\Cref{Tele_Pic_Even}). We deduce \Cref{Thm_B} from \cref{Thm_A} by a generalized Kummer theory, which we develop in this paper. 

\subsubsection{Higher cyclotomic extensions}
We shall now outline our approach to \Cref{Thm_A}.
As mentioned above, the Galois extensions of $\Sph_{K(n)}$ are governed by the Lubin-Tate spectrum $E_n$, whose construction relies on the theory of complex orientations. In the absence of an analogue of this construction in the $T(n)$-local world, we take our cue from classical algebra, where we have a natural source of abelian Galois extensions -- the \textit{cyclotomic extensions}.  Namely, for a commutative ring $R$, we have the $m$-th cyclotomic extension 
\(
    R[\omega_m] := R[t]/\Phi_m(t),
\)
where $\Phi_m(t)$ is the $m$-th cyclotomic polynomial, whose roots are the primitive $m$-th roots of unity. If $m$ is invertible in $R$, this extension is Galois (though not necessarily connected) with respect to the natural action of $(\ZZ/m)^\times$.

As a concrete example, consider $R=\QQ_p$. Starting integrally, for every $d\in\NN$, the cyclotomic extension $\ZZ_p[\omega_{p^d-1}]$ splits into a product of copies of the $\ZZ/d$-Galois extension $W(\FF_{p^d})$, exhibiting the latter as a subextension of a cyclotomic one. After inverting $p$, we get $\QQ_p(\omega_{p^d-1})=W(\FF_{p^d})[p^{-1}]$, which assemble into the maximal unramified extension    
$\QQ_p^{\mathrm{ur}} := \bigcup_d \QQ_p(\omega_{p^d-1})$ of $\QQ_p$\footnote{We denote by $\QQ_p(\omega_m)$ the splitting field of $\Phi_m(t)$ over $\QQ_p$, as opposed to $\QQ_p[\omega_m]:=\QQ_p[t]/\Phi_m(t)$ which may be not connected, but rather a product of copies of $\QQ_p(\omega_m)$.}. However, as $p$ is now invertible, we have also the cyclotomic extensions of $p$-power order, which assemble into $\QQ_p(\omega_{p^\infty}) := \bigcup_r \QQ_p[\omega_{p^r}].$
It is a classical theorem in number theory that all \textit{abelian} Galois extensions of $\QQ_p$ can be obtained in this way. 
\begin{thm*}[Kronecker-Weber]
    We have 
    \(
        \Gal(\cl{\QQ}_p/\QQ_p)^\ab \simeq 
        \widehat{\ZZ}\times \ZZ_p^\times,
    \)
    where $\Gal(\cl{\QQ}_p/\QQ_p) \to \widehat{\ZZ}$ classifies $\QQ_p^{\mathrm{ur}}$, and the $p$-adic cyclotomic character $\chi\colon \Gal(\cl{\QQ}_p/\QQ_p) \to \ZZ_p^\times$ classifies $\QQ_p(\omega_{p^\infty})$.
\end{thm*}

Incidentally, for $1\le n< \infty$, we also have 
\(
    \Morex_n^\ab \simeq 
    \widehat{\ZZ}\times \ZZ_p^\times.
\)\footnote{For height $n=1$, this similarity was also discussed in \cite[\S5.5]{RognesGal}.}
The finite Galois extensions of $\Sph_{K(n)}$ that are classified by the map $\Morex_n \to \widehat{\ZZ}$, are the $K(n)$-localizations of the \textit{spherical Witt vectors} $\Sph W(\FF_{p^d})$ (see \cite[Example 5.2.7]{Lurie_Ell2}). Hence, just as for $\QQ_p$, they can be obtained from cyclotomic extensions of order prime to $p$ (see \Cref{K(n)_Local_Witt} and \Cref{Unramified_K(n)_Cyclo}). Similarly, the $T(n)$-localizations of $\Sph W(\FF_{p^d})$ constitute a lift of the said Galois extensions of $\Sph_{K(n)}$ to Galois extensions of $\Sph_{T(n)}$.
However, unlike for $\QQ_p$, the element $p$ is not invertible in $\Sp_{K(n)}$, and in fact, a $K(n)$-local commutative algebra \textit{can not} admit primitive $p$-power roots of unity (\cite[Theorem 1.3]{devalapurkar2020roots}). 
Nevertheless, and it is the main insight leading to the results of this paper, the higher semiadditivity of the $\infty$-categories $\Sp_{K(n)}$ (in the sense of \cite{AmbiKn}) allows one to view the Galois extensions classified by $\det\colon \Morex_n \to \ZZ_p^{\times}$, as a ``higher analogue'' of the classical cyclotomic extensions of $p$-power order. 
Furthermore, the higher semiadditivity of the $\infty$-categories $\Sp_{T(n)}$ (\cite[Theorem A]{TeleAmbi}), is what allows us to construct their $T(n)$-local lifts. We shall now explain these ideas in more detail. 

We begin by reformulating the construction of the (ordinary) cyclotomic extensions in a way which lands itself to $\infty$-categorical generalizations. Recall that a $p^r$-th root of unity in a commutative ring $R$, is a homomorphism $C_{p^r} \to R^\times,$ and it is called \textit{primitive}, if  it is nowhere (in the algebro-geometric sense) of order $p^{r-1}$ (see \Cref{Def_Roots_Of_Unity}). 
For a given $R$, the functor which assigns to every $R$-algebra, the set of its $p^r$-th roots of unity, is corepresented by the group algebra $R[C_{p^r}]$. Consider now the short exact sequence of abelian groups
\[
    0 \to C_p \to C_{p^r} \to C_{p^{r-1}} \to 0,
\]
the associated $R$-algebra homomorphism 
$f\colon R[C_{p^r}] \to R[C_{p^{r-1}}]$, and the set map $\iota\colon C_p \into R[C_{p^r}].$ If $p$ is invertible in $R$, we can define the idempotent
\[
    \varepsilon := 
    \frac{1}{p}\sum_{g\in C_p} \iota(g) \qin
    R[C_{p^r}],
\]
which splits $f$. That is, inverting $\varepsilon$ and $1-\varepsilon$ respectively, yields a decomposition 
\[
    R[C_{p^r}] \simeq R[C_{p^{r-1}}] \times R[\omega_{p^r}].
\]
In particular, we get that the cyclotomic extension $R[\omega_{p^r}]$ corepresents the set of \textit{primitive} $p^r$-th roots of unity. 
Furthermore, the natural action of the group $(\ZZ/p^r)^\times$ on $C_{p^r}$ induces an action on the group algebra $R[C_{p^r}]$, which then restricts to the cyclotomic extension $R[\omega_{p^r}]$ by the invariance of $\varepsilon$, making it a $(\ZZ/p^r)^\times$-Galois extension of $R$.
Reformulated in this way, the construction of the cyclotomic extensions can be carried out for a commutative algebra object $R$ in any additive symmetric monoidal $\infty$-category, provided that $p$ is an invertible element in the ring $\pi_0(R)$ (see \cite[Theorem 3]{schwanzl1999roots}). An extension of these ideas, which allows adjoining roots of any invertible element, was studied in \cite{lawson2020roots}.

To define \textit{higher} cyclotomic extensions, we first observe that for a commutative algebra $R$ in a symmetric monoidal $\infty$-category $\cC$, the group-like commutative monoid  (or equivalently, the connective spectrum) of units $R^\times$, need not be discrete in general. Taking advantage of that, we define a \textit{height $n$ root of unity} of $R$, to be a morphism of the form $C_{p^r} \to \Omega^nR^\times$. For such a higher root of unity, we also have a corresponding notion of primitivity (see \Cref{Def_Higher_Roots_Of_Unity}).  The functor assigning to each $R$-algebra the space of its height $n$ roots of unity is corepresented by the higher group algebra $R[B^nC_{p^r}]$. By analogy with the above, we consider the fiber sequence 
\[
    B^nC_p \to B^nC_{p^r} \to B^nC_{p^{r-1}},
\]
the associated morphism of commutative $R$-algebras 
$f\colon R[B^nC_{p^r}] \to R[B^nC_{p^{r-1}}]$,
and map of spaces 
\[
    \iota \colon B^nC_p \to \Map(\one_\cC, R[B^nC_{p^r}]).
\]
To proceed, assume now that $\cC$ is stable and $n$-semiadditive. This allows one to integrate families of morphisms in $\cC$ indexed by $n$-finite spaces. In particular, we can consider the \textit{cardinality} of the $n$-finite space $B^nC_p$. This is given by integrating over $B^nC_p$, the unit map $\one_\cC \oto{1_R} R$, and is denoted by
\[
    |B^nC_p| := \int_{B^nC_p} 1_R \qin 
    \pi_0(R) = \pi_0\Map(\one_\cC,R).
\]
Recall from \cite[Definition 3.1.6]{AmbiHeight}, that $R$ is said to be of height $\le n$, if the element $|B^nC_p|$ is invertible in $\pi_0(R)$, in which case we can define
\[
    \varepsilon := 
    \frac{1}{|B^nC_p|} \int_{x\in B^nC_p} \iota(x) \qin
    \pi_0(R[B^nC_{p^r}]).
\]
Note that for $n=0$ we have $|C_p|=p$, hence, $R$ is of height $0$ precisely when $p$ is invertible in the ring $\pi_0(R)$. We then show that, as in the case $n=0$, the element $\varepsilon$ is idempotent and induces a $(\ZZ/p^r)^\times$-equivariant decomposition 
\[
    R[B^nC_{p^r}] \simeq 
    R[B^nC_{p^{r-1}}] \times \cyc[R]{p^r}{n},
\]
such that the projection onto the first factor coincides with $f$ (\Cref{Transfer_Idempotent}). As a result, the commutative $R$-algebra $\cyc[R]{p^r}{n}$ corepresents the space of \textit{primitive} $p^r$-th roots of unity of height $n$ (\Cref{Cyclo_Rep_Primitive}). We call  $\cyc[R]{p^r}{n}$ the \textit{height $n$ cyclotomic extension} of $R$ of order $p^r$.

To apply the abstract construction of higher cyclotomic extensions to the chromatic world, we recall from \cite[\S 4.4]{AmbiHeight}, that the semiadditive height generalizes the chromatic height. Namely, all objects of $\Sp_{T(n)}$, and hence also of $\Sp_{K(n)}$, are of semiadditive height exactly $n$. We then prove that the resulting $(\ZZ/p^r)^\times$-equivariant algebras $\Sph_{K(n)}[\omega_{p^r}^{(n)}]$ are Galois (\Cref{Cyclo_Galois}). Furthermore, by comparing the infinite cyclotomic extension $\Sph_{K(n)}[\omega_{p^\infty}^{(n)}]$, with Westerland's $R_n$ \cite{Westerland}, we deduce that it is 
classified by\footnote{This requires one to choose a normalizable formal group law in the sense of \cite[Definition 5.3.1]{AmbiKn}. In \cite{Westerland},
Westerland uses the Honda formal group law, in which case one has to replace $\det$ with $\detpm$.} $\det\colon \Morex_n \to \ZZ_p^\times$
(see \Cref{Cyc_Char_K(n)} and the following discussion). Thus, the determinant map can be viewed as the higher chromatic analogue of the $p$-adic cyclotomic character. 
In the same spirit, the realization of all the abelian Galois extensions of $\Sph_{K(n)}$ in terms of (ordinary and higher) cyclotomic extensions can be viewed as the higher chromatic analogue of the Kronecker-Weber theorem. Finally, we deduce \Cref{Thm_A} from the above, by showing that the $T(n)$-local higher cyclotomic extensions $\cyc[\Sph_{T(n)}]{p^r}{n}$ are Galois as well (\Cref{Cyclo_Galois}), using the nilpotence theorem in the guise of ``nil-conservativity'' (see \cite[\S4.4]{TeleAmbi}).

\subsubsection{Kummer theory}

We now outline the relationship between abelian Galois extensions and the Picard spectrum, which allows us to deduce \Cref{Thm_B} from \Cref{Thm_A}. Classically, given a field $k$ which admits a primitive $m$-th root of unity, and a finite abelian group $A$ which is $m$-torsion, Kummer theory identifies the set of isomorphism classes of  $A$-Galois extensions of $k$, with $\Ext^1_\ZZ(A^*,k^\times),$ where $A^*=\hom(A,\QQ/\ZZ)$ is the Pontryagin dual of $A$\footnote{E.g., for $A=\ZZ/m$, we have $\Ext^1_\ZZ((\ZZ/m)^*,k^\times) = (k^\times)/(k^\times)^m$, so that cyclic Galois extensions of order $m$ are classified by invertible elements of the base field up to $m$-th powers, see \cite{birch1967cyclotomic}.}. 
One way to construct this identification is to observe that for every $A$-Galois extension $L/k$ we can simultaneously diagonalize the action of all the elements of $A$ on $L$, producing an eigenspace decomposition
\(
    L \simeq \bigoplus_{\varphi \in A^*} L_\varphi
\)
as $k$-vector spaces. The $L_\varphi$-s turn out to be all 1-dimensional, and the multiplication of $L$ restricts to give isomorphisms 
$L_\varphi \otimes L_\psi \iso L_{\varphi + \psi}.$
As a result, a choice of basis elements $0 \neq x_\varphi \in L_\varphi$ provides a 1-cocycle representative of a class in the group $\Ext^1_\ZZ(A^*,k^\times)$, which can then be shown to depend only on $L$ and to completely characterize it.

For a more general commutative ring $R$, which admits a primitive $m$-th root of unity (so, in particular, $m\in R^\times$), and an $A$-Galois extension $S$ of $R$, one can still produce a decomposition
\(
    S \simeq \bigoplus_{\varphi \in A^*} R_\varphi
\) 
and isomorphisms 
$R_\varphi \otimes R_\psi \iso R_{\varphi + \psi}$
as before. However, this only implies that the $R_\varphi$-s are \textit{invertible} $R$-modules, rather than that $R_\varphi \simeq R$. This leads to a classification of $A$-Galois extensions of $R$, which involves both the Picard group $\Pic(R)$ and the group of units $R^\times$. These groups can be recognized as the $\pi_0$ and $\pi_1$ respectively, of the \textit{Picard spectrum} of $R$, which we denote by $\pic(R)$. 

Generalizing this, we show that for every additive presentable symmetric monoidal $\infty$-category $\cC$, such that the commutative ring $\pi_0\one$ admits a primitive $m$-th root of unity, there is a homotopy equivalence of spaces (\Cref{Kummer_Theory})
\[
    \GalExt{\cC}{A} \simeq \Map_{\Sp^\cn}(A^*, \pic(\cC)),
\]
where on the left-hand side, we have the full subcategory (which turns out to be an $\infty$-groupoid) of $\calg(\cC)^{BA}$ consisting of $A$-Galois extensions of the unit $\one_\cC$. This can be considered as a general form of ``Kummer theory'' in the context of $\infty$-categories. The main difficulty in establishing the above homotopy equivalence is to handle the multiplicativity of the eigenspace decomposition coherently. To this end, we realize the eigenspace decomposition, under the above assumptions, as a symmetric monoidal equivalence (\Cref{DFT_Equiv})
\[
    \Fur \colon
    \Fun(BA, \cC)_{\mathrm{Ptw}} \iso 
    \Fun(A^*, \cC)_{\Day},
\]
where the subscript `$\mathrm{Ptw}$' indicates the usual point-wise symmetric monoidal structure, while the subscript `$\Day$' indicates the Day-convolution symmetric monoidal structure. This can be viewed as a general form of the discrete Fourier transform. 

Applying the above to $A=\ZZ/m$ and taking $\pi_0$, gives rise to a (non-canonically) split short exact sequence of abelian groups (\Cref{Kummer_Cyclic})
\[
    0 \to
    (\pi_0\one^\times) / (\pi_0\one^\times)^m \to 
    \pi_0 \GalExt{\cC}{\ZZ/m} \to 
    \Pic^{\re}(\cC)[m] \to 
    0,
\]
where $\Pic^\re(\cC) \le \Pic(\cC)$ is the subgroup of invertible objects of monoidal dimension $1$ (\Cref{Def_Pic_Even}), and $\Pic^\re(\cC)[m]$ is its $m$-torsion subgroup.
We note that when $\cC$ is $p$-complete for some odd prime $p$, the $\ZZ_p$-algebra $\pi_0\one$ always admits primitive $(p-1)$-st roots of unity. Thus, to every $\ZZ/(p-1)$-Galois extension of $\one$, corresponds a (possibly trivial) $(p-1)$-torsion element of $\Pic^\re(\cC)$. 

Specializing to the chromatic world, we show that the $K(n)$-local Picard object $Z_n$, which corresponds to the higher cyclotomic $\ZZ/(p-1)$-Galois extension $\cyc[\Sph_{K(n)}]{p}{n}$, generates the group (\Cref{Pic_Kn})
\[
    \Pic^\re(\Sp_{K(n)})[p-1] \simeq \ZZ/(p-1).
\]
We deduce that $\cyc[\Sph_{T(n)}]{p}{n}$ corresponds to a $T(n)$-local Picard object $Z_n^f \in \Pic^\re(\Sp_{T(n)})[p-1]$, which lifts $Z_n$, implying \Cref{Thm_B}. 
We use a variation of the above method to produce non-trivial $T(n)$-local Picard objects in the case $p=2$ as well, using the three $\ZZ/2$-subextensions of $\cyc[\Sph_{T(n)}]{8}{n}$ (\Cref{Tele_Pic_Even}). 

\subsubsection{Faithfulness \& descent}
Taking the colimit over all $p^r$-th cyclotomic extensions, we  obtain the infinite cyclotomic extension 
\[
    R_n := \cyc[\Sph_{K(n)}]{p^\infty}{n} =
    \colim\, \cyc[\Sph_{K(n)}]{p^r}{n}.
\]
This continuous $\ZZ_p^\times$-Galois extension of $\Sph_{K(n)}$, which is classified by $\det\colon \Morex_n \to \ZZ_p^\times$, enables several key constructions in $\Sp_{K(n)}$. Among them, are the class $\zeta_n \in \pi_{-1}(\Sph_{K(n)})$ (see \cite[\S8]{DH})
and the determinant sphere 
$\Sphdet{K(n)} \in \Pic(\Sp_{K(n)})$
(see \cite{barthel2022constructing, Westerland, goerss2005resolution}).
Using our results, we can similarly construct the $T(n)$-local infinite cyclotomic extension $R_n^f := \cyc[\Sph_{T(n)}]{p^\infty}{n}$. Assuming $R_n^f$ is \textit{faithful}, one could lift $\zeta_n$ and $\Sphdet{K(n)}$ to the $T(n)$-local world. However, while all \textit{finite} Galois extensions of $\Sph_{T(n)}$ are faithful, we do not know whether the infinite Galois extension $R_n^f$ is faithful\footnote{In an earlier stage of this project, we believed that we have a proof for the faithfulness of $R_n^f$, which led to \cite[Remark 8.5.3]{BBBCX2019}. However, while writing this paper we have discovered a crucial gap in the argument. For a more thorough discussion see \cite{Fourier}.}. 
As far as we know, $R_n^f$ might be even $K(n)$-local, in which case the faithfulness of $R_n^f$ would be equivalent to the telescope conjecture. 
As an example, one can argue directly to show that $R_1^f$ is both faithful and isomorphic to $R_1$, which leads to a new proof of the telescope conjecture at height $n=1$. A more detailed account of this circle of ideas will appear elsewhere.

\subsection{Conventions}

Throughout the paper, we work in the framework of $\infty$-categories
(a.k.a. quasi-categories), and in general follow the notation of \cite{htt}
and \cite{HA}. The terminology and notation for all concepts related to higher semiadditivity and (semiadditive) height are as in \cite{AmbiHeight}. In addition,
\begin{enumerate}
    \item We use the notation $\hom(X,Y)$ for the enriched/internal hom-objects, as opposed to $\Map(X,Y)$ which always denotes the mapping space. 
    
    \item For an object $X$ in a monoidal $\infty$-category $\cC$, we write $\Omega^\infty X$ for $\Map(\one,X)$ and 
    $\pi_0X$ for $\pi_0\Map(\one,X)$.
    
    \item We denote by $\Prl$ the $\infty$-category of presentable $\infty$-categories and colimit preserving functors, and by 
    \[
        \Prl_{\tsadi_n} \sseq \Prl_{\st}^{\sad n} \sseq \Prl_{\st} \sseq \Prl_\add \sseq \Prl
    \]
    the full subcategories spanned by $\infty$-categories which are additive, stable, stable $n$-semiadditive and stable $n$-semiadditive of semiadditive height $n$ (with respect to an implicit prime $p$).\footnote{In the language of \cite[\S 5.2]{AmbiHeight}, these properties are classified by the \textit{modes} $\Sp^\cn$, $\Sp$, $\tsadi^{[n]}$, and $\tsadi_n$ respectively.} 
    
    \item For an abelian group $A$ and a natural number $m$ we denote by  $A[m]$ the subgroup of $m$-torsion elements in $A$. 
\end{enumerate}

\subsection{Acknowledgments}

We would like to thank Tobias Barthel, Clark Barwick, Agn{\`e}s Beaudry, Mark Behrens, Robert Burklund, Ehud de Shalit, Jeremy Hahn, Mike Hopkins, John Rognes, Nathaniel Stapleton, Dylan Wilson, and Allen Yuan for useful discussions. We also like to thank the entire Seminarak group, especially Shaul Barkan and Shay Ben Moshe, for useful comments on the paper's first draft. Finally, we would like to thank the anonymous referee for numerous corrections and suggestions, which greatly improved the paper.

The first author is supported by the Adams Fellowship Program of the Israel Academy of Sciences and Humanities. 
The second author is supported by ISF1588/18 and BSF 2018389.  

\section{Galois Theory}
We begin by discussing some special features of Galois extensions following \cite{RognesGal}, under the assumption that the classifying space of the acting group is ambidextrous with respect to the $\infty$-category. We shall work mainly in the setting of \emph{additive} presentable $\infty$-categories.  These include the stable presentable $\infty$-categories as well as ordinary additive presentable categories, such as that of abelian groups. 

The main result of \S 2.1, is that the Galois property can be detected by nil-conservative functors (\Cref{Galois_Nil}), and of \S 2.2, that it can be characterized by the fact that a certain associated lax symmetric monoidal functor is strong monoidal (\Cref{func_crit_galois}).

\subsection{Definition and Detection}

We begin by recalling some terminology and notation regarding local systems. Given an $\infty$-category $\cC$ and a space $A$ we denote by $\cC^A$ the $\infty$-category of functors $A \to \cC$ and refer to its objects as $\cC$-valued local systems on $A$. If $\cC$ is (symmetric) monoidal then $\cC^A$ is (symmetric) monoidal with respect to the pointwise structure. For a map $f \colon A \to B$, restriction along it provides a functor $f^*\colon \cC^B \to \cC^A$, which is (symmetric) monoidal when $\cC$ is. When $\cC$ admits $A$-shaped limits/colimits, $f^*$ admits a left/right adjoint which we denote by $f_!$ and $f_*$ respectively.

When $f\colon A\to B$ is weakly $\cC$-ambidextrous in the sense of \cite[Definition 4.1.11]{AmbiKn}, there is a canonical norm map $\Nm_f \colon f_! \to f_*$ and $f$ is called $\cC$-ambidextrous if $\Nm_f$ is an isomorphism. In this case, we obtain for every pair of objects $X,Y\in \cC^B$ an integration operation (see \cite[Definition 2.1.11]{TeleAmbi})
\[
    \int_f \colon \Map(f^*X,f^*Y) \to \Map(X,Y).
\]
In the special case of $f\colon A \to \pt$ we denote $\int_f \Id_X\colon X \to X$ by $|A|$ and think of it as ``multiplication by the cardinality of $A$''.
Recall also that $\cC$ is said to be $m$-semiadditive if every $m$-finite map is $\cC$-ambidextrous.

We next recall the definition of a Galois extension from \cite{RognesGal}:

\begin{defn}[Rognes] \label{Galois} 
    Let $\cC\in\calg(\Pr_{\add})$,
    let $G$ be a finite group and let $R\in\calg(\cC^{BG})$.
    We say that $R$ is a \textbf{$G$-Galois} extension (or just Galois)
    if it satisfies the following two conditions:
    \begin{enumerate}
        \item The canonical map $\one\to R^{hG}$ is an isomorphism.
        \item The canonical map $R\otimes R\to\prod_{G}R$, given informally by
        $x\otimes y\mapsto(x\cdot\sigma y)_{\sigma\in G}$, is an isomorphism. 
    \end{enumerate}
    A Galois extension $R$ is called \textbf{faithful} if in addition the functor 
    \[
        (-)\otimes R \colon \cC \longrightarrow \cC
    \]  
    is conservative. We denote by $\GalExt{\cC}{G} \sseq \calg(\cC^{BG})$ the full subcategory spanned by $G$-Galois extensions. 
\end{defn}

\begin{rem}
    For $S\in \calg(\cC)$, by a $G$-Galois extension of $S$, we shall mean a $G$-Galois extension in the symmetric monoidal $\infty$-category $\Mod_S(\cC)$ in the sense of \Cref{Galois}. 
\end{rem}

\begin{rem} \label{rem:Galois_semiadd_faithful}
    It is proved in \cite[Proposition 6.3.3]{RognesGal} that faithfulness of a Galois extension $R$ is equivalent to the condition that the norm map 
    $$\Nm\colon R_{hG} \to R^{hG}$$ 
    is an isomorphism. We shall be particularly interested in situations where $BG$ is $\cC$-ambidextrous (e.g. when $\cC$ is $1$-semiadditive or $|G|$ is invertible in $\cC$), in which case this condition is satisfied automatically. 
\end{rem}

Unlike in the classical Galois theory for fields, Galois extensions are not required to be \emph{connected}. 
In particular, for every group $G$, there is always the ``trivial'' $G$-extension:
\begin{example}[Split Galois extension]\label{Galois_Split}
    Let $\cC\in\calg(\Pr_{\add})$ and let $G$ be a finite group with $e\colon \pt \to BG$ the inclusion of the basepoint. The functor $e_*\colon \cC \to \cC^{BG}$ is lax symmetric monoidal and the induced object 
    \[
        e_*\one \simeq \prod_{G}\one\quad\in\quad\calg(\cC^{BG}) 
    \]
    is a $G$-Galois extension, where $G$ acts by permuting the factors according to the regular action of $G$ on itself. We say that a $G$-Galois extension $R$ is \textbf{split} if it is isomorphic to $e_*\one$ as a $G$-equivariant commutative algebra. 
\end{example}

The underlying object of a Galois extension is always dualizable  (see \cite[Proposition 6.2.1]{RognesGal} and \cite[Proposition 6.14]{AkhilGalois}). To detect Galois extensions, it will be useful to establish certain closure properties for dualizable objects.

\begin{prop}
\label{Dualizable_Colimit} 
Let $\cC \in \calg(\Pr)$ and let $I\in \cat_\infty$. If the tensor product of $\cC$ preserves $I^{\op}$-shaped limits in each variable, then the dualizable objects in $\cC$ are closed under $I$-shaped colimits.
\end{prop}

\begin{proof}
    We denote by $\hom(X,Y)\in \cC$ the internal hom-object of $X,Y\in\cC$. An object $X\in\cC$ is dualizable if and only if for every $Y\in\cC$, the canonical map 
    \[
        \hom(X,\one) \otimes Y \to 
        \hom(X,\one \otimes Y) \simeq
        \hom(X,Y)
    \]
    is an isomorphism (e.g. see \cite[Theorem 2.2]{ponto2014linearity}). Let $X = \colim_{a\in I} X_a$, such that $X_a \in \cC$ is dualizable for all $a\in I$. For every $Y\in \cC$ the canonical map above fits into a commutative diagram
   \begin{align*}
       \xymatrix{
       \hom(\colim_{a\in I}X_{a},\one)\otimes Y\ar[d]^\wr\ar[rr] &  &
       \hom(\colim_{a\in I}X_{a},Y)\ar[d]^\wr\\
       (\invlim_{a\in I^{\op}}\hom(X_{a},\one))\otimes Y\ar[r]^\sim & 
       \invlim_{a\in I^{\op}}(\hom(X_{a},\one)\otimes Y)\ar[r]^\sim &
       \invlim_{a\in I^{\op}}\hom(X_{a},Y).
    }
    \end{align*}
    The vertical arrows are isomorphisms since $\hom(-,-)$ takes colimits in the first variable into limits. The bottom left arrow is an isomorphism because the tensor product preserves $I^\op$-limits in each variable and the bottom right arrow is an isomorphism because each $X_a$ is dualizable. It follows that the top map is an isomorphism and hence that $X$ is dualizable.
\end{proof}

\begin{rem}\label{Stable_Dualizable}
    For $\cC$ stable, the tensor product preserves finite limits, and we recover the classical fact that dualizable objects in $\cC$ are closed under finite colimits (see, e.g., \cite{may2001additivity}).
\end{rem}

\begin{cor}
\label{Dualizable_Sad}
    Let $\cC \in \calg(\Pr)$ and let $A$ be a $\cC$-ambidextrous space. The dualizable objects in $\cC$ are closed under $A$-shaped limits and colimits.
\end{cor}

\begin{proof}
    Since $A$ is $\cC$-ambidextrous, $A$-shaped limits coincide with $A$-shaped colimits. It, therefore, suffices to show that the dualizable objects are closed under $A$-shaped colimits. Moreover, since the tensor product preserves $A$-shaped colimits in each variable, it also preserves $A$-shaped limits in each variable (see \cite[Proposition 2.1.8]{AmbiHeight}). Therefore, the claim follows from \Cref{Dualizable_Colimit}.
\end{proof}

\begin{rem}\label{Dualizable_Constant}
    The special case of a constant $A$-shaped colimit was already treated in \cite[Proposition 3.3.6]{TeleAmbi}, where we have further shown that $\dim(A\otimes\one) = |LA| \in \pi_0\one$ (\cite[Corollary 3.3.10]{TeleAmbi}).
\end{rem}

The following proposition shows that in the stable setting, the Galois property can be detected by \emph{nil-conservative} functors (\cite[Definition 4.4.1]{TeleAmbi}).

\begin{prop} \label{Galois_Nil}
    Let $F\colon\cC\to\cD$ be a nil-conservative functor in $\calg(\Pr_{\st})$, let $G$ be a finite group such that $BG$ is $\cC$-ambidextrous, and let $R\in\calg(\cC^{BG})$. If $F(R)\in\calg(\cD^{BG})$ is Galois and $R$ is dualizable in $\cC$, then $R$ is Galois.
 \end{prop}

\begin{proof}
    First, by \cite[Corollary 3.3.2]{TeleAmbi}, the space $BG$ is also $\cD$-ambidextrous. 
    Now, since $BG$ is $\cC$- and $\cD$-ambidextrous and $F$ preserves colimits, $F$ also preserves $BG$-shaped limits by \cite[Proposition 2.1.8]{AmbiHeight}. Thus, applying $F$ to the maps
    \[
        \one \to R^{hG}\quad,\quad R\otimes R\to \prod_G R
    \]    
    in conditions (1) and (2) of \Cref{Galois}, we get the corresponding maps 
    \[
        \one \to F(R)^{hG}\quad,\quad F(R)\otimes F(R)\to \prod_G F(R)
    \]
    for $F(R)\in\calg(\cD^{BG})$. Since $F(R)$ is Galois, these maps are isomorphisms.
    Since the underlying object of $R$ is dualizable in $\cC$, all the objects $\one$, $R^{hG}$, $R\otimes R$ and $\prod_G R$ are dualizable as well. Indeed, $R^{hG}$ is dualizable by \Cref{Dualizable_Sad}, and the other three by standard arguments. Thus, as nil-conservative functors are conservative on dualizable objects (see \cite[Proposition 4.4.4]{TeleAmbi}), we get that $R$ is Galois.
\end{proof}

\subsection{Twisting Functors}
It will be useful for the sequel, to observe that the Galois property can be characterized using the following notion:
\begin{defn}
    \label{Twisting_Functor}
    For every $\cC \in \calg(\Pr_\add)$ and 
    $R\in \calg(\cC^{BG}),$ 
    we define the \textbf{twisting functor} of $R$ to be the composition 
    \[
        T_R \colon
        \cC^{BG} \oto{R\otimes (-)} 
        \cC^{BG} \oto{(-)^{hG}} 
        \cC.
    \] 
\end{defn}

The functor $T_R$ is lax symmetric monoidal as a composition of the lax symmetric monoidal functor $(-)^{hG}$ and the functor $R\otimes(-)$, which is itself lax symmetric monoidal as a composition of the functors in the free-forgetful symmetric monoidal adjunction 
\[
    \cC^{BG} \oto{F} 
    \Mod_R(\cC^{BG}) \oto{U} 
    \cC^{BG}.
\]
We note the following immediate consequence of assuming that $BG$ is $\cC$-ambidextrous: 
\begin{lem}
\label{Twisting_Functor_Linear}
    If $BG$ is $\cC$-ambidextrous, then $T_R$ preserves colimits and is $\cC$-linear in the sense that for all $X\in\cC^{BG}$ and $Z\in \cC$, the canonical map 
    \[
        T_R(X)\otimes Z =
        (R\otimes X)^{hG} \otimes Z \oto{\beta_*}
        (R\otimes X \otimes Z)^{hG} =
        T_R(X\otimes Z)
    \]
    is an isomorphism.
\end{lem}
\begin{proof}
    The norm map
    \[
        \Nm \colon (R\otimes X)_{hG} \to 
        (R\otimes X)^{hG} = 
        T_R(X)
    \]
    is an isomorphism, and hence the functor $T_R$ is isomorphic to the colimit preserving functor $X\mapsto (R\otimes X)_{hG}$. Furthermore, for every $Z \in \cC$ consider the colimit preserving functor $(-) \otimes Z \colon \cC \to \cC$ and the associated commutative norm diagram (\cite[Theorem 3.2.3]{TeleAmbi}), which for every $X\in \cC^{BG}$ is of the form:
    \[\begin{tikzcd}
    	{(R\otimes X \otimes Z)_{hG}} && {(R\otimes X \otimes Z)^{hG}} \\
    	{(R\otimes X)_{hG} \otimes Z} && {(R\otimes X)^{hG} \otimes Z.}
    	\arrow["\Nm", from=2-1, to=2-3]
    	\arrow["{\beta_!}", from=2-1, to=1-1]
    	\arrow["\Nm", from=1-1, to=1-3]
    	\arrow["{\beta_*}", from=1-3, to=2-3]
    	\arrow["\sim"', from=1-1, to=1-3]
    	\arrow["\sim"', from=2-1, to=2-3]
    	\arrow["\wr"', from=2-1, to=1-1]
    \end{tikzcd}\]
    It follows that $\beta_*$ is an isomorphism as well.
\end{proof}

\begin{rem}\label{Restricted_Twisting_Functor}
    Using the free-forgetful adjunction $\Prl \adj \Mod_\cC(\Prl)$, the $\cC$-linearity of $T_R$ can be rephrased as follows. Let $\overline{T}_R$ be the restriction of $T_R$ along $\Spc^{BG}\to\cC^{BG}.$ The functor $\overline{T}_R \colon \Spc^{BG}\to \cC$ 
    is colimit preserving and
    \[
        T_R\colon
        \cC^{BG} \simeq
        \cC\otimes \Spc^{BG} \to
        \cC
    \] 
    is the corresponding $\cC$-linear functor.
\end{rem}  

The lax symmetric monoidal structure of $T_R$ can be used to characterize the Galois property for $R$.

\begin{prop}
\label{func_crit_galois}
    Let $\cC\in \calg(\Pr_\add)$ and let $G$ be a finite group such that $BG$ is $\cC$-ambidextrous. 
    A $G$-equivariant commutative ring  $R\in\calg({\cC}^{BG})$ is Galois, if and only if $T_R$ is (strong) symmetric monoidal.
\end{prop}

\begin{proof}
    The Galois property of $R$ can be related to the properties of the functor $T_R$ as follows. First, for $BG\oto{q}\pt$, the unitality of $T_R$ amounts to the unit map 
    \[
        \one \to T_R(q^*\one) \simeq R^{hG}
    \] 
    being an isomorphism, i.e., it is equivalent to the first Galois condition for $R$. 
    Second, let $\pt \oto{e} BG$ denote the basepoint. We have, on the one hand,
    \[
        T_R(e_*\one) \otimes T_R(e_*\one) \simeq 
        (\prod_G R)^{hG} \otimes (\prod_G R)^{hG} \simeq 
        R \otimes R
    \]
    and on the other,
    \[
        T_R(e_*\one \otimes e_*\one) \simeq 
        (\prod_{G\times G} R)^{hG} \simeq 
        \prod_G R.
    \]
    The second isomorphism follows from the fact that the diagonal action of $G$ on $G\times G$ is free with quotient $G$. 
    Moreover, the canonical map, induced by $T_R$ being lax symmetric monoidal,
    \[
        R\otimes R \simeq 
        T_R(e_*\one) \otimes T_R(e_*\one) \to T_R(e_*\one\otimes e_*\one) \simeq
        \prod_G R
    \] 
    is exactly the map appearing the in second Galois condition for $R$. Hence, the second condition is equivalent to the structure map 
    $T_R(X)\otimes T_R(Y) \to T_R(X\otimes Y)$ 
    to be an isomorphism in the special case $X=Y=e_*\one\simeq e_!\one$. In particular, if $T_R$ is strong symmetric monoidal, then $R$ is Galois.
    
    Conversely, assume that $R$ is Galois. By the $\cC$-linearity of $T_R$ (\Cref{Twisting_Functor_Linear}), to show that $T_R$ is strong symmetric monoidal, it suffices to show that the restriction $\overline{T}_R\colon \Spc^{BG} \to \cC$ is strong symmetric monoidal (\Cref{Restricted_Twisting_Functor}). By the above, the structure map for the symmetric monoidality of $\overline{T}_R$ is an isomorphism in the case $X=Y=e_!(\pt)$. The local system $e_!(\pt)$ generates $\Spc^{BG}$ under colimits, and $\overline{T}_R$ is colimit preserving. It follows that $\overline{T}_R$ is strong symmetric monoidal and hence so is $T_R$.
\end{proof}

The strong symmetric monoidality of the twisting functor implies that it induces a ``descent'' map $\Pic(\cC^{BG}) \to \Pic(\cC)$. This allows one to construct Picard objects in $\cC$ by twisting Picard objects in $\cC^{BG}$ (see e.g. \cite{barthel2022constructing}).
Though we shall adopt a somewhat different perspective, our construction of Picard objects from Galois extensions in the next section fits into this paradigm.   
    
\section{Kummer Theory}
\label{kummer}

In this section, we study the relationship between abelian Galois extensions and the Picard spectrum. As in \S2, we shall work mainly in the setting of additive $\infty$-categories. 

In \S 3.1, we review the notion of (primitive) roots of unity (\Cref{Def_Roots_Of_Unity}) and prove a general form of the ``orthogonality of characters'' (\Cref{Characters}). In \S 3.2, we give a general form of the discrete Fourier transform (\Cref{DFT_Equiv}) and use the results of \S2.2 to characterize the Galois property of a commutative algebra in terms of its Fourier transform (\Cref{DFT_Galois}). In \S3.3, we use this characterization to establish the general form of Kummer theory (\Cref{Kummer_Theory}), and analyze the special case of a cyclic group (\Cref{Kummer_Cyclic}). We conclude with a certain variant for constructing Picard objects out of $\ZZ/2$-Galois extensions (\Cref{R_minus_one_Pic}), which will play a role in the chromatic world when $p=2$.

\subsection{Character Theory}
\label{Section_Roots_of_Unity}

\subsubsection{Roots of unity}

Following \cite[\S 1.3]{ando2018parametrized}, for every $\cC\in\calg(\Prl)$, there is a unique symmetric monoidal colimit preserving functor $\Spc \to \cC$, which induces an adjunction 
\[
    \calg(\Spc)\adj\calg(\cC).
\]
Furthermore, $\calg(\Spc)\simeq\CMon(\Spc)$ contains a coreflective full subcategory of group-like commutative monoids $\CMon^{\rm{gp}}(\Spc) \sseq \CMon(\Spc),$ which is equivalent to the full subcategory of connective spectra $\Sp^\cn \sseq \Sp$. Composing these adjunctions, we get an adjunction of the form
\[
    \one[-]\colon \Sp^\cn \adj 
    \calg(\cC):(-)^\times.
\]
We think of the left adjoint $\one[-]$ as the \textbf{group algebra} functor, and of the right adjoint $(-)^\times$ as the commutative \textbf{group of units}. 

Specializing to the case $\cC = \cat_\infty$ with its Cartesian symmetric monoidal structure,  $\calg(\cat_\infty)$ is the $\infty$-category of symmetric monoidal $\infty$-categories. In this case, the functor 
\[
    \one[-]\colon \Sp^\cn \to \calg(\cat_\infty)
\] 
is fully faithful with essential image those symmetric monoidal $\infty$-categories in which all morphisms are invertible and all objects are $\otimes$-invertible. 
We shall thus abuse notation and regard a connective spectrum also as a symmetric monoidal $\infty$-category via this fully faithful embedding. 

\begin{defn} \label{def:pic}
    For $\cC\in\calg(\cat_\infty)$, the \textbf{Picard spectrum} of $\cC$ is given by
    \[
        \pic(\cC) := \cC^\times \qin \Sp^\cn
    \]
    and 
    \[
        \Pic(\cC) := \pi_0(\pic(\cC)) \qin \Ab
    \] 
    is the \textbf{Picard group}.
\end{defn}

Less formally, the Picard spectrum of $\cC$ consists of tensor invertible objects of $\cC$ with the tensor product as a coherently commutative group operation. 
This is a (usually non-trivial) delooping of the connective spectrum $\one^\times_\cC$ in the following sense: 
\[
    (\Omega\pic(\cC))_{\ge0} \simeq 
    \one_\cC^\times
    \qin \Sp^\cn.
\]
The counit of the adjunction $\Sp^\cn \adj \calg(\cat_\infty)$ provides a symmetric monoidal functor $\pic(\cC) \to \cC$, which is the non-full embedding of the $\otimes$-invertible objects and the isomorphisms between them into $\cC$. 

\begin{rem}
    For a large symmetric monoidal $\infty$-category $\cC$, the spectrum $\pic(\cC)$ might a-priori be large as well. However, if $\cC$ is presentable, the spectrum $\pic(\cC)$ is (essentially) small (see e.g. \cite[Remark 2.1.4]{mathew2016picard}).
\end{rem}

Having introduced the space of units of a commutative algebra, we can now further consider \emph{roots of unity}.

\begin{defn}[Roots of Unity]\label{Def_Roots_Of_Unity}
    Let $\cC\in\calg(\Prl_\add)$ and let $R\in \calg(\cC)$. For every $m\in\NN$,
    \begin{enumerate}
        \item We define the space of \textbf{$m$-th roots of unity} in $R$ by 
        \[
            \mu_m(R) := 
            \Map_{\Sp^\cn}(C_m,R^\times),
        \]
        where $C_m$ is the cyclic group of order $m$.
        \item  We say that an $m$-th root of unity
            $C_m\oto{\omega} R^\times$ is \textbf{primitive}, if $R$ is $m$-divisible (i.e. $m$ is invertible in $\pi_0R$), and for every $d$ which strictly divides $m$, the only commutative $R$-algebra $S$ 
            for which there exists a dotted arrow rendering the diagram of connective spectra
            \[
                \xymatrix{
                    C_m\ar@{->>}[d]\ar[rr] &  & 
                    R^{\times}\ar[d] \\
                    C_d\ar@{-->}[rr] &  & 
                    S^{\times}
                }
            \]
            commutative, is $S=0$. 
            We denote by $\mu_m^\prim(R) \sseq \mu_m(R)$ the union of connected components of primitive $m$-th roots of unity.
    \end{enumerate}
    By convention, a (primitive) $m$-th root of unity of $\cC$ is a (primitive) $m$-th root of unity of $\one_\cC$.
\end{defn}

Employing the adjunction 
$\one[-] \dashv (-)^\times,$
the functor 
$\mu_m \colon \calg(\cC) \to \Spc$ 
is corepresented by the group algebra $\one[C_m].$ If we further assume that $\one$ is $m$-divisible, then for every divisor $d\mid m$, the map 
$\one[C_m] \to \one[C_d]$ 
can be identified with
$\one[C_m] \to \one[C_m][\varepsilon_d^{-1}],$ 
for the idempotent
\[
    \varepsilon_d = 
    \frac{d}{m}\sum_{a\in d\cdot C_m}a
    \qin \pi_0(\one[C_m]).
\]
\begin{defn}[Cyclotomic Extensions]
\label{Def_Cyclo_Ext}
    Let $\cC\in\calg(\Prl_\add)$, such that $m$ is invertible in $\cC$. We define the \textbf{$m$-th cyclotomic extension} to be 
    \[
        \one[\omega_m] :=
        \one[C_m][\varepsilon^{-1}],
    \]
    where
    \[
        \varepsilon = 
        \prod_{1\le d< m,\,d\mid m}(1-\varepsilon_d)
        \qin \pi_0(\one[C_m]).
    \]
    The commutative algebra $\one[\omega_m]$ carries a tautological (primitive) $m$-th root of unity denoted $\omega_m$.
\end{defn}

By the above discussion, the cyclotomic extension $\one[\omega_m]$ corepresents the functor of primitive $m$-th roots of unity 
$\mu^\prim_m \colon \calg(\cC) \to \Spc.$
Namely, for all $R\in\calg(\cC)$ we have a natural isomorphism
\[
    \mu_m^\prim(R) \simeq 
    \Map_{\calg(\cC)}(\one[\omega_m],R).
\]

\begin{example}\label{Roots_Universal}
    For the $\infty$-category  
    $\cC = \Sp^\cn$
    of connective spectra, the $m$-th cyclotomoic extension
    \[
        \Sph[\inv{m}] \to \Sph[\inv{m},\omega_m]
    \]
    is the unique \'{e}tale extension  which on $\pi_0$ induces the ordinary $m$-cyclotomic extension (see \cite[Theorem 7.5.0.6]{HA})
    \[
        \ZZ[\inv{m}] \to 
        \ZZ[\inv{m},\omega_m] :=
        \ZZ[\inv{m},t]/\Phi_m(t).
    \]
    Here, $\Phi_m(t)$ is the $m$-th cyclotomic polynomial. 
\end{example}

For $\cC \in \calg(\Prl_\add)$, we have a unique symmetric monoidal colimit preserving functor $\Sp^\cn \to \cC$, whose right adjoint (the ``underlying connective spectrum'') we denote by $X \mapsto \underline{X}$.
\begin{lem}\label{Roots_Underlying_Specturm}
    Let $\cC\in\calg(\Prl_\add)$ and let $R\in\calg(\cC)$. For every $m$, there is a canonical isomorphism 
    $\mu_m(R) \simeq \mu_m(\underline{R}),$
    which restricts to an isomorphism 
    $\mu^\prim_m(R) \simeq
    \mu^\prim_m(\underline{R}),$
    if $R$ is $m$-divisible. 
\end{lem}
\begin{proof}
    The first claim follows from the adjunction $\calg(\Sp^\cn) \adj \calg(\cC)$ as follows:
    \[
        \mu_m(R) =
        \Map_{\calg(\cC)}(\one[C_m],R) \simeq
        \Map_{\calg(\Sp^\cn)}(\Sph[C_m],\underline{R}) =
        \mu_m(\underline{R}).
    \]
    Assuming $R$ is $m$-divisible, we can without loss of generality assume that $\one$ is also $m$-divisible, by replacing $\cC$ with $\Mod(R)$. 
    Thus, the second claim follows similarly:
    \[
        \mu^\prim_m(R) =
        \Map_{\calg(\cC)}(\one[\omega_m],R) \simeq
        \Map_{\calg(\Sp^\cn)}(\Sph[\inv{m},\omega_m],\underline{R}) =
        \mu^\prim_m(\underline{R}).
    \]
\end{proof}

We deduce that \Cref{Roots_Universal} is universal in the  sense that a primitive $m$-th root of unity in $\cC$ is the same as a symmetric monoidal colimit preserving functor from 
$\Mod_{\Sph[\inv{m},\omega_m]}(\Sp^\cn)$
to $\cC$.
\begin{prop}
\label{Univ_Cycl_SM}
    Let $\cC\in\calg(\Prl_\add)$. For every $m$, we have 
    \[
        \mu_m^\prim(\cC) \simeq
        \Map_{\calg(\Prl)}(\Mod_{\Sph[\inv{m},\omega_m]}(\Sp^\cn),\cC)
        \qin \Spc.
    \]
\end{prop}

\begin{proof}
By \cite[Corollary 4.8]{GepUniv}  we have an equivalence $\Prl_\add \simeq \Mod_{\Sp^\cn}(\Prl)$.
    Thus, by \cite[Theorems 4.8.5.11, 4.8.5.16 and Corollary 4.8.5.21]{HA}, we have an adjunction
    \[
        \Mod_{(-)}(\Sp^\cn)\colon 
        \calg(\Sp^\cn) \adj \calg(\Prl_\add)
        \colon \underline{\one}_{(-)}.
    \]
    Applying this to $\Sph[\inv{m},\omega_m] \in \calg(\Sp^\cn)$ and $\cC \in \calg(\Prl_\add)$, we get by \Cref{Roots_Underlying_Specturm},
    \[
        \mu_m^\prim(\one_\cC) \simeq 
        \mu_m^\prim(\underline{\one}_\cC) \simeq
        \Map_{\calg(\Sp^\cn)}(\Sph[\inv{m},\omega_m],\underline{\one}_\cC) \simeq
        \Map_{\calg(\Prl_\add)}(\Mod_{\Sph[\inv{m},\omega_m]}(\Sp^\cn),\cC).
    \]
\end{proof}

We also deduce that for an $m$-divisible commutative algebra $R$, the space of (primitive) $m$-th roots of unity is discrete and depends only on $\pi_0(R)$.
\begin{prop}\label{Roots_Discrete}
    Let $\cC\in\calg(\Prl_\add)$ and let $R\in\calg(\cC)$ which is $m$-divisible.
    We have a canonical bijection $\mu_m(R) \simeq \mu_m(\pi_0R)$, which restricts to a bijection $\mu^\prim_m(R) \simeq \mu^\prim_m(\pi_0R).$
\end{prop}
    
\begin{proof}
    By \Cref{Roots_Underlying_Specturm}, it suffices to consider the universal case $\cC = \Sp^\cn$. In this case, we have that
    \[
        \Omega^\infty(R^\times) \sseq 
        \Omega^\infty(R)
        \qin \Spc
    \]
    is an inclusion of connected components. Thus, $\pi_nR \simeq \pi_n R^\times$ for all $n\ge1$. Namely, the fiber $R^\times_{\ge1}$ (in $\Sp$) of the truncation map $R^\times \to \pi_0R^\times$ has the same homotopy groups as the spectrum $R_{\ge1}$. 
    We therefore deduce that $R^\times_{\ge1}$ is $m$-divisible and hence  
    $\Map_{\Sp^\cn}(C_m,R_{\ge1})=0.$
    It follows that 
    \[
        \mu_m(R) =
        \Map_{\Sp^\cn}(C_m,R^\times) \iso \Map_{\Sp^\cn}(C_m,\pi_0R^\times) =
        \mu_m(\pi_0R).
    \]
    Since the invertibility of an idempotent is a condition on $\pi_0$, under this bijection, primitive roots correspond to primitive roots. 
\end{proof}

Of specific importance for us, is the following special case:
\begin{cor}\label{p_Complete_Roots}
    Let $\cC\in\calg(\Prl_\add)$ and let $R\in\calg(\cC)$ be $p$-complete\footnote{That is, $\hom_\cC(X,R)\in \Sp^\cn$ is $p$-complete for all $X\in\cC$.} for some prime $p$. For every $m\mid (p-1)$, the commutative algebra $R$ admits a primitive $m$-th root of unity.
\end{cor}

\begin{proof}
    First of all, since $R$ is $p$-complete, it is $m$-divisible. Now, by \Cref{Roots_Discrete}, it suffices to show that $\pi_0R$ admits a primitive $m$-th root of unity. This follows from the fact that $\pi_0R$ is a $\ZZ_p$-algebra, and $\ZZ_p$ admits primitive $m$-th roots of unity given by Teichm{\"u}ller lifts.
\end{proof}

\subsubsection{Characters}
As in ordinary commutative algebra, primitive roots of unity in $\cC$ allow us to set up a \emph{character theory} for $\cC$.
Let $A$ be a finite $m$-torsion abelian group with  Pontryagin dual denoted by
\[
    A^* := \hom(A,C_m) = \hom(A,\QQ/\ZZ).
\]
Given $\cC \in \calg(\Prl_\add)$ with a choice of a primitive $m$-th root of unity 
\(
    \omega\colon C_m \to \one^\times
\)
(so that in particular $m$ is invertible in $\cC$),
the canonical pairing of $A$ with $A^*$ induces  a map of spectra
\[
    A^* \otimes A \to 
    C_m \oto{\,\,\omega \,\,}
    \one^\times \to
    \Omega\pic(\cC).
\]
This map corresponds to a map of connective spectra
\[
    A^* \to 
    \hom(A,\Omega \pic(\cC)) \simeq 
    \hom(\Sigma A,\pic(\cC)).
\]
\begin{defn}
\label{def_aigen_spaces}
    Let $\cC\in \calg(\Pr_\add)$ with a primitive $m$-th root of unity $\omega$, and let $A$ be a finite $m$-torsion abelian group. We define a map of connective spectra
    \[
        \one(-) \colon 
        A^* \to 
        \pic(\cC^{BA})
    \]
    to be the composition
    \[
        A^* \to 
        \hom(\Sigma A,\pic(\cC)) \to
        \pic(\cC)^{BA} \simeq
        \pic(\cC^{BA}),
    \]
    where the first map is the one given above and the second is induced from the counit $\Sph[BA]\to \Sigma A$ by pre-composition. Even though the construction of $\one(-)$ depends on $\omega$, we shall keep this dependence implicit.   
\end{defn}
Intuitively, for every character $A \oto{\varphi} C_m$, the object $\one(\varphi)\in \cC^{BA}$ is the unit $\one \in \cC$, on which the group $A$ acts through the composition of the character $\varphi$ with $C_m\oto{\omega}\one^\times$. The fact that $\one(-)$ is a map of connective spectra encodes in a coherent way the $A$-equivariant identities
\[
    \one(0) \simeq \one \quad , \quad 
    \one(\varphi + \psi) \simeq 
    \one(\varphi) \otimes \one(\psi).   
\]
For $X\in\cC^{BA}$, we define its twist by a character $\varphi\in A^*$, to be
\[
    X(\varphi) := 
    X\otimes \one(\varphi)
    \qin \cC^{BA}.
\]
We shall implicitly treat an object $X\in\cC$ as an object of $\cC^{BA}$ with a trivial action. 
The main fact we shall need about this construction is the following analogue of the ``orthogonality of characters'' from classical algebra:
\begin{prop}
\label{Characters}
Let $\cC\in \calg(\Pr_\add)$ with a primitive $m$-th root of unity and let $A$ be a finite $m$-torsion abelian group. 
\begin{enumerate}
    \item For every $X\in \cC$ and $\varphi \in A^*$, we have  
    \[
        X(\varphi)^{hA}\simeq 
        \begin{cases} 
            X & \varphi = 0 \\ 
            0 & \text{ else}  
        \end{cases}
        \qin \cC.
    \]
    \item For every $X\in \cC$, we have
    \[
        \prod_{a \in A} X \simeq 
        \bigoplus_{\varphi\in A^*} X(\varphi)
        \qin \cC^{BA},
    \]
    where on the left side we have the induced representation (i.e. $A$ acts by permuting the factors).
\end{enumerate}
\end{prop}

\begin{proof}
    (1) Since $BA$ is $\cC$-ambidextrous, we have (by \cite[Proposition 3.3.1]{TeleAmbi})
    \[
        X(\varphi)^{hA} =
        (X\otimes \one(\varphi))^{hA} \simeq
        X\otimes (\one(\varphi)^{hA}).
    \]
    Thus, it suffices to show the claim for $X = \one$. By \Cref{Univ_Cycl_SM}, we have a colimit preserving symmetric monoidal functor 
    $F\colon\Mod_{\Sph[\inv{m},\omega_m]}(\Sp^\cn) \to \cC,$
    which in particular takes the unit $\Sph[\inv{m},\omega_m]$ to the unit $\one$.
    Since $A$ is $m$-torsion, by \cite[Proposition 2.1.8]{AmbiHeight}, $F$ also preserves $A$-homotopy fixed points, thus it suffices to prove the claim for $\Sph[\inv{m},\omega_m]$. Since $\pi_*$ preserves $A$-homotopy fixed points for $m$-divisible spectra, the result for $\Sph[\inv{m},\omega_m]$ follows from the analogous fact for $\pi_*(\Sph[\inv{m},\omega_m])=\pi_*(\Sph)[\inv{m},\omega_m].$ 
    Finally, in the case $\varphi=0$ the action of $A$ on $\pi_*(\Sph)[\inv{m},\omega_m]$ is trivial so the statement is clear. For $\varphi \ne 0$, there is $a\in A$ such that \[\varphi(a) \ne 0 \qin \ZZ/m\ZZ \subset \QQ/\ZZ,\] and since $\omega_m$ is a primitive $m$-th root of unity, $\omega_m^{\varphi(a)}-1$ is invertible in $\pi_0(\Sph)[\inv{m},\omega_m] \simeq \ZZ[\inv{m},\omega_m]$. Thus,   $\omega_m^{\varphi(a)}-1$ acts invertibly on  $\pi_*(\Sph)[\inv{m},\omega_m]$, and therefore $(\pi_*(\Sph)[\inv{m},\omega_m](\varphi))^{hA} \simeq 0$. 

    (2) It again suffices to consider the case $X=\one$. Under the free-forgetful adjunction
    $\cC \adj \cC^{BA}$, the non-equivariant map $\one(\varphi)\iso \one$ corresponds to the map 
    \[
        \iota_\varphi \colon 
        \one(\varphi) \to 
        \prod\limits_{a\in A} \one,
    \]
    which on the $a$-factor is given by multiplication with $\omega_m^{\varphi(a)}$. Hence, the induced map
    \[
        \iota \colon
        \bigoplus\limits_{\varphi\in A^*} \one(\varphi) \to
        \prod_{a \in A} \one
    \]
    is represented by the discrete Fourier transform $(A^* \times A)$-matrix $c_{\varphi,a} =\omega_m^{\varphi(a)}$. The square of the determinant of this Fourier transform matrix is $|A|^{|A|}$,  which is invertible in the ring $\ZZ[\inv{m},\omega_m]$ since $A$ is an $m$-torsion group. Hence, this matrix is invertible also in $\pi_0(\one)$. 
\end{proof}

\subsection{Fourier Transform}

\subsubsection{Construction}

Consider the composition 
\[
    A^* \oto{\one(-)}  \pic(\cC^{BA}) \to 
    \cC^{BA},
\]
of symmetric monoidal functors,
in which the second functor is the counit map described below \Cref{def:pic}. We shall denote it again by $\one(-)$. Since the dual of an invertible object coincides with its inverse we have
\[
    \one(\varphi)^\vee \simeq \one(-\varphi).
\]
Consider the following composition of functors
\[
    \widehat{\Fur}\colon
      A^*\times \cC^{BA} \oto{\one(-)^\vee \times \Id }
    \cC^{BA} \times \cC^{BA} \oto{\otimes}
    \cC^{BA} \oto{(-)^{hA}}
    \cC.
\]
On the level of objects, for every $X\in\cC^{BA}$ and $\varphi\in A^*$, we have 
\[
    \widehat{\Fur}(\phi,X) \simeq 
    X(-\varphi)^{hA}.
\]
This should be thought of as extracting from $X$ the eigenspace corresponding to the character $\varphi$. Taking the mate of the above functor under the exponential law, we get:

\begin{defn}[Fourier Transform]\label{DFT}
    Let $\cC\in\calg(\Pr_\add)$ with a choice of a primitive $m$-th root of unity and let $A$ be a finite $m$-torsion abelian group. We define the $\cC$-\textbf{Fourier transform} to be the functor
    \[
        \Fur \colon \cC^{BA} \to
        \Fun(A^*,\cC)
    \]
    given by
    \(
        \Fur(X)_\varphi := \widehat{\Fur}(\phi,X).
    \)
\end{defn}
The category of functors from $A^*$ to $\cC$ can be endowed with the \textit{Day convolution} symmetric monoidal structure, which we denote by $\Fun(A^*,\cC)_\Day$ (see \cite[\S 2.2.6]{HA}, \cite{glasman2016day}). By \cite[Example 2.2.6.9]{HA}, the construction of $\Fun(A^*,\cC)_\Day$ is a special case of the  norm construction for $\infty$-operads, in the sense of \cite[Definition 2.2.6.1]{HA}. Thus, by its universal property, we have an equivalence of $\infty$-categories
\[
    \Fun^{\tlax}(A^* \times \cC^{BA},\cC) \simeq
    \Fun^{\tlax}(\cC^{BA}, \Fun(A^*,\cC)_\Day).
\]
Since $\widehat{\Fur}$ is lax symmetric monoidal, as a composition of functors that are canonically such, the functor $\Fur$ acquires a lax symmetric monoidal structure as well. In fact,

\begin{prop}
\label{DFT_Equiv}
    Let $\cC\in\calg(\Pr_\add)$ with a choice of a primitive $m$-th root of unity and let $A$ be a finite $m$-torsion abelian group. The $\cC$-Fourier transform 
    \[
        \Fur \colon 
        \cC^{BA} \to
        \Fun(A^*,\cC)_\Day
    \]
    is a (strong) symmetric monoidal equivalence.
\end{prop}

\begin{proof}
    We first show that $\Fur$ is an equivalence of $\infty$-categories (ignoring the symmetric monoidal structure) by showing that it admits a fully faithful and essentially surjective left adjoint. The functor $\Fur$ admits a left adjoint
    \[
        \invFur \colon 
        \Fun(A^*,\cC) \to 
        \cC^{BA},
    \]
    given by tensoring pointwise with the functor $\one(-)\colon A^*\to \cC^{BA}$ followed by taking the direct sum over $A^*$. Thus, its value on objects is given by 
    \[
        \invFur(\{X_\varphi\}) =
        \bigoplus_{\varphi\in A^*} X_\varphi(\varphi).
    \]
    To show that $\invFur$ is fully faithful, it suffices to show that the unit of the adjunction $\invFur \dashv \Fur$ is an isomorphism. 
    Unwinding the definitions and using \Cref{Characters}(1), we get
    \[
        \Fur(\invFur (\{X_\varphi\}))_\psi \simeq 
        (\bigoplus_{\varphi \in A^*} X_\varphi(\varphi-\psi)) ^ {hA} \simeq X_\psi
    \]
    and that the unit map under this identification is the identity. 
    Now, we observe that for all $X\in\cC$, the induced representations (see \cref{Characters}(2))
    \[
        \prod_{A} X \simeq 
        \bigoplus_{\varphi\in A^*} X(\varphi)
    \]
    are in the essential image of $\invFur$. Since these generate $\cC^{BA}$ under colimits (by \cite[Proposition 4.3.8]{AmbiKn}), and $\invFur$ is fully faithful, we deduce that $\invFur$ is essentially surjective and hence is an equivalence. 
    
    We now turn to the preservation of the symmetric monoidal structure. To show that $\Fur$ is strong symmetric monoidal, it suffices to consider objects of the form $X(\varphi)$ for $X\in\cC$ and $\varphi\in A^*$, as they generate $\cC^{BA}$ under colimits and $\Fur$ is colimit preserving (being an equivalence). For such objects, we have by \Cref{Characters}(1)
    \[
        \Fur(X(\varphi))_\psi \simeq 
        X(\varphi-\psi)^{hA} \simeq 
        \begin{cases}
            X & \psi=\varphi \\
            0 & \text{else}
        \end{cases}
    \] 
    and the structure map
    \[
        \Fur(X(\varphi))\otimes \Fur(Y(\psi)) \to
        \Fur((X\otimes Y)(\varphi+\psi))
    \]
    is the obvious isomorphism. One can similarly show that $\Fur$ is unital and hence strong symmetric monoidal. 
\end{proof}

\subsubsection{Fourier of rings}
In the situation of \Cref{DFT_Equiv}, the symmetric monoidal equivalence 
\[
    \Fur \colon \cC^{BA} \iso \Fun(A^*,\cC)_\Day
 \]
induces an equivalence of the $\infty$-categories of commutative algebra objects. By \cite[Example 2.2.6.9]{HA}, we have
\[
    \calg(\cC)^{BA} \simeq
    \calg(\cC^{BA}) \iso 
    \calg(\Fun(A^*,\cC)_\Day) \simeq
    \Fun^\tlax(A^*,\cC).
\]

\begin{rem}
    Informally, this equivalence expresses the fact that for $R\in\calg(\cC)^{BA}$, the $A$-equivariant decomposition into eigenspaces
    \[
        R \simeq \bigoplus_{\varphi\in A^*} R_\varphi(\varphi)
        \qin \cC^{BA}
    \]
    is also compatible with the \emph{multiplicative} structure. Namely, the unit and multiplication maps of $R$ decompose respectively through maps
    \[
        \one \to R_0 \quad \text{and} \quad  
        R_\varphi \otimes R_\psi \to R_{\varphi+\psi},
    \]
    in a coherent way. 
\end{rem}

Given $R\in \calg(\cC^{BA})$, we shall now express the lax symmetric monoidal functor 
$\Fur(R)\colon A^* \to \cC,$
in terms of the twisting functor $T_R$ of \Cref{Twisting_Functor}. For this, we first discuss the following general setting. Let $\cD\in \calg(\Prl)$ and let $R\in\calg(\cD)$. The functor 
$R\otimes(-)\colon \cD \to \cD$
can be made lax symmetric monoidal in two ways. First, as a composition of the functors in the free-forgetful symmetric monoidal adjunction
\[
    \cD \oto{F_R} 
    \Mod_R(\cD) \oto{U_R}
    \cD.
\]
Second, the tensor product functor 
$\cD \times \cD \oto{\otimes} \cD$
corresponds to a lax symmetric monoidal functor
\[
    S_{(-)}\colon\cD \to \Fun(\cD,\cD)_\Day,
\]
which on objects is given by 
$S_X(Y) = X\otimes Y.$
This induces a functor on the $\infty$-categories of commutative algebras
\[
    S_{(-)}\colon\calg(\cD) \to \Fun^\tlax(\cD,\cD),
\]
so that $S_R(-)=R\otimes(-)$, becomes lax symmetric monoidal. We shall need the fact that these two lax symmetric monoidal structures on the functor $R\otimes (-)$ are in fact equivalent. 

\begin{prop}\label{Lax_Two_Ways}
     Let $\cD\in \calg(\Prl)$ and let $R\in\calg(\cD)$. We have an isomorphism
     \[
        U_R\circ F_R \simeq S_R
        \qin \Fun^\lax(\cD,\cD).
     \]
\end{prop}
\begin{proof}
    For convenience we write $U:=U_R$ and $F:=F_R$.
    We observe that the composition
    \[
        \cD \times \Mod_R(\cD) \oto{F\times \Id}
        \Mod_R(\cD) \times \Mod_R(\cD) \oto{\otimes} \Mod_R(\cD)
    \]
    induces the lax symmetric monoidal functor
    \[
        S_{(-)}\circ F\colon 
        \Mod_R(\cD) \to 
        \Fun(\cD,\Mod_R(\cD))_\Day 
    \]
    Consider the following diagram of symmetric monoidal $\infty$-categories and lax symmetric monoidal functors:
\[
        \xymatrix{
            \cD\ar[d]^{F}\ar[rr]^-{S_{(-)}} &  & \Fun(\cD,\cD)_{\Day}\ar[d]^{F\circ(-)}\\
            \Mod_{R}(\cD)\ar[d]^{U}\ar[rr]^-{S_{(-)}\circ F} &  & \Fun(\cD,\Mod_{R}(\cD))_{\Day}\ar[d]^{U\circ(-)}\\
            \cD\ar[rr]^-{S_{(-)}} &  & \Fun(\cD,\cD)_{\Day}.
        }
    \]
    The top square commutes by construction. The bottom square is obtained from the top square by taking right adjoints of the vertical functors. Thus, it canonically ``lax commutes'' in the sense that we have the Beck-Chevalley natural transformation of lax symmetric monoidal functors
    \[
        \beta \colon 
        S_{U(-)} \to 
        U\circ F \circ S_{U(-)} \iso 
        U\circ S_{FU(-)} \circ F \to U \circ S_{(-)} \circ F,
    \]
    where the first and last maps are the unit and counit of the respective adjunctions. For every $M\in\Mod_R(\cD)$ and $X\in\cD$ this is the composition
    \[
        M\otimes X \to 
        R \otimes (M \otimes X) \iso 
        (R\otimes M) \otimes_R (R\otimes X) \to 
        M \otimes_R (R\otimes X),
    \]
    where the first map is induced by the unit $\one \to R$ and the last by the action $R\otimes M \to M$, and hence is an isomorphism for all $M$ and $X$. Therefore the diagram commutes up to homotopy.
    Applying $\calg(-)$ to it, we get that the composition
    \[
        \calg(\cD) \oto{S_{(-)}} 
        \Fun^{\lax}(\cD,\cD) \oto{F\circ(-)}
        \Fun^{\lax}(\cD,\Mod_{R}(\cD)) \oto{U\circ(-)}
        \Fun^{\lax}(\cD,\cD)
    \]
    can be identified with the composition
        \[
        \calg(\cD) \oto{\quad F\quad} 
        \calg_{R}(\cD) \oto{\quad U\quad}
        \calg(\cD) \oto{\quad S_{(-)}\quad}
        \Fun^{\lax}(\cD,\cD).
    \]
    Applying this to $\one\in \calg(\cD)$, we get 
    \(
        U\circ F \simeq S_R
        \,\in\, \Fun^\tlax(\cD,\cD).
    \)   
\end{proof}

\begin{prop}\label{Fourier2Twisting}
    Let $\cC\in\calg(\Pr_\add)$ with a choice of a primitive $m$-th root of unity and let $A$ be a finite $m$-torsion abelian group. For $R\in\calg(\cC^{BA})$, the functor 
    $\Fur(R)\colon A^* \to \cC$
    is homotopic, as a lax symmetric monoidal functor, to the composition
    \[
        A^* \oto{\one(-)^\vee} \cC^{BA} \oto{T_R} \cC.
    \]
\end{prop}

\begin{proof}
    Unwinding the definitions, for all $\varphi\in A^*$ we have
    \[
        \Fur(R)_\varphi \simeq 
        R(-\varphi)^{hA} \simeq 
        T_R(\one(-\varphi)).
    \]
    More precisely, we have 
    \[
        T_R \simeq 
        (-)^{hA} \circ (U_R\circ F_R)
    \]
    and
    \[
        \Fur(R) \simeq 
        (-)^{hA} \circ S_R \circ \one(-)^\vee     
    \]
    as lax symmetric monoidal functors. Thus, the claim follows from \Cref{Lax_Two_Ways}.
\end{proof}

As a consequence, we obtain a characterization of the Galois property of $R\in\calg(\cC^{BA})$, in terms of its Fourier transform $\Fur(R)\colon A^* \to \cC$. 
\begin{cor}
\label{DFT_Galois}
     Let $\cC\in\calg(\Pr_\add)$ with a choice of a primitive $m$-th root of unity and let $A$ be a finite $m$-torsion abelian group. A commutative algebra $R\in \calg(\cC^{BA})$ is Galois if and only if the lax symmetric monoidal functor $\Fur(R):A^*\to \cC$ is strong symmetric monoidal. 
\end{cor}

\begin{proof}
    Since $BA$ is $\cC$-ambidextrous, by \Cref{func_crit_galois}, $R$ is Galois if and only if $T_R$ is strong symmetric monoidal. By \Cref{Fourier2Twisting} we have an equivalence of lax symmetric monoidal functors $\Fur(R)\simeq T_R \circ \one(-)^\vee$. Thus, we wish to show that $T_R$ is strong symmetric monoidal if and only if its pre-composition with $\one(-)^\vee$ is strong symmetric monoidal. 
    Since $T_R$ is $\cC$-linear and colimit preserving (by \Cref{Twisting_Functor_Linear}), it remains to show that the image of $\one(-)^\vee$, or equivalently, of $\one(-)$ generates $\cC^{BA}$ under colimits and tensoring with objects of $\cC$. The $\infty$-category $\cC^{BA}$ is generated under colimits by the induced objects 
    \[
        e_*X = \prod_{A}X\simeq X\otimes \prod_A \one
    \] 
    for $X\in \cC$, where $e\colon \pt \to BA$ is a base-point. Consequently, it is generated under colimits and tensoring with objects of $\cC$ by the single object $\prod_{A} \one$. Finally, by \Cref{Characters}(2) applied to $X=\one$ we have an isomorphism 
    \[
    \prod_{A} \one \simeq \bigoplus_{\varphi \in A^*} \one(\varphi), 
    \]
    and hence the generator $\prod_A \one$ is a direct sum of objects in the image of $\one(-)$.

    
\end{proof}

\subsection{Galois and Picard}

Using the results of the previous subsection, we obtain the following $\infty$-categorical version of Kummer theory:
\begin{thm}[Kummer Theory]
\label{Kummer_Theory} 
    Let $\cC\in\calg(\Pr_\add)$ with a choice of a primitive $m$-th root of unity $\omega \in \mu^{\mathrm{prim}}_m(\cC) $ and let $A$ be a finite $m$-torsion abelian group. The $\cC$-Fourier transform induces an isomorphism  
    \[
        \GalExt{\cC}{A} \iso
        \Map_{\Sp^\cn}(A^*,\pic(\cC)),
    \]
    natural in the pair $(\cC,\omega)$.
    Moreover, one can replace $\pic(\cC)$ with its $1$-truncation $\pic(\cC)_{\le1}$ in the above isomorphism.
\end{thm}
\begin{proof}
    In view of \Cref{DFT_Galois}, the natural equivalence 
    \[
        \Fur\colon
        \calg(\cC^{BA})\iso
        \Fun^{\tlax}(A^*,\cC)
    \]
    restricts to a natural equivalence 
    \[
        \GalExt{\cC}{A}\iso
        \Fun^\otimes(A^*,\cC).
    \]
    Since $A^*$ is an abelian group, we have 
    \[
        \Fun^\otimes(A^*,\cC) \simeq
        \Map_{\calg(\cat_\infty)}(A^*,\cC) \simeq
        \Map_{\Sp^\cn}(A^*,\pic(\cC)).
    \]
    Finally, for $n\ge2$, we have
    \[
        \pi_n\pic({\cC}) \simeq 
        \pi_{n-1}(\one^\times) \simeq
        \pi_{n-1}(\one),
    \] 
    which is $m$-divisible (since $\cC$ admits a primitive $m$-th root of unity). Thus, we get
    \[
        \GalExt{\cC}{A} \simeq \Map_{\Sp^\cn}(A^*,\pic(\cC)) \simeq
        \Map_{\Sp^\cn}(A^*,\pic(\cC)_{\le 1}).
    \] 
\end{proof}

To summarize, given $R\in \calg(\cC^{BA})$, we have a decomposition into eigenspaces
$R \simeq \bigoplus_{\varphi\in A^*} R_\varphi$ 
as objects of $\cC$, and the unit and multiplication of $R$ are induced from maps 
\[
    \one \to R_0\quad,\quad 
    R_\varphi \otimes R_\psi \to R_{\varphi+\psi}.
\]
Now, $R$ is \emph{Galois} if and only if those maps are \emph{isomorphisms}, in which case the $R_\varphi$-s are invertible and assemble into a map
$R_{(-)}\colon A^*\to \pic(\cC).$

\begin{rem}\label{Galois_Group}
    The equivalence of \Cref{Kummer_Theory} induces an abelian group structure on the set $\pi_0(\GalExt{\cC}{A})$. In fact, this set always admits a canonical group structure, even without assuming the existence of primitive roots of unity. The objects of $\GalExt{\cC}{A}$ can be viewed as local systems of commutative algebras on $BA$. The external product $R\boxtimes S$ of two such, as a local system on $BA\times BA$, can be pushed forward along the addition map $BA\times BA\oto{\alpha} BA$ to produce a new local system $R+_AS:=\alpha_*(R\boxtimes S)$ of commutative algebras on $BA$. It can be shown that if $R$ and $S$ are $A$-Galois extensions, then $R+_AS$ is an $A$-Galois extension and that this operation endows $\pi_0(\GalExt{\cC}{A})$ with an abelian group structure. In the situation of \Cref{Kummer_Theory}, this group structure coincides with the one induced from $\pic(\cC)$.
\end{rem}

\subsubsection{Cyclic group}

We shall now analyze the case $A = \ZZ/m$ in greater detail. 
For a symmetric monoidal $\infty$-category $\cC$ and a dualizable object $X\in \cC$, we can form the symmetric monoidal dimension (a.k.a Euler characteristic) $\dim(X)\in \pi_0(\one)$ (e.g., see \cite[Defintion 2.2]{PontoShulman}). The symmetric monoidal dimension satisfies 
\[
    \dim(\one) =1\quad
    \text{and} \quad
    \dim(X\otimes Y) = \dim(X)\cdot \dim(Y).
\] 
Hence, it restricts to a group homomorphism
\(
    \dim\colon \Pic(\cC) \to (\pi_0\one)^\times.
\)
We shall now describe this homomorphism in terms of the spectrum $\pic(\cC)$.
\begin{prop}
\label{eta_dim}
    Let $\cC$ be a symmetric monoidal $\infty$-category. The homomorphism  
    \[
        \pi_0\pic(\cC) \simeq
        \Pic(\cC) \oto{\dim } 
        (\pi_0\one)^\times \simeq \pi_1\pic(\cC) 
    \]
    is given by pre-composition with the Hopf map 
    $\eta\in \pi_1(\Sph)$.
\end{prop}

\begin{proof}
    The space $\Omega^\infty\Sph$ admits a structure of a commutative monoid in $\Spc$ that we can regard as a symmetric monoidal $\infty$-category. 
    An element $Z\in \Pic(\cC)$ is classified by a map of connective spectra $\Sph\to \pic(\cC)$, which corresponds to a symmetric monoidal functor  
    $
    \Omega^\infty\Sph \to \cC
    $
    sending 
    $1\in \ZZ = \pi_0\Sph$ to $Z$. 
    Since both the dimension and pre-composition with $\eta$ are natural in $\cC$, it suffices to prove the claim for $\cC = \Sph$ and $Z = 1$. 
    
    In this case, we have 
    \[
        \dim(1)\in \pi_1\Sph \simeq \ZZ/2 \cdot \eta,
    \] 
    so we only need to show that $\dim(1)\ne 0$. 
    For this, it suffices to produce some example of an invertible object with a non-trivial dimension. For example, in $\cC =\Sp$ we have
    \[
        \dim(\Sigma\Sph) = -1
        \qin \ZZ^\times = \pi_0\Sph^\times.
    \]
\end{proof}

\begin{cor}\label{Dim_Pic_pm}
    Let $\cC$ be a symmetric monoidal $\infty$-category. For every $X\in \Pic(\cC)$, we have $\dim(X)^2=1$. In particular, if $\pi_0\one$ is a connected ring and $2$ is invertible in $\pi_0\one$, then $\dim(X)=\pm1$.
\end{cor}
\begin{proof}
    The first part follows \Cref{eta_dim} and the fact that $\eta \in \pi_1\Sph$ is 2-torsion. 
    Now, if $2$ is invertible and $\pi_0\one$ admits no non-trivial idempotents, then the only solutions to the equation $t^2-1=0$ are $t=\pm1$.
\end{proof}

Given the above, we shall be interested in the following variant of the Picard group: 
\begin{defn}\label{Def_Pic_Even}
    The \textbf{even Picard group} of a  symmetric monoidal $\infty$-category $\cC$, is the subgroup
    $\Pic^\re(\cC) \le \Pic(\cC)$ given by the kernel of the map
    \(
    \Pic(\cC) \oto{\dim} (\pi_0\one_\cC)^\times.
    \)
\end{defn}

We shall now describe the collection of isomorphism classes of $\ZZ/m$-Galois extensions in $\cC$ in terms of the homotopy groups of the Picard spectrum of $\cC$.
\begin{prop}\label{Kummer_Cyclic}
 Let $\cC\in\calg(\Pr_\add)$ with a choice of a primitive $m$-th root of unity $\omega \in \mu^{\mathrm{prim}}_m(\cC)$. We have a short exact sequence of abelian groups 
    \[
        0 \to 
        (\pi_0\one^\times) / (\pi_0\one^\times)^m \to 
        \pi_0\GalExt{\cC}{\ZZ/m} \to 
        \Pic^\re(\cC)[m] \to 
        0
    \]
    which is natural in the pair $(\cC,\omega)$. 
    Moreover, this sequence splits (though not naturally).
\end{prop}

\begin{proof}
Throughout the proof, we work in the $\infty$-category $\Sp^\cn$. In particular, for $X,Y\in \Sp^\cn$ we denote by $\hom(X,Y)$ the internal mapping object in  \emph{connective} spectra. By \Cref{Kummer_Theory}, we have a natural isomorphism
\[
    \pi_0\GalExt{\cC}{\ZZ/m} \simeq 
    \pi_0\hom(\ZZ/m,\pic(\cC)_{\le1}).
\]

Let $\Sph/\eta$ be the cofiber of the map $\Sigma \Sph \oto{\eta} \Sph$. Since $(\Sph/\eta)_{\le 1} \simeq \ZZ$, we get 
\[
\hom(\ZZ,\pic(\cC)_{\le 1}) \simeq 
\hom(\Sph/\eta,\pic(\cC)_{\le 1}).
\]
Hence, $\hom(\ZZ,\pic(\cC)_{\le 1})$ is the fiber of the map 
\[
    \pic(\cC)_{\le 1} \oto{\eta}
    \Omega\pic(\cC)_{\le 1}
    \qin \Sp^\cn.
\]
By \Cref{eta_dim}, we have
\[
    \pi_0\hom(\ZZ,\pic(\cC)_{\le1}) \simeq 
    \ker(\Pic(\cC)\oto{\dim} (\pi_0\one)^\times) \simeq
    \Pic^\re(\cC)
\]
and we also have
\[
    \pi_1\hom(\ZZ,\pic(\cC)_{\le1}) \simeq 
    \pi_1\pic(\cC)_{\le1} \simeq
    (\pi_0\one)^\times
\]
\[
    \pi_n \hom(\ZZ,\pic(\cC)_{\le1}) = 0 \quad,\quad 
    \forall n\ge2.
\]

Thus, inspecting the long exact sequence in homotopy groups associated with the natural fiber sequence 
\[
    \hom(\ZZ/m,\pic(\cC)_{\le1}) \to 
    \hom(\ZZ,\pic(\cC)_{\le1}) \oto{m}
    \hom(\ZZ,\pic(\cC)_{\le1}),
\]
we get a natural short exact sequence of abelian groups
\[
    0 \to 
    (\pi_0 \one^\times)/(\pi_0 \one^\times)^m \to 
    \pi_0\hom(\ZZ/m,\pic(\cC)_{\le 1}) \to 
    \Pic^\re(\cC)[m] \to 
    0.
\]

Further, since $\hom(\ZZ,\pic(\cC)_{\le 1})$ is a $\ZZ$-module, it splits (non-canonically) as a direct sum
\[
    \hom(\ZZ,\pic(\cC)_{\le1}) \simeq
    \Pic^\re(\cC) \oplus
    \Sigma(\pi_0\one)^\times
\] 
and thus we get a splitting for the above exact sequence.
\end{proof}

The following example shows that \Cref{Kummer_Theory} indeed generalizes classical Kummer theory for field extensions. 
\begin{example}
    For a field $k$ and $\cC = \mathrm{Vect}_k$, we have $\Pic(\cC)=0$. Hence,  if $k$ contains a primitive  $m$-th root of unity, \Cref{Kummer_Cyclic} reduces to the classical fact that the isomorphism classes of $\ZZ/m$-Galois extensions of $k$ are in bijection with the set
    $(k^\times)/(k^\times)^m.$
\end{example}

At the other extreme, we have the following:
    
\begin{example}
    Let $C$ be a smooth projective algebraic curve over an algebraically closed field $k$ whose characteristic is prime to $m$ (and hence, admits primitive $m$-th roots of unity), and let $\cC$ be the category of quasi-coherent sheaves on $C$. We have 
    $(k^\times)/(k^\times)^m = 0$, while 
    $\Pic^\re(\cC)[m]$ is the $m$-torsion of the Jacobian of $C$. 
    In this case, \Cref{Kummer_Cyclic}
    recovers the classification of cyclic $m$-covers of $C$ by the $m$-torsion points on the Jacobian.
\end{example}

\subsubsection{A $\ZZ/2$-variant}

In the case $A=\ZZ/2$, one can carry out the construction of Picard objects out of $\ZZ/2$-Galois extensions with fewer assumptions on the ambient category than in \Cref{Kummer_Cyclic}. For convenience, we shall use here the multiplicative notation $\mu_2=\{\pm1\}$, instead of the additive $\ZZ/2$, for the group of order $2$. For simplicity, we shall assume that all the $\infty$-categories under consideration are stable. 

\begin{defn}\label{R_minus_one}
    Let $\cC\in \calg(\Pr_\st)$ and let $R\in \GalExt{\cC}{\mu_2}$. We denote by $\cl{R}$ the cofiber of the unit map $\one \to R$.  
\end{defn}

When $2$ is invertible in $\cC$, and hence $-1 \in \pi_0 \one$ is a primitive 2-nd root of unity, we have by Kummer theory a splitting
\(
    R\simeq \one \oplus \cl{R},
\)
and furthermore, $\cl{R} \in \Pic(\cC)$ (see the discussion after \Cref{Kummer_Theory}). It turns out that the invertibility of $\cl{R}$ holds regardless of whether $2$ is invertible.

\begin{prop}\label{R_minus_one_Pic}
Let $\cC \in \calg(\Pr_\st)$.
For every $R\in \GalExt{\cC}{\mu_2}$, we have $\cl{R}\in \Pic(\cC)$. 
\end{prop}

\begin{proof}
If $R$ is split-Galois then
$\cl{R}\simeq \one \in \Pic(\cC)$.
We now reduce the general case to the split case. The object $\cl{R}\in \cC$ is the cofiber of a map between dualizable objects and hence dualizable (see \Cref{Stable_Dualizable}). Hence, it suffices to show that the evaluation map $\cl{R} \otimes \cl{R}^\vee \to \one$ is an isomorphism. This can be checked after applying the conservative symmetric monoidal functor 
\[
R\otimes(-)\colon \cC \to \Mod_R(\cC).
\]
The image of $R$ under this functor is split-Galois, and so the general case follows from the split case.
\end{proof}

We stress, however, that unlike the case where $2$ is invertible in $\cC$, the element $\cl{R}\in \Pic(\cC)$ need not be $2$-torsion. 

\begin{example}[{see \cite[Proposition 5.3.1]{RognesGal}}]
\label{ex_KO_KU}
    We have $\KU \in \GalExt{\Mod_{\KO}(\Sp)}{\mu_2}$. The unit map $\KO \to \KU$ fits into the (non-split) Bott periodicity cofiber sequence
    \[
        \Sigma \KO \oto{\eta} 
        \KO \to 
        \KU \qin \Mod_{\KO}(\Sp). 
    \]
    It follows that $\cl{\KU} \simeq \Sigma^{2}\KO$. Hence, by real Bott periodicity, $\cl{\KU} \in \Pic(\Mod_{\KO}(\Sp))$ is of order $4$.
\end{example}

\begin{war}
    More generally, when $2$ is not invertible in $\cC$, the function
    \[
        \cl{(-)}\colon 
        \pi_0(\GalExt{\cC}{\mu_2}) \to
        \Pic(\cC)
    \]
    need \emph{not} be a group homomorphism with respect to the group structure on the source given by \Cref{Galois_Group}.
\end{war}

\section{Higher Cyclotomic Theory}

In this section, we define and study ``higher'' cyclotomic extensions in the setting of higher semiadditive stable $\infty$-categories. These are the higher (semiadditive) height analogues of the cyclotomic extensions of \Cref{Def_Cyclo_Ext}. 
We shall work primarily in $\Prl_\st^{\sad{n}} \sseq \Prl$ for some $n\ge0$, which is the full subcategory of $\Prl$, spanned by stable $n$-semiadditive $\infty$-categories. We also fix an implicit prime $p$, with respect to which one can consider semiadditive height. We recall from \cite[Theorem C]{AmbiHeight}, that every $\infty$-category in $\Prl_\st^{\sad{n}}$ splits into a product of $\infty$-categories according to height. Moreover, the finite height factors are $\infty$-semiadditive (\cite[Theorem A]{AmbiHeight})\footnote{To be precise, the height $n=0$ factor is only \textit{$p$-typically} $\infty$-semiadditive.}. We shall mainly concentrate on the full subcategory $\Prl_{\tsadi_n} \sseq \Prl_\st^{\sad{n}}$ of those $\infty$-categories which are of height $n$.

We begin in \S4.1, by discussing primitive higher roots of unity (\Cref{Def_Higher_Roots_Of_Unity}), and continue in \S4.2, with the higher cyclotomic extensions which corepresent them (\Cref{Def_Higher_Cyclo} and \Cref{Cyclo_Rep_Primitive}).

\subsection{Higher Roots of Unity}

In \Cref{Def_Roots_Of_Unity}, we have recalled the space $\mu_m(R)$ of $m$-th roots of unity of a commutative algebra object $R$ in a symmetric monoidal $\infty$-category $\cC$.
By decomposing $m$ into a product of distinct prime powers
\(
    m=p_{1}^{r_{1}}\cdots p_{s}^{r_{s}},
\)
we obtain a decomposition of the functor $\mu_{m}\colon \calg(\cC) \to \Spc$ into a product
\[
    \mu_{m}\simeq
    \mu_{p_{1}^{r_{1}}}\times \cdots \times
    \mu_{p_{s}^{r_{s}}}.
\]
We may thus restrict attention to the case $m=p^{r}$. 
While the definition of $p^r$-th roots of unity is rather general, the notion of \emph{primitive} roots behaves well only when $R$ is \emph{$p$-divisible}, in which case $\mu_{p^r}(R)$ is \emph{discrete} (see \Cref{Roots_Discrete}). 
In the terminology of \cite[Definition 3.1.6]{AmbiHeight}, the condition that $R$ is $p$-divisible, amounts to $R$ having (semiadditive) height $0$. 
More generally, when $\cC$ is higher semiadditive,
the properties of the construction $\mu_{p^{r}}(R)$ turn out to be closely related to the height of $R$. 
To begin with,
\begin{prop}\label{Higher_Roots_Trunc}
    Let $\cC\in\calg(\Pr_{\st}^{\sad n})$
    and let $R\in\calg(\cC)$. If $R$ is of height $\le n$,
    then for all $r\in\NN$ the space $\mu_{p^{r}}(R)$ is $n$-truncated.
\end{prop}

\begin{proof}
    By \cite[Proposition 2.4.7]{AmbiHeight}, we have $R[B^{n+1}C_{p^r}]\simeq R$. We thus get a sequence of isomorphisms 
    \[
        \Omega^{n+1}\mu_{p^r}(R) \simeq \Omega^{n+1}\Map_{\calg(\Spc)}(C_{p^r},R^\times) \simeq\Map_{\calg(\Spc)}(B^{n+1}C_{p^r},R^\times) \simeq 
    \]
    \[
        \Map_{\calg(\cC)}(\one[B^{n+1}C_{p^r}],R) \simeq
        \Map_{\calg_{R}(\cC)}(R[B^{n+1}C_{p^r}],R) \simeq
        \Map_{\calg_{R}(\cC)}(R,R) \simeq
        \pt.
    \]
    Since all connected components of the space $\mu_{p^r}(R)$ are isomorphic, it follows that it is $n$-truncated.
\end{proof}

As we shall demonstrate, when $R$ is of height exactly $n$, the set $\pi_n(\mu_{p^r}(R))$ serves as a good substitute for the set $\pi_0(\mu_{p^r}(R))$ of ordinary $p^r$-th roots of unity of $R$. 
With that in mind, we introduce the following generalization of \Cref{Def_Roots_Of_Unity}:

\begin{defn}[Higher Roots of Unity]
\label{Def_Higher_Roots_Of_Unity}
    Let $\cC\in\calg(\Pr_{\st}^{\sad n})$ and let $R\in\calg(\cC)$. For every prime $p$ and $r\in\NN$,
    \begin{enumerate}
        \item We define the space of \textbf{$p^r$-th roots of unity of height $n$} in $R$ to be 
        \[
            \mu^{(n)}_{p^r}(R) := 
            \Omega^n\mu_{p^r}(R) \simeq
            \Map_{\Sp^\cn}(C_{p^r},\Omega^n R^\times).
        \]
        \item  We say that a higher root of unity
            $C_{p^r}\oto{\omega} \Omega^n R^\times$ is \textbf{primitive}, if $R$ is of height $n$ and the only commutative $R$-algebra $S$ 
            for which there exists a dotted arrow rendering the diagram of spectra
            \[
                \xymatrix{
                    C_{p^r}\ar@{->>}[d]\ar[rr] &  & 
                    \Omega^n R^{\times}\ar[d] \\
                    C_{p^{r-1}}\ar@{-->}[rr] &  & 
                    \Omega^n S^{\times}
                }
            \]
            commutative, is $S=0$. 
            We denote by 
            $\mu_{p^r}^{(n),\prim}(R) \sseq \mu_{p^r}^{(n)}(R)$ 
            the union of connected components of height $n$ primitive $p^r$-th roots of unity.
    \end{enumerate}
    By convention, a height $n$ (primitive) $p^r$-th root of unity of $\cC$ is a height $n$ (primitive) $p^r$-th root of unity of $\one_\cC$.
\end{defn}

The (higher) $p^r$-th roots of unity for various $r$-s are interrelated in two ways. First, for all $k\le r$, the \emph{surjective} group homomorphisms 
\(
    C_{p^r} \onto C_{p^k}
\)
induce, by pre-composition, natural transformations
\[
    \mu^{(n)}_{p^k}(R) \simeq
    \Map(C_{p^k}, \Omega^n R^\times) \to
    \Map(C_{p^r}, \Omega^n R^\times) \simeq 
    \mu^{(n)}_{p^r}(R).
\]
We can think of this as the inclusion of the (higher) $p^k$-th roots of unity into the (higher) $p^r$-th roots of unity.
Second, the \emph{injective} group homomorphisms
$C_{p^{r-k}} \into C_{p^r}$
induce, by pre-composition, natural transformations
\[
    (-)^{p^k}\colon
    \mu^{(n)}_{p^r}(R) \simeq
    \Map(C_{p^r}, \Omega^n R^\times) \to
    \Map(C_{p^{r-k}}, \Omega^n R^\times) \simeq 
    \mu^{(n)}_{p^{r-k}}(R).
\]
We can think of this as raising a (higher) $p^r$-th root of unity to the $p^k$-th power to get a (higher) $p^{r-k}$-th root of unity. 

\begin{prop}\label{Primitive_Power}
    Let $\cC\in\calg(\Pr_{\tsadi_n})$ and let $R\in \calg(\cC)$. For $0\le k<r$, a higher root of unity 
    $\omega \in \mu^{(n)}_{p^r}(R)$ 
    is primitive, if and only if
    $\omega^{p^k} \in \mu^{(n)}_{p^{r-k}}(R)$
    is primitive.
\end{prop}
\begin{proof}
    This follows from the definition of primitivity (\Cref{Def_Higher_Roots_Of_Unity}) and the fact that we have a pushout diagram in $\Sp^\cn$ of the form
    \[
        \xymatrix{C_{p^{r-k}}\ar@{->>}[d]\ar@{^{(}->}[r] &
        C_{p^r}\ar@{->>}[d]\\
        C_{p^{r-k-1}}\ar@{^{(}->}[r] & 
        C_{p^{r-1}}
        }
    \]
\end{proof}

\subsection{Higher Cyclotomic Extensions}

\subsubsection{Definition and properties}
We shall now mimic the construction of cyclotomic extensions, which corepresent primitive roots of unity, to produce \emph{higher} cyclotomic extensions, which corepresent primitive \emph{higher} roots of unity. 
For $\cC \in \calg(\Pr_{\st})$ and a fixed $r\in\NN$, the functor 
\[
    \mu_{p^{r}}^{(n)}\colon \calg(\cC) \to \Spc
\]
is corepresented by the group algebra
$\one[B^{n}C_{p^r}].$ 
The group homomorphism $q\colon C_{p^r} \onto C_{p^{r-1}}$ induces a map of commutative groups in spaces
$q_n\colon B^nC_{p^r} \to B^nC_{p^{r-1}}$
and hence a map of group algebras
\[
    \overline{q}_n\colon
    \one[B^n C_{p^r}] \to 
    \one[B^n C_{p^{r-1}}] 
    \qin \calg(\cC).
\]
The map $\overline{q}_n$ corepresents the inclusion 
$\mu^{(n)}_{p^{r-1}}(R) \into \mu^{(n)}_{p^r}(R)$ discussed above. 
The key point is that 
if $\cC$ is higher semiadditive of \textit{height $n$}, then we can realize $\overline{q}_n$ as a splitting of an idempotent in 
$\pi_0(\one[B^nC_{p^r}])$. 
To translate between local systems and modules we need the following general fact, which seems to be well known, but for which we could not find a reference in the literature. 

\begin{prop}\label{Locsys_Mod}
    Let $\cC\in\calg(\Prl)$ and let $B$ be a pointed connected space. There is a natural equivalence of $\cC$-linear $\infty$-categories
    \[
        \cC^B \simeq \LMod_{\one[\Omega B]}(\cC)
        \qin \Mod_\cC(\Prl).
    \]
\end{prop}
\begin{proof}
    We first consider the case $\cC = \Spc$. Let $\pt\oto{e} B$ be the base point of $B$ and let $M=e_!(\pt)$ in $\Spc^B$. The functor $F_M\colon \Spc \to \Spc^B$, which is given by point-wise product with $M$, is left adjoint to the pullback functor
    \(
        e^* \colon \Spc^B \to \Spc.
    \) 
    Since $e^*$ is itself a symmetric monoidal, conservative left adjoint, it follows from \cite[Proposition 4.8.5.8]{HA}, that $\Spc^B$ is equivalent to $\LMod_{\End(M)}(\Spc).$ 
    Finally, under the Grothendieck construction equivalence $\Spc^B \simeq \Spc_{/B}$, the object $M=e_!(\pt)$ corresponds to $\pt\oto{e}B$ and its endomorphisms are given by $\Omega B \in \alg_{\EE_1}(\Spc)$. For a general $\cC \in \calg(\Prl)$, we shall deduce the claim by tensoring the equivalence
    \[
        \Spc^B \simeq \LMod_{\Omega B}(\Spc)
        \qin \Prl
    \]
    with $\cC$ in $\Prl$. Indeed, it follows from \cite[Proposition 4.8.1.17]{HA} that 
    \(
        \cC \otimes \Spc^B \simeq \cC^B,
    \)
    and from \cite[Theorems 4.8.4.6 and 4.8.5.16]{HA} that
    \[
        \cC \otimes \LMod_{\Omega B}(\Spc) \simeq
        \LMod_{\Omega B}(\cC) \simeq
        \LMod_{\one[\Omega B]}(\cC).
    \]
\end{proof}

This allows us to use the results of  \cite[\S4.3]{AmbiHeight}, to deduce the following:
\begin{prop}
\label{Transfer_Idempotent}
    Let $\cC\in\calg(\Pr_{\tsadi_n})$. There exists an idempotent 
    $\varepsilon \in \pi_0(\one[B^nC_{p^r}]),$ 
    such that
    \[
        \one[B^nC_{p^r}][\varepsilon^{-1}] \simeq
        \one[B^nC_{p^{r-1}}],
    \]
    and under this isomorphism, the canonical map $\one[B^nC_{p^r}]
    \to \one[B^nC_{p^r}][\varepsilon^{-1}]$ 
    is identified with $\overline{q}_n$.
\end{prop} 
\begin{proof}
    By the naturality of the equivalence of $\infty$-categories in \Cref{Locsys_Mod} 
    \[
        \cC^{B^{n+1}C_{p^r}} \simeq
        \Mod_{\one[B^nC_{p^r}]}(\cC),
    \]
    restriction of scalars along $\overline{q}_n$ is identified with the functor 
    $q_{n+1}^*\colon \cC^{B^{n+1}C_{p^{r-1}}} \to \cC^{B^{n+1}C_{p^{r}}}.$
    By \cite[Theorem 4.3.2]{AmbiHeight}, this induces an equivalence of $\infty$-categories 
    \[
        \cC^{B^{n+1}C_{p^r}} \iso 
        \cC^{B^{n+1}C_{p^{r-1}}} \times (\cC^{B^{n+1}C_{p^{r-1}}})^\perp,   
    \]
    where 
    $(\cC^{B^{n+1}C_{p^{r-1}}})^\perp \sseq \cC^{B^{n+1}C_{p^r}}$ 
    is the full subcategory spanned by the objects $X$, for which $(q_{n+1})_*X=0$. 
    Let 
    \[
        \varepsilon\colon \Id_{\cC^{B^{n+1}C_{p^r}}} \to \Id_{\cC^{B^{n+1}C_{p^r}}}
    \]
    be the idempotent natural endomorphism which projects onto the essential image of  $\cC^{B^{n+1}C_{p^{r-1}}}$ under the functor $q_{n+1}^*$. This corresponds to a natural endomorphism
    \[
        \overline{\varepsilon}\colon \Id_{\Mod_{\one[B^{n}C_{p^r}]}(\cC)} \to \Id_{\Mod_{\one[B^{n}C_{p^r}]}(\cC)},
    \]
    which evaluates at $\one[B^{n}C_{p^r}]$ to an idempotent element in the commutative ring $\pi_0(\one[B^nC_{p^r}])$. By construction, the decomposition
    \[ 
        \one[B^nC_{p^r}] \iso \one[B^nC_{p^r}][\overline{\varepsilon}^{-1}] \times 
        \one[B^nC_{p^r}][(1-\overline{\varepsilon})^{-1}]
    \]
    identifies the projection onto the first factor with $\overline{q}_n$.
\end{proof}

In \cite[Proposition 4.3.4]{AmbiHeight}, we have also provided an explicit description of the idempotent $\varepsilon$ of \Cref{Transfer_Idempotent} in the language of local systems. Translating into the language of rings, we get the following description of $\varepsilon \in \pi_0(\one[B^nC_{p^r}])$, in terms of the higher semiadditive structure of $\cC$. The fiber of 
$q_n\colon B^nC_{p^r} \to B^nC_{p^{r-1}}$ is isomorphic to $B^nC_p$, and we thus get a map of spaces 
$\iota \colon B^nC_p \to \Omega^\infty \one[B^nC_{p^r}].$
The idempotent $\varepsilon$ can be identified with the ``average of $\iota$'', in the sense that
\[
    \varepsilon = 
    \frac{1}{|B^nC_p|}\int_{B^nC_p}\iota \qin
    \pi_0(\one[B^nC_{p^r}]).
\]
When $n=0$, we recover the classical formula of \Cref{Def_Cyclo_Ext}, for the case $m=p^r$ (see also \cite[Example 4.3.3]{AmbiHeight}).

\begin{rem}
    Consider the $\infty$-group $G=B^nC_p$
    and the canonical maps
    \[
        f \colon
        \one[G]^{hG} \to \one[G]
    \]
    and
    \[
        g \colon 
        \one[G] \to 
        \one[G]_{hG} \simeq
        \one.
    \]
    It can be shown that if $\cC$ is $\infty$-semiadditive of height $n$ (and hence in particular $G$ is \textit{$\cC$-stably dualizable} in the sense of \cite[Definition 2.3.1]{rognes2005stably}), then
    $h = g\circ f$ is invertible and $f\circ h^{-1} = \varepsilon$.
    In the case $\cC = \Sp_{K(n)}$, the fact that $h$ is invertible, which suffices for the construction of $\varepsilon$, was first observed in \cite[Example 5.4.6]{rognes2005stably}. We thank John Rognes for explaining to us this alternative description of $\varepsilon$. 
\end{rem}

We are now ready to give the main definition of the paper.
\begin{defn}[Higher Cyclotomic Extensions]\label{Def_Higher_Cyclo}
Let $\cC\in\calg(\Pr_{\tsadi_n})$. 
For every integer $r \ge 1$, we define 
\[
    \cyc{p^r}{n} := \one[B^nC_{p^r}][(1-\varepsilon)^{-1}]
    \qin \calg(\cC),
\]
where $\varepsilon \in \pi_0(\one[B^nC_{p^r}])$ is the idempotent provided by \Cref{Transfer_Idempotent}. For every $R\in\calg(\cC)$, we define 
\[
    \cyc[R]{p^r}{n} :=  R\otimes\cyc{p^r}{n}
    \qin \calg_R(\cC).
\]
We refer to it as the (height $n$) $p^r$-th \textbf{cyclotomic extension} of $R$.
\end{defn}

As promised, the higher cyclotomic extensions indeed corepresent the higher primitive roots:

\begin{prop}\label{Cyclo_Rep_Primitive}
Let $\cC\in\calg(\Pr_{\tsadi_n})$. 
The object $\cyc{p^r}{n}\in\calg(\cC)$
corepresents the functor 
\[
    \mu^{(n),\prim}_{p^r}\colon\calg(\cC)\to
    \Set \sseq \Spc.
\] 
\end{prop}

\begin{proof}
    By \Cref{Higher_Roots_Trunc}, the essential image of $\mu_{p^r}^{(n),\prim}$ is contained in the full subcategory $\Set \sseq \Spc$. Using the adjunction
    \[
        \one[-] \colon \CMon(\mathcal{S}) \adj \calg(\cC) \colon (-)^{\times},
    \]
    we see that for $R\in\calg(\cC)$, a higher root of unity
    $\one[B^nC_{p^r}] \oto{\omega} R$
    is primitive, if and only if 
    \[
        \one[B^{n}C_{p^{r-1}}]\otimes_{\one[B^{n}C_{p^r}]}R \simeq 0.
    \]
    By the decomposition
    \[
        \one[B^{n}C_{p^r}] \iso
        \one[B^{n}C_{p^{r-1}}] \times \cyc{p^r}{n},
    \]
    the above holds if and only if the map $\omega \colon \one[B^nC_{p^r}] \to R$ factors through the projection map $\one[B^nC_{p^r}] \to \cyc{p^r}{n}$.  
\end{proof}

The higher cyclotomic extensions enjoy some additional pleasant properties: 
\begin{prop}
\label{Cyclotomic_Dualizable_Faithful}
Let $\cC\in\calg(\Pr_{\tsadi_n})$.
    \begin{enumerate}
        \item $\cyc{p^r}{n}$ is dualizable as an object of $\cC$ for all $r\in\NN$.
        
        \item $\cyc{p}{n}$ is faithful. 
    \end{enumerate}
\end{prop}

\begin{proof}
    (1) We have a fiber sequence  
    \[
        \cyc{p^r}{n} \to 
        \one[B^{n}C_{p^r}] \to
        \one[B^{n}C_{p^{r-1}}]
        \qin \cC.
    \]
    Since $\cC$ is $\infty$-semiadditive, $B\otimes\one$ is a dualizable object of $\cC$ for every $\pi$-finite space $B$ (see \Cref{Dualizable_Sad}), and dualizable objects in a stable category are closed under (co)fibers (see \Cref{Stable_Dualizable}). 
    
    (2) For $n=0$ the object $\cyc{p}{n}$ is the fiber of the fold map 
    \[
        \one^{\oplus p} \simeq \one[C_p]  \to \one
    \] 
    and hence isomorphic to $\one^{\oplus (p-1)}$. Tensoring with this object is conservative since it contains the unit as a direct summand. 
    Assume now that $n\ge 1$.  
    For every object $X\in\cC$ we have a fiber sequence
    \[
        X\otimes\cyc{p}{n}\to 
        X\otimes\one[B^{n}C_p]\to 
        X
        \qin \cC.
    \]
    Therefore, $X\otimes\cyc{p}{n}=0$, if and only if the fold map $X\otimes B^{n}C_p\to X$ is an isomorphism. We wish to deduce that $X=0$. Indeed, for $n\ge 1$ by \cite[Proposition 2.4.7]{AmbiHeight}, we get that $X$ is of height $<n$. Since $\cC$ is of height $n$, the only object $X\in\cC$ which is of height $<n$ is $X=0$. 
\end{proof}

\subsubsection{Infinite cyclotomic extensions}
From \Cref{Cyclo_Rep_Primitive}, it follows in particular, that we have canonical maps 
$\cyc{p^{r-1}}{n} \to \cyc{p^{r}}{n}$ corepresenting the natural transformation  $\omega \mapsto \omega^p$ on primitive roots of unity (see \Cref{Primitive_Power}). Gathering the cyclotomic extensions for all $r\ge0$ along these maps we get:
\begin{defn}
    Let $\cC\in\calg(\Pr_{\tsadi_n})$. We define 
    \[
        \cyc{p^\infty}{n} :=
        \colim_{r\in\NN}\, \cyc{p^r}{n}.
    \]
\end{defn}
Loosely speaking, the commutative algebra  $\cyc{p^\infty}{n}$ corepresents choices of compatible systems of height $n$ primitive roots of unity 
$\omega_p, \omega_{p^2}, \omega_{p^3},\dots$
such that 
$\omega_{p^r}^p = \omega_{p^{r-1}}$ for all $r\in\NN$.
The infinite cyclotomic extension $\cyc{p^\infty}{n}$ can also be constructed directly by splitting off an idempotent from a group algebra. The group homomorphism $\ZZ_p \oto{\times p} \ZZ_p$ induces a map of commutative algebras
\[
    \overline{q}\colon
    \one[B^{n+1}\ZZ_p] \to
    \one[B^{n+1}\ZZ_p].
\]
The following is a direct analogue, and a consequence, of \Cref{Transfer_Idempotent} for the case $r=\infty$:
\begin{prop}\label{Infinite_Transfer_Idempotent}
    Let $\cC\in\calg(\Pr_{\tsadi_n})$ for some $n\ge 1$. There exists an idempotent element 
    $\varepsilon \in \pi_0(\one[B^{n+1}\ZZ_p]),$ 
    such that
    \[
        \one[B^{n+1} \ZZ_p][(1-\varepsilon)^{-1}] \simeq
        \cyc{p^\infty}{n}
    \]
    and 
    \[
        \one[B^{n+1} \ZZ_p][\varepsilon^{-1}] \simeq
        \one[B^{n+1} \ZZ_p],
    \]
    and under the second isomorphism, the canonical map 
    $\one[B^{n+1} \ZZ_p] \to \one[B^{n+1} \ZZ_p][\varepsilon^{-1}]$ is identified with $\overline{q}$. 

\end{prop}
\begin{proof}
    Using \Cref{Transfer_Idempotent} for every $r\ge0$,  and taking the colimit, we get that $C_{p^\infty} \oto {\times p} C_{p^\infty}$ 
    induces an idempotent $\varepsilon \in \pi_0(\one[B^nC_{p^\infty}]),$
    such that 
    \[
        \one[B^n C_{p^\infty}][\varepsilon^{-1}] \simeq
        \one[B^n C_{p^\infty}]
    \]
    and
    \[
        \one[B^n C_{p^\infty}][(1-\varepsilon)^{-1}] \simeq
        \cyc{p^\infty}{n}.
    \]
    Now, the short exact sequence of abelian groups
    \[
        0 \to \ZZ_p \to \QQ_p \to C_{p^\infty} \to 0
    \]
    induces a Bockstein homomorphism
    \[
        B^n C_{p^\infty} \to B^{n+1}\ZZ_p,
    \]
    which becomes an isomorphism upon $p$-completion. Since $\cC$ is assumed to be of height $\ge1$, it is $p$-complete and the result follows.
\end{proof}

\subsubsection{Equivariance and Galois}
For every $\cC\in\calg(\Pr_\st)$ and  $R\in\calg(\cC)$, the space 
\[
    \mu^{(n)}_{p^{r}}(R) = 
    \Map_{\Sp^\cn}(C_{p^r}, \Omega^n R^\times)
\]
admits a canonical action of the group $(\ZZ/p^{r})^{\times}$ by pre-composition. If $\cC$ is higher semiadditive and of height $n$, then $\mu^{(n)}_{p^{r}}(R)$ is discrete (\Cref{Higher_Roots_Trunc}). Furthermore, since for every commutative $R$-algebra $S$ the subset $\mu^{(n)}_{p^{r-1}}(S) \sseq \mu^{(n)}_{p^{r}}(S)$ is closed under the action of $(\ZZ/p^{r})^{\times}$, so is the subset of primitive roots $\mu^{(n),\prim}_{p^{r}}(R) \sseq
\mu^{(n)}_{p^{r}}(R)$. We therefore obtain an action of $(\ZZ/p^{r})^{\times}$ on the corepresenting object $
\one[\omega^{(n)}_{p^{r}}]\in\calg(\cC)$
making the map 
$\one[B^nC_{p^r}] \to \cyc{p^r}{n}$
equivariant with respect to $(\ZZ/p^{r})^{\times}$.
Given $\cC\in\calg(\Pr_{\tsadi_n})$, it is natural to ask whether the objects 
\[
    \cyc{p^r}{n}\qin
    \calg(\cC^{B(\ZZ/p^r)^\times})
\]
are \emph{Galois}. For $n=0$, this is always the case. However, for $n=1$ a counter example was constructed by Yuan in \cite{yuan2022sphere}.
In the next section, we shall address this question for higher semiadditive $\infty$-categories arising in chromatic homotopy theory.

\section{Chromatic Applications}
\label{Chromatic}

In this final section, we apply the general theory of higher cyclotomic extensions to the chromatic world and deduce the main results of the paper. We begin in \S5.1 by showing that the higher cyclotomic extensions in $\Sp_{K(n)}$ are Galois, and deduce using the results of \S2.1, that the same holds for $\Sp_{T(n)}$ (\Cref{Cyclo_Galois}). Then, in \S5.2, we review the Galois theory of $\Sp_{K(n)}$, and identify the quotients of the Morava stabilizer group corresponding to the (higher) cyclotomic extensions of $\Sph_{K(n)}$ (\Cref{Cyc_Char_K(n)} and \Cref{Unramified_K(n)_Cyclo}). 
In particular, we deduce that all the abelian Galois extensions of $\Sph_{K(n)}$ can be obtained as a combination of ordinary and higher cyclotomic extensions. 
In \S5.3, we apply the results of \S3 to relate the higher cyclotomic extensions of $\Sph_{K(n)}$ to the $K(n)$-local Picard group (\Cref{Pic_Kn} and \Cref{W_Different}). Finally, in \S5.4, we establish the consequences of the above for the Galois extensions of $\Sp_{T(n)}$ (\Cref{Tele_Gal}) and its Picard group (\Cref{Tele_Pic_Odd} and \Cref{Tele_Pic_Even}).

\subsection{Cyclotomic Galois Extensions}
We fix a natural number $n\ge1$, a prime number $p$, and a formal group law $\Gamma$ of height $n$ over $\FF_p$ (which will be kept implicit throughout). We denote by $K(n)$ and $E_n$ the Morava $K$-theory and Lubin-Tate ring spectra associated to $\Gamma$. In particular, the homotopy groups of $K(n)$ and $E_n$ are given by\footnote{In the literature, $E_n$ often denotes a closely related ring spectrum whose homotopy groups are $W(\FF_{p^n})[[u_1,\dots,u_{n-1}]][u^\pm]$.}
\[
    \pi_* K(n) = \FF_p[v_n^\pm] \quad , \quad |v_n| = 2(p^n-1)
\]
\[
    \pi_* E_n = W(\overline{\FF}_p)[[u_1,\dots,u_{n-1}]][u^\pm] \quad, \quad |u_i| = 0, \, |u| = 2.
\]
We view $E_n$ as an object of $\calg(\Sp_{K(n)})$ (see \cite{goerss2005moduli,hopkins2014elliptic}) and denote the symmetric monoidal $\infty$-category of $K(n)$-local $E_n$-modules as follows:
\[
   \Theta_n := \Mod_{E_n}(\Sp_{K(n)}).
\]
For $M\in\Theta_n$, we consider $\pi_*(M)$ as a graded module over the twisted continuous group algebra of $\Morex_n$ over $\pi_*E_n$. Namely, as an object in the category of \textit{Morava modules} (see \cite[Definition 3.37]{barthel2019chromatic}). If $\pi_{\mathrm{odd}}(M)=0$, we say that the Morava module of $M$ is \textit{even}. In this case, one can consider the equivalent data of $\pi_0(M)$ as a module over the twisted continuous group algebra of $\Morex_n$ over $\pi_0(E_n)$, which we call the \textbf{even Morava module} of $M$.
For $X\in\Sp$, we refer to the (even) Morava module of $L_{K(n)}(E_n\otimes X)$, simply as the (even) Morava module of $X$.

In addition, we let $F(n)$ be some finite spectrum of type $n$, with a $v_n$-self map $v\colon \Sigma^d F(n) \to F(n)$, and an associated ''telescope'' 
\[
    T(n) := F(n)[v^{-1}] = 
    \colim (F(n) \oto{v} \Sigma^{-d}F(n) \oto{v} \Sigma^{-2d}F(n) \oto{v} \dots).
\]

The $\infty$-categories $\Sp_{K(n)}$, $\Sp_{T(n)}$, and $\Theta_n$ are all $\infty$-semiadditive and of height $n$ (\cite[Propostion 4.4.4 and Theorem 4.4.5]{AmbiHeight}, see also \cite[Theorem 5.2.1]{AmbiKn} and \cite[Theorem A]{TeleAmbi}). That is, we have 
\[
    \Sp_{K(n)},\;\Sp_{T(n)},\; \Theta_n
    \qin \calg(\Prl_{\tsadi_n}).
\]
Thus, we can consider height $n$ cyclotomic extensions in each one of them. 
Our first goal is to show that all these extensions are Galois. We begin by showing that in $\Theta_n$, the higher cyclotomic extensions are, in fact, \emph{split Galois} (\Cref{Galois_Split}). 
\begin{prop}\label{En_Split_Galois}
    For every $r\in\NN$, there is a $(\ZZ/p^r)^\times$-equivariant commutative ring isomorphism
    \[
        E_n[\omega^{(n)}_{p^r}] \simeq \prod_{(\ZZ/p^r)^\times} E_n.
    \]
\end{prop}

\begin{proof}
    For every finite abelian $p$-group $A$, we denote by 
    \[
        A^*\simeq \hom(A,\QQ_p/\ZZ_p)
    \]    
    the Pontryagin dual of $A$. By \cite[Corollary 5.3.26]{AmbiKn}, we have an isomorphism
    \[
        E_n[B^nA] \simeq E_n^{A^*} \qin 
        \calg(\Sp_{K(n)}),
    \]
    which is furthermore natural in $A$. In particular, when $A$ is of exponent $p^r$, this isomorphism is equivariant with respect to the $(\ZZ/p^r)^\times$-action on $A$ given by scalar multiplications. Consider the $(\ZZ/p^r)^\times$-equivariant decomposition
    \[
        E_n[B^nC_{p^r}] \simeq E_n[B^nC_{p^{r-1}}] \times E_n[\omega^{(n)}_{p^r}].
    \]
    The group homomorphism $C_{p^r} \onto C_{p^{r-1}}$ induces an injection on Pontryagin duals, which we can identify with the embedding $C_{p^{r-1}} \into C_{p^r}$, whose image is $pC_{p^r}$. Noting that 
    \[
        C_{p^r} \smallsetminus pC_{p^r} \simeq (\ZZ/p^r)^\times,
    \] 
    it follows that we have a $(\ZZ/p^r)^\times$-equivariant isomorphism 
    \[ 
        E_n[\omega^{(n)}_{p^r}] \simeq \prod_{(\ZZ/p^r)^\times} E_n \qin
        \calg(\Sp_{K(n)}).
    \]
\end{proof}

Using nil-conservativity, we can now deduce that the higher cyclotomic extensions of $\Sp_{T(n)}$ (and hence $\Sp_{K(n)}$) are Galois as well.

\begin{prop}
    \label{Cyclo_Galois}
    For all $r\in\NN$, the $p^r$-th cyclotomic extensions in $\Sp_{K(n)}$ and $\Sp_{T(n)}$ are faithful Galois extensions.
\end{prop}

\begin{proof}
    The nilpotence theorem (\cite{nilp2}) implies that the functors
    \[
        \Sp_{T(n)} \oto{L_{K(n)}} \Sp_{K(n)} \oto{E_n\otimes(-)} \Theta_n 
    \]
    are nil-conservative (see \cite[Corollary 5.1.17]{TeleAmbi}). Moreover, $\Sph_{T(n)}[\omega^{(n)}_{p^r}]$ and  $\Sph_{K(n)}[\omega^{(n)}_{p^r}]$ are dualizable by \Cref{Cyclotomic_Dualizable_Faithful}(1). Thus, the Galois property of these extensions follows from Propositions \ref{En_Split_Galois} and \ref{Galois_Nil}. Since $\Sp_{T(n)}$ and $\Sp_{K(n)}$ are $\infty$-semiadditive, these extensions are faithful (see \Cref{rem:Galois_semiadd_faithful}).
\end{proof}

\subsection{The \textit{K}(\textit{n})-local Cyclotomic Character} 

In classical algebra, Galois theory allows one to classify the Galois extensions of a commutative ring in terms of its Galois group. For example, the sequence of $p^r$-th cyclotomic extensions $\QQ_p(\omega_{p^r})$ is classified by the ($p$-adic) cyclotomic character 
\[
    \chi \colon \Gal(\QQ_p) \to \ZZ_p^\times.
\]
By the work of Devinatz-Hopkins \cite{DH}, Rognes \cite{RognesGal}, Baker-Richter \cite{BakerRichterGalois} and Mathew \cite{AkhilGalois}, the Galois extensions of $\Sph_{K(n)}$ can be similarly classified in terms of the (extended) Morava stabilizer group. 
In this subsection, we define the higher analogue of the $p$-adic cyclotomic character for $\Sp_{K(n)}$, which classifies the higher cyclotomic extensions  $\cyc[\Sph_{K(n)}]{p^r}{n}$, and prove that it identifies with the determinant map of the Morava Stabilizer group (see \Cref{Cyc_Char_K(n)}). We also show that the canonical map $\Morex_n \to \widehat{\ZZ}$ classifies the (ordinary) prime to $p$ cyclotomic extensions (\Cref{Unramified_K(n)_Cyclo}), by analogy with the map  $\Gal(\QQ_p) \to \widehat{\ZZ}$, which classifies the maximal unramified extension of $\QQ_p$.

\subsubsection{Morava stabilizer group}
We begin with a recollection of the Galois theory of $\Sp_{K(n)}$. The commutative ring $\pi_0(E_n)$ carries the universal deformation of the formal group
$\cl{\Gamma} = \Gamma \times_{\FF_p} \cl{\FF}_p.$ 
As such, it is acted on by the following group:
\begin{defn}\label{Morava_Group}
    For every integer $n\ge1$, the height $n$ (extended) \textbf{Morava stabilizer group} $\Morex_n$ is defined to be the group of automorphisms of $\cl{\Gamma}$ over $\FF_p$. That is, the group of pairs $(\sigma,\varphi)$, where 
    $\sigma \in \Gal(\FF_p)$ 
    and 
    $\varphi\colon \sigma^*\cl{\Gamma} \iso \cl{\Gamma}.$ 
    We denote by 
    \[
        \pi \colon \Morex_n \onto 
        \Gal(\FF_p) \simeq
        \widehat{\ZZ}
    \]
    the projection $(\sigma,\varphi)\mapsto \sigma$, whose kernel is  $\aut(\cl{\Gamma}/\cl{\FF}_p)$.
\end{defn}

In \cite{DH}, Devinatz and Hopkins have lifted the canonical continuous action of the pro-finite group $\Morex_n$ on the $(p,u_1,\dots,u_{n-1})$-adic ring $\pi_0(E_n)$, to a continuous action on $E_n$ itself as an object of $\calg(\Sp_{K(n)})$, which allows taking continuous fixed points with respect to closed subgroups. This action was shown to exhibit $E_n$ as a pro-finite Galois extension of $\Sph_{K(n)}$  in \cite[Theorem 5.4.4]{RognesGal}. Furthermore, $E_n$ itself is algebraically closed (by \cite[Theorem 1.1]{BakerRichterGalois} for $p$ odd, and \cite[\S10.2 and Theorem 6.29]{AkhilGalois} for all $p$). Consequently, $\Morex_n$ classifies Galois extensions of $\Sph_{K(n)}$ via a version of the ``Galois correspondence'' that we recall (with some paraphrasing) from \cite{AkhilGalois}. For every finite group $G$ and a continuous group homomorphism $\rho\colon \Morex_n \to G$, we equip
\[
        \Cnt(G,E_n):= \prod_G E_n
        \qin \calg(\Sp_{K(n)})
\]
with the (continuous) \textbf{$\rho$-twisted action}  of $\Morex_n$, which, in addition to the standard action on each factor, permutes the factors through $\rho$ and the left regular action of $G$ on itself. 
In particular, on homotopy groups, $g \in G$ acts by the formula
\[
    g\cdot (x_h)_{h\in G} = (gx_{\rho(g)^{-1}h})_{h\in G} 
    \quad\text{for}\quad 
    (x_h)_{h\in G} \in \prod_G \pi_*(E_n).
\]
In addition, the group $G$ acts on $C(G,E_n)$ by permuting the factors through the \emph{right} regular action of $G$ on itself, and the two actions clearly commute. Thus,  $C(G,E_n)^{h\Morex_n}$ acquires a $G$-action.    

\begin{prop}[{``$K(n)$-local Galois Correspondence'', \cite[Theorem 10.9, Proposition 5.32]{AkhilGalois}}]
\label{Galois_corr_k(n)}
    Let $G$ be a finite group. Taking $\Morex_n \to G$ to $\Cnt(G,E_n)^{h\Morex_n}$, establishes a bijection 
    \[
        \{\emph{\textrm{cont. homomorphisms }}\Morex_n\to G \}\,/\,\emph{\textrm{conj.}}
        \quad \simeq \quad
        \{G\emph{\text{-Galois ext. of }}\, \Sph_{K(n)}\}\,/\, \emph{\textrm{iso}}.
    \]
    In particular, for a surjective homomorphism with kernel $U\le \Morex_n$, the corresponding Galois extension is given by $E_n^{hU}$ (with the residual $G$-action).  
\end{prop}

We deduce that the Morava module of a Galois extension $R$ can be described in terms of the $\rho$-twisted action.

\begin{prop}
\label{inv_Gal_K(n)}
For $R\in \GalExt{\Sp_{K(n)}}{G}$ classified by $\rho \colon \Morex_n \to G$, the Morava module of $R$ is even, and there is an isomorphism of even Morava modules
\[
    \pi_0(E_n\otimes R)\simeq \Cnt(G,\pi_0(E_n)),
\] 
where on the left hand side the action of $\Morex_n$ is induced from the action on $E_n$ and the trivial action on $R$, and on the right hand side, it is the $\rho$-twisted action. In particular, the $\Morex_n$ action on $\pi_0(E_n \otimes R)$ determines $\rho$ up to conjugation.
\end{prop}

\begin{proof}
By \Cref{Galois_corr_k(n)}, we have 
\(
    R\simeq \Cnt(G,E_n)^{h\Morex_n}.
\) 
By \cite[Theorem 1(iii)]{DH}, we have a canonical $G$-equivariant isomorphism
\[
     \pi_*(E_n \otimes \Cnt(G,E_n)^{h\Morex_n}) \iso 
    C(G, C(\GG_n, \pi_* E_n))^{hG} \simeq
    \pi_*\Cnt(G,E_n),
\]
so in particular this holds on the level of $\pi_0$.

\end{proof}

\subsubsection{The $p$-adic cyclotomic character}

By \Cref{Galois_corr_k(n)}, the cyclotomic extensions $\cyc[\Sph_{K(n)}]{p^r}{n}$ are classified by a sequence of homomorphisms  
$\chi_r \colon \Morex_n \to (\ZZ/p^r)^\times$. Since $\chi_{r+1}$ identifies with $\chi_r$ upon reduction modulo $p^r$ for every $r\ge 0$, these assemble into a single continuous group homomorphism \[
    \chi\colon \Morex_n \to 
    \ZZ_p^\times \simeq \invlim_{r\in\NN}\, (\ZZ/p^r)^\times.
\]
The map $\chi$ thus classifies, via the $K(n)$-local Galois correspondence, the infinite cyclotomic extension
\[
    \cyc[\Sph_{K(n)}]{p^\infty}{n} =
    \colim_{r\in\NN}\,\, \cyc[\Sph_{K(n)}]{p^r}{n}.
\]

\begin{defn}\label{Def_Cyclo_Char_Kn}
    We refer to 
    $\chi\colon\Morex_n\to \ZZ_p^\times$
    as the \textbf{$p$-adic cyclotomic character} of $\Sp_{K(n)}$. 
\end{defn}

We would like to describe the $p$-adic cyclotomic character in terms of the description of $\Morex_n$ as the group of automorphisms of the formal group $\cl{\Gamma}$. The case of an odd prime $p$ for the Honda formal group law was carried out in \cite{Westerland}. The general case follows similarly using the results of \cite{AmbiKn} expressing the $K(n)$-homology of Eilenberg-Maclane spaces in terms of alternating powers of the associated formal group following \cite{ravenel1980morava}. For completeness, we shall provide the details of the argument.

From now on, we shall consider $\cl{\Gamma}$ as a connected $p$-divisible group $\cl{\Gamma} = \colim \cl{\Gamma}[p^r]$ with $\cl{\Gamma}[p^r]$ the corresponding finite flat group schemes of $p^r$-torsion. As in  \cite[Construction 3.2.1]{AmbiKn}, for a finite flat group scheme $G$ and an integer $d\ge 1$, one associates a finite flat group scheme $\Alt_G^{(d)}$ called the $d$-th alternating power of $G$.  Again as in \cite[Corollarly 3.5.4]{AmbiKn}, given a $p$-divisible group $\Upsilon$ of dimension $1$, we can now assemble the finite flat group schemes of $p$-power torsion $\Alt_{\Upsilon[p^r]}^{(d)}$ to a $p$-divisible group \[
\Alt^{(d)}_{\Upsilon} := \colim \Alt^{(d)}_{\Upsilon[p^r]}.
\]
Moreover, when $\Upsilon$ is of height $m$ the $p$-divisible group $\Alt^{(d)}_{\Upsilon}$ is of height $\binom{m}{d}$ and dimension $\binom{m-1}{d}$.
In particular, when $d=m$ the $p$-divisible group $\Alt_{\Upsilon}^{(m)}$ is \'{e}tale of height $1$. 

Returning to our $p$-divisible group $\cl{\Gamma}$, the top alternating power $\Alt^{(n)}_{\cl{\Gamma}}$ is an \'{e}tale $p$-divisible  group of height $1$ over the algebraically closed field $\cl{\FF}_p$, and hence its $\cl{\FF}_p$-points identify non-canonically with the group $\QQ_p/\ZZ_p$. As a result there is a \emph{canonical} isomorphism 
\[
    \aut(\Alt^{(n)}_{\cl{\Gamma}}(\cl{\FF}_p))\simeq \ZZ_p^\times.
\]
 
The group $\Morex_n$ acts on both $\Alt^{(n)}_{\cl{\Gamma}}$ (by functoriality) and on $\cl{\FF}_p$ via $\pi\colon \Morex_n \onto \Gal(\FF_p)$, hence it acts on the group of $\cl{\FF}_p$-points $\Alt^{(n)}_{\cl{\Gamma}}(\cl{\FF}_p)$.

\begin{prop}\label{cyclo_char_alt_power}
    The cyclotomic character $\chi\colon \Morex_n \to \ZZ_p^\times$ identifies with the map
    \[
        \Morex_n \longrightarrow \aut(\Alt^{(n)}_{\cl{\Gamma}}(\cl{\FF}_p)) \simeq \ZZ_p^\times,
    \]
    classifying the action discussed above.
\end{prop}

\begin{proof}

    It suffices to show that for each $r \ge 0$ the map in the statement agrees with $\chi$ after reduction modulo $p^r$, which we denote by 
    $\chi_r \colon \Morex_n \to (\ZZ/p^r)^\times$. 
    Via the Galois correspondence $\chi_r$ corresponds to the finite cyclotomic extension $\cyc[\Sph_{K(n)}]{p^r}{n}$. Hence, by \Cref{inv_Gal_K(n)}, we have a $\Morex_n$-equivariant isomorphism
    \[
        \pi_0(E_n\otimes \cyc[\Sph_{K(n)}]{p^r}{n}) \simeq C((\ZZ/p^r)^\times,\pi_0E_n),
    \]
    where on the right hand side we have the so-called $\chi_r$-twisted action. By collecting together the terms in the decomposition
    \[
        \Sph_{K(n)}[B^nC_{p^r}] \simeq
        \cyc[\Sph_{K(n)}]{p^0}{n} \oplus
        \cyc[\Sph_{K(n)}]{p^1}{n} \oplus \dots \oplus
        \cyc[\Sph_{K(n)}]{p^r}{n},
    \]
    and reducing modulo the maximal ideal of $\pi_0(E_n)$
    we similarly get a $\Morex_n$-equivariant isomorphism
    \[
        \pi_0(K(n) \otimes \Sph_{K(n)}[B^nC_{p^r}]) \simeq C(\ZZ/p^r,\cl{\FF}_p),
    \]
    where again on the right hand side we have 
    the $\chi_r$-twisted action. 
    
    On the other hand, by \cite[Theorem 2.0.1]{AmbiKn} there is a $\Morex_n$-equivariant isomorphism 
    \[
    \pi_0(K(n)\otimes \Sph_{K(n)}[B^nC_{p^r}]) \simeq\mathcal{O}(\Alt_{\cl{\Gamma}[p^r]}^{(n)}),
    \]
    where $\mathcal{O}(-)$ stands for the algebra of regular functions on a scheme.     
    Since $\Alt_{\cl{\Gamma}[p^r]}^{(n)}$ is an \'{e}tale finite flat group scheme over $\cl{\FF}_p$, its algebra of regular functions can be described as 
    \[
        \mathcal{O}(\Alt_{\cl{\Gamma}}^{(n)}[p^r])\simeq C(\Alt_{\cl{\Gamma}}^{(n)}[p^r](\cl{\FF}_p),\cl{\FF}_p)\simeq C(\ZZ/p^r,\cl{\FF}_p), 
    \]
    and this isomorphism is $\Morex_n$-equivariant, where the action on the right is twisted by the reduction modulo $p^r$ of the map in the statement. 
\end{proof}

The above result provides  a completely algebraic description of the cyclotomic character $\chi$. 
To compute it more explicitly, we need to recall some facts about the structure of the group $\Morex_n$. Let $\mathcal{D}_n$ be a division algebra over $\QQ_p$ of invariant $1/n\in\QQ/\ZZ$, and $\mathcal{O}_n\sseq\mathcal{D}_n$ the maximal order. The group of units $\mathcal{O}_n^\times \subseteq \mathcal{O}_n$ is isomorphic to $\aut(\cl{\Gamma}/\cl{\FF}_p)$, which is the kernel of $\Morex_n \xonto{\pi} \Gal(\cl{\FF}_p/\FF_p)$. 
As for any finite dimensional division algebra, there is a determinant (a.k.a reduced norm) multiplicative map $\det\colon \mathcal{D}_n \to \QQ_p$, which restricts to a group homomorphism $\det\colon\mathcal{O}_n^\times \to \ZZ_p^\times.$

\begin{thm}\label{Cyc_Char_K(n)}
The restriction of the $p$-adic cyclotomic character 
\(
    \chi \colon \Morex_n \to 
      \ZZ_p^\times
\)
to the subgroup 
$\mathcal{O}_n^\times \triangleleft \Morex_n$ 
is the determinant map.
\end{thm}

\begin{proof}
    By \Cref{cyclo_char_alt_power}, we have to show that the action of $\mathcal{O}_n^\times\subseteq \Morex_n$ on $\Alt_{\cl{\Gamma}}^{(n)}(\cl{\FF}_p) \simeq \ZZ_p^\times$ is via the determinant map.
    Recall that to the $p$-divisible group $\cl{\Gamma}$ over $\cl{\FF}_p$ we can associate its \emph{Dieudonn\'e module} $\DM(\cl{\Gamma})$, which is in particular a free $\WW(\cl{\FF}_p)$-module of rank $n$ (see \cite[\S1.3]{AmbiKn}). 
    The action of $\mathcal{O}_n\simeq \End_{\cl{\FF}_p}(\cl{\Gamma})$ on the Dieudonn\'e module  $\DM(\cl{\Gamma})$ gives rise to a $\ZZ_p$-algebra map 
    \[
    i\colon \mathcal{O}_n \to \End_{\WW(\cl{\FF}_p)}(\DM(\cl{\Gamma})) \simeq \Mat_{n\times n}(\WW(\cl{\FF}_p)).
    \]
    Extending scalars along $\ZZ_p \to \QQ_p$ we get a $\QQ_p$-algebra map 
    \[
    \mathcal{D}_n \to \End_{\WW(\cl{\FF}_p)}(\DM(\cl{\Gamma})) \otimes \QQ_p \simeq \Mat_{n\times n}(\widehat{\QQ}_p^{\mathrm{ur}}),
    \]
    where $\widehat{\QQ}_p^\mathrm{ur} = \WW(\cl{\FF}_p)[1/p]$ is the completion of the maximal unramified extension of $\QQ_p$. 
    Since $\mathcal{D}_n$ is a central division algebra of dimension $n^2$ over $\QQ_p$, after extending scalars on the source along $\QQ_p \to \widehat{\QQ}_p^\mathrm{ur}$, we get an isomorphism of $\widehat{\QQ}_p^\mathrm{ur}$-algebras
    \[
        \mathcal{D}_n\otimes_{\QQ_p}\widehat{\QQ}_p^{\mathrm{ur}} \iso \Mat_{n\times n}(\widehat{\QQ}_p^{\mathrm{ur}}).
    \]
    
    By the definition of the reduced norm, the map $\det\colon \mathcal{D}_n \to \QQ_p^\times$ is therefore the restriction of the ordinary determinant
    $\det \colon \Mat_{n\times n}(\widehat{\QQ}_p^{\mathrm{ur}})\to \widehat{\QQ}_p^{\mathrm{ur}}$ along the inclusion $\mathcal{D}_n\into \Mat_{n\times n}(\widehat{\QQ}_p^{\mathrm{ur}})$ above. Since the determinant of a matrix is given by its action on the top alternating power of a vector space, we deduce that the map 
    \[
        \det\colon \mathcal{O}_n^\times\to \ZZ_p^\times \into \WW(\cl{\FF}_p)^\times
    \] 
    can be written as the composition 
    \[
        \mathcal{O}_n^\times = \aut_{\cl{\FF}_p}(\cl{\Gamma}) \oto{\DM} \aut_{\WW(\cl{\FF}_p)}(\DM(\cl{\Gamma})) \oto{\wedge^n(-)^\vee} \aut_{\WW(\cl{\FF}_p)}(\wedge^n\DM(\cl{\Gamma})^\vee) \simeq \WW(\cl{\FF}_p)^\times.  
    \]
    
    Finally, by \cite[Theorem 3.3.1]{AmbiKn} we have a natural identification 
    \[
        \wedge_{\WW(\cl{\FF}_p)}^n\DM(\cl{\Gamma})^\vee \simeq \DM(\Alt^{(n)}_{\cl{\Gamma}}).
    \]
    We deduce that the action of $\mathcal{O}_n$ on the Dieudonn\'e module $\DM(\Alt^{(n)}_{\cl{\Gamma}})$ is via the determinant map. Since the map 
    \[
        \ZZ_p^\times \simeq \aut(\Alt^{(n)}_{\cl{\Gamma}}) \oto{\DM} \aut(\DM(\Alt^{(n)}_{\cl{\Gamma}}))\simeq \WW(\cl{\FF}_p)^\times 
    \] 
    is the canonical inclusion, this implies that the action of $\mathcal{O}_n^\times$ on $\Alt_{\cl{\Gamma}}^{(n)}$ is via the determinant map as well. 
\end{proof}

\Cref{Cyc_Char_K(n)} identifies the $p$-adic cyclotomic character 
$\chi\colon \Morex_n \to \ZZ_p^\times$
only on the kernel of the map
\(
    \pi \colon \Morex_n \onto 
    \Gal(\cl{\FF}_p/\FF_p) \simeq
    \widehat{\ZZ}. 
\)
However, it is possible to identify $\chi$ on the entire group $\Morex_n$ as well.
The choice of $\Gamma$ (namely, the choice of an $\FF_p$-form of $\cl{\Gamma}$) yields a section of $\pi$, and hence, a semidirect product decomposition 
\(
    \Morex_n \simeq
    \widehat{\ZZ} \ltimes \mathcal{O}_n^\times.
\)
It therefore remains to identify the restriction of $\chi$ to the subgroup 
$\widehat{\ZZ} \leq \Morex_n$ 
under this decomposition, which we denote by 
\[
    \chi_{\gal}\colon \widehat{\ZZ} \longrightarrow \ZZ_p^\times.
\]

While the $p$-divisible group $\Alt_{\cl{\Gamma}}^{(n)}$ is isomorphic to the constant $p$-divisible group $\underline{\QQ_p/\ZZ_p}$, this is no longer necessarily the case for $\Alt_{\Gamma}^{(n)}$. In fact, the isomorphism class of $\Alt_{\Gamma}^{(n)}$ depends on $\Gamma$, and might or might not be split. 
In general, $\Alt_{\Gamma}^{(n)}$ is an $\FF_p$-form of $\underline{\QQ_p/\ZZ_p}$, and therefore corresponds to a continuous cohomology class in
\[
    H^1_c(\Gal(\FF_p),\aut(\underline{\QQ_p/\ZZ_p}))\simeq \hom_c(\Gal(\FF_p),\ZZ_p^\times).
\]

By the classical theory of Galois forms, we have the following:
\begin{prop}\label{Cyclo_Char_Gal}
    The cohomology class classifying $\Alt_\Gamma^{(n)}$ is  $\chi_\gal\colon \widehat{\ZZ} \to \ZZ_p^\times$.  
\end{prop}

\begin{proof}
    By \Cref{cyclo_char_alt_power}, the group $\Gal(\FF_p)\simeq\widehat{\ZZ}$ acts on the $\cl{\FF}_p$-points of $\Alt_{\cl{\Gamma}}^{(n)}$ via $\chi_\gal$. By inspecting the construction of the cohomology class corresponding to an $\FF_p$-form, this action is given by the mentioned cohomology class.  
\end{proof}

Combining \Cref{Cyc_Char_K(n)} with \Cref{Cyclo_Char_Gal} we get a complete algebraic description of the $p$-adic cyclotomic character 
\[
    \chi \colon 
    \Morex_n \simeq 
    \widehat{\ZZ} \ltimes \mathcal{O}_n^\times \to
    \ZZ_p^\times.
\]
Namely, 
\[
    \chi(u,a) = \det(a) \chi_\gal(a)
    \quad \text{for all}\quad a\in \mathcal{O}_n^\times,\ u\in \widehat{\ZZ},
\]
where $\chi_\gal$ is as in \Cref{Cyclo_Char_Gal}.

\begin{example}
Assume that $\Gamma$ is \emph{normalizable} in the sense of \cite[Definition 5.3.1]{AmbiKn}. That is, we have an isomorphism $\Alt_\Gamma^{(n)}\simeq \underline{\QQ_p/\ZZ_p}$ defined over $\FF_p$. In this case, with respect to the splitting  
$\Morex_n \simeq\widehat{\ZZ} \ltimes \mathcal{O}_n^\times$
defined by $\Gamma$, the map $\chi_\gal$ is trivial and 
\[
    \chi(u,a) = \det(a) \quad \text{for all}\quad a\in \mathcal{O}_n^\times,\ u\in \widehat{\ZZ}.
\]
\end{example}

The next example is a reformulation of a computation carried out in \cite[Proposition 3.20]{Westerland}.
\begin{example}[Westerland]
    Let $p$ be an odd prime and let $\Gamma$ be the Honda formal group law of height $n$ over $\FF_p$. The form $\Alt_{\Gamma}^{(n)}$ is classified in this case by the cocycle $\chi_\gal(u) = (-1)^{u(n-1)}$. Consequently, with respect to the splitting 
    \[\Morex_n \simeq
        \widehat{\ZZ} \ltimes \mathcal{O}_n^\times
    \]    
    defined by $\Gamma$, the cyclotomic character $\chi$ is given by 
    \[
        \chi(u,a) = (-1)^{u(n-1)}\det(a) \quad \text{for all} \quad a\in \mathcal{O}_n^\times,\ u\in \widehat{\ZZ}.
    \]
    This is the map denoted by $\det_\pm$ in \cite[\S 1.1]{Westerland}. Namely, for $n$ even the Honda formal group is not normalizable, which introduces the sign factor in $\det_{\pm}$.
\end{example}


For future use, we record here a mild variation on \cite[Lemma 1.33]{bobkova2018topological} regarding the fixed points of the action of $\Morex_n$ on the ring $\pi_0E_n$.

\begin{prop}
\label{Morex_Fixed_Points}
    Let $N\triangleleft\Morex_n$ be the kernel of the cyclotomic character $\chi\colon \Morex_n \to \ZZ_p^\times$. We have
    \[
        (\pi_0E_n)^N = 
        \ZZ_p \sseq 
        \pi_0E_n.
    \]
\end{prop}

\begin{proof}  
    Recall that $\pi_0(E_n)=W(\cl{\FF}_p)[[u_1,\dots,u_{n-1}]]$, and that $W(\cl{\FF}_p)[\inv{p}]$ is isomorphic to $\widehat{\QQ}_p^{\mathrm{ur}}$, the completion of the maximal unramified extention of $\QQ_p$.
    Since $\pi_0 E_n$ is torsion free, it embeds in $\pi_0E_n[\inv{p}]$. Therefore, it suffice to show that $(\pi_0 E_n[\inv{p}])^N = \QQ_p$.
    Consider the subgroup $\mathcal{O}_n^\times \le \Morex_n$,
    and recall that the algebra $\mathcal{O}_n$ has the following presentation:
    \[
        \mathcal{O}_n \simeq 
        W(\FF_{p^n})\{S\}/(S^n=p,\,Sx=\varphi(x)S\quad\forall x\in W(\FF_{p^n})),
    \]
    where $S$ is a non-commutative indeterminate and $\varphi\colon W(\FF_{p^n}) \to W(\FF_{p^n})$ is the (unique) lift of the Frobenius endomorphism of $\FF_{p^n}$. 
    By \cite[Proposition 3.3]{DHAct}, we have an $\mathcal{O}_n^\times$-equivariant embedding\footnote{This embedding exhibits the target as the completion of the source with respect to its unique maximal ideal. We also remark that the $w_i$-s do not belong to the image of $\pi_0E_n$.}
    \[
        \pi_0 E_n[\inv{p}] \into
        \widehat{\QQ}_p^{\mathrm{ur}}[[w_1,...,w_{n-1}]],
    \]
    such that the action of $\mathcal{O}_n^\times$ on the right hand side is $\widehat{\QQ}_p^{\mathrm{ur}}$-linear, and each $x\in W(\FF_{p^n})^\times \le \mathcal{O}_n^\times$ acts on a power series $f=f(w_1,\dots,w_{n-1})$ by 
    \[
        (x\cdot f)(w_1,...,w_{n-1}) = f(\frac{\varphi(x)}{x}w_1,\dots,\frac{\varphi^{n-1}(x)}{x}w_{n-1}).
    \]
    It will suffice to show that
    \[
        \widehat{\QQ}_p^{\mathrm{ur}}[[w_1,...,w_{n-1}]]^N = \QQ_p. 
    \]
    Consider now the subgroup  
    \[
        W^{(1)}(\FF_{p^n})^\times := 
        W(\FF_{p^n})^\times \cap N\le 
        \mathcal{O}_n^\times.
    \]
   If $f$ is fixed by $N$, and hence by $W^{(1)}(\FF_{p^n})^\times$, the only monomials $w_1^{d_1}w_2^{d_2}\cdots w_{n-1}^{d_{n-1}}$, that can appear in $f$ with non-zero coefficients, are those for which
    \[
        x^{d_1+d_2+...+d_{n-1}} = \varphi(x)^{d_1}\varphi^2(x)^{d_2}\cdots\varphi^{n-1}(x)^{d_{n-1}}, 
        \quad \forall x\in W^{(1)}(\FF_{p^n})^\times.
    \]

    For a general element $x\in W(\FF_{p^n})^\times \le \mathcal{O}_n^\times,$ the determinant $\det(x)$ coincides with the norm
    $\Nm(x):=\prod_{i=0}^{n-1}\varphi^i(x).$
    Taking $p$-adic logarithm on the above displayed formula, this implies that the equation
    \[
        (d_1+d_2+\dots+d_{n-1})y = 
        d_1\varphi(y)+d_2\varphi^2(y)+\dots+d_{n-1}\varphi^{n-1}(y) \tag{$*$}
    \]
    holds for every $y \in W(\FF_{p^n})$ with  $\tr(y) = \sum_{i=0}^{n-1}\varphi^i(y) = 0$ and a sufficiently high $p$-adic valuation. Since $(*)$ is a linear equation, it in fact holds for all $y\in \QQ_p(\omega_{p^n-1}) = W(\FF_{p^n})[\inv{p}]$, such that $\tr(y) = 0$.
    We deduce, by the linear independence of the $\varphi^i$-s, that $d_1=\dots=d_{n-1}=0$. This means that
    $f$ has to be constant, i.e., an element of $\widehat{\QQ}_p^{\mathrm{ur}}\sseq \pi_0E_n[\inv{p}]$.
    
    Finally, we have a semi-direct product decomposition
    $\Morex_n \simeq  \widehat{\ZZ}\ltimes \mathcal{O}_n^\times,$
    by which we identify the topological generator $1\in\widehat{\ZZ}$ with an element $\sigma \in \Morex_n$.
    Since $\det = \Nm\colon W(\FF_{p^n})^\times \to \ZZ_p^\times$ is surjective (\cite[Proposition III.1.2]{milne1997class}), there exists an element 
    $a\in W(\FF_{p^n})^\times$, 
    with $\det(a) = \det(\sigma).$
    Thus, we get an element $a^{-1}\sigma \in N$, which acts on $\widehat{\QQ}_p^{\mathrm{ur}} \sseq \pi_0E_n[\inv{p}]$ as the Frobenius (see \cite[\S3.2.2]{barthel2019chromatic}). By  the Ax-Sen-Tate theorem, the fixed points of $a^{-1}\sigma$ on $\widehat{\QQ}_p^\mathrm{ur}$ are $\QQ_p \sseq \pi_0E_n[\inv{p}]$ (\cite{ax1970zeros}).
\end{proof}

\subsubsection{The total cyclotomic character}
We conclude this subsection by discussing the Galois extensions classified by the map $\pi \colon \Morex_n \onto \widehat{\ZZ}$ from \Cref{Morava_Group}. Roughly speaking, $\pi$  classifies the \emph{ordinary}, i.e. height 0, cyclotomic extensions of $\Sph_{K(n)}$ of order prime to $p$ (see \Cref{Unramified_K(n)_Cyclo} for the precise statement). This perspective is originally due to Rognes (see \cite[\S5.4.6]{RognesGal}) and we review it for completeness.  

We begin by considering the Galois extensions of the  $p$-complete sphere $\Sph_p \in \Sp$. Since $\Sph_p$ is connective, by \cite[Theorem 6.17]{AkhilGalois}, all Galois extensions of $\Sph_p$ (i.e. of $\Mod_{\Sph_p}$) are \emph{algebraic}. Namely, they are {\'e}tale and, by applying $\pi_0$, correspond bijectively to the (ordinary) Galois extensions of the ring $\pi_0(\Sph_p)=\ZZ_p.$
The Galois extensions of $\ZZ_p$ are in turn classified by the Galois group 
\[
    \Gal(\ZZ_p) \simeq \Gal(\FF_p) \simeq \widehat{\ZZ}.
\]
More concretely, the finite quotients $\widehat{\ZZ}\onto\ZZ/m$ correspond to the rings of Witt vectors $W(\FF_{p^m})$ with the action given by the lift of Frobenius. Hence, the corresponding Galois extensions of $\Sph_p$ are the rings of \emph{spherical} Witt vectors $\Sph W(\FF_{p^m})$, which are characterized by being {\'e}tale over $\Sph_p$ and having 
$\pi_0(\Sph W(\FF_{p^m}))\simeq W(\FF_{p^m})$
(see \cite[Example 5.2.7]{Lurie_Ell2}).

\begin{prop}\label{K(n)_Local_Witt}
    For every $m\in\NN$, the composition
    \[
        \Morex_n \oto{\pi} 
        \widehat{\ZZ} \onto \ZZ/m
    \]
    classifies the $\ZZ/m$-Galois extension 
    $L_{K(n)}\Sph W(\FF_{p^m})$
    of $\Sph_{K(n)}$.
\end{prop}
\begin{proof}
    By \Cref{inv_Gal_K(n)}, it suffices to show that $L_{K(n)}\Sph W(\FF_{p^m})$ is Galois and the even Morava module 
    \[
        \pi_0(E_n \otimes L_{K(n)}\Sph W(\FF_{p^m})) \simeq
        \pi_0(E_n) \otimes W(\FF_{p^m})
    \]
    is equivariantly isomorphic to $C(\ZZ/m,\pi_0(E_n))$ with the $\pi$-twisted $\Morex_n$-action. The second claim follows from the fact that the action of $\Morex_n$ on the coefficient ring 
    $W(\cl{\FF}_p) \sseq \pi_0(E_n)$
    factors through $\pi$ and is given again by the lift of Frobenius (see \cite[\S3.2.2]{barthel2019chromatic}). We now observe that the first claim follows from the second. Indeed, by \cite[Theorem 3.24]{Westerland}, if a $K(n)$-local commutative ring spectrum $R$ has a Morava module isomorphic to $C(G,\pi_0(E_n))$ for some $\rho\colon \GG_n \onto G$, then $R$ is isomorphic to the Galois extension $E_n^{h\ker(\rho)}$ and hence in particular Galois.

\end{proof}

\begin{rem}\label{Pi_Galois}
    In the language of \cite[Definition 6.8]{AkhilGalois}, the map $\pi\colon \Morex_n \onto \widehat{\ZZ}$ is the map induced on (weak) Galois groups by the functor 
    $L_{K(n)}\colon \Mod_{\Sph_p} \to \Sp_{K(n)}.$ 
\end{rem}

The relation to cyclotomic extensions of order prime to $p$ (i.e. of height zero) is as follows:
\begin{cor}[Rognes]\label{Unramified_K(n)_Cyclo}
    For every $m\in\NN$, the composition
    \[
        \Morex_n \xonto{\,\,\pi\,\,} 
        \widehat{\ZZ} \xonto{\quad} \ZZ/m
        \xhookrightarrow{p^{(-)}}\ZZ/(p^m-1)^\times
    \]
    classifies the (non-connected) cyclotomic Galois extension $\Sph_{K(n)}[\omega_{p^m-1}]$.
\end{cor}

\begin{proof}
    By \Cref{K(n)_Local_Witt}, it suffices to show that the composition
    \[
        f\colon \widehat{\ZZ} \onto \ZZ/m \xhookrightarrow{p^{(-)}}\ZZ/(p^m-1)^\times
    \]
    classifies $\Sph_p[\omega_{p^m-1}]$. Since all Galois extensions of $\Sph_p$ are algebraic (\cite[Theorem 6.17]{AkhilGalois}), it suffices to show that the Galois extension of $\ZZ_p=\pi_0\Sph_p$ classified by $f$ is $\ZZ_p[\omega_{p^m-1}]$.
    The splitting of the cyclotomic polynomial $\Phi_{p^m-1}(t)$ into irreducible factors over $\ZZ_p$ induces an isomorphism of the ring
    \[
        \ZZ_p[\omega_{p^m-1}] \simeq
        \ZZ_p[t]/\Phi_{p^m-1}(t)
    \] 
    with a product of $\frac{\phi(p^m-1)}{m}$ copies of $W(\FF_{p^m})$. Moreover, as a $\ZZ/(p^m-1)^\times$-equivariant ring, $\ZZ_p[\omega_{p^m-1}]$ is isomorphic to the induction of $W(\FF_{p^m})$ along the group homomorphism
    \[
        p^{(-)}\colon \ZZ/m \into \ZZ/(p^m-1)^\times
    \]
    and hence the claim follows.
\end{proof}

\begin{rem}
    For every $N\in\NN$ with $(N,p)=1$, we have $N\mid (p^m-1)$ for some $m\in\NN$. Thus, $\pi\colon \Morex_n \onto \widehat{\ZZ}$ accounts for all \textit{prime to $p$} cyclotomic extensions of $\Sph_{K(n)}$.  
\end{rem}

Taken together, $\pi$ and $\chi$ assemble into a single map 
\[
    \chi_\tot\colon
    \Morex_n \onto \widehat{\ZZ} \times \ZZ_p^\times,
\]
which we call the \textbf{(total) cyclotomic character}. We recall the following standard fact:
\begin{prop}\label{Morava_Abelianization}
    The map $\chi_\tot$ exhibits 
    $\widehat{\ZZ} \times \ZZ_p^\times$ 
    as the (profinite) abelianization of $\Morex_n$.
\end{prop}
\begin{proof}
    Let $\mathcal{D}_n$ be a division algebra over $\QQ_p$ of invariant $\frac{1}{n}$, so that $\mathcal{O}_n$ is the maximal order in $\mathcal{D}_n$. We may present $\mathcal{O}_n$ as 
    \[
            \mathcal{O}_n \simeq 
            W(\FF_{p^n})\{S\}/(S^n=p,\,Sx=\varphi(x)S\quad\forall x\in W(\FF_{p^n})).
    \]
    Then, $S$ is a  uniformizer of $\mathcal{D}_n$ and hence
    there is a split short exact sequence 
    \[
    1 \to \mathcal{O}_n^\times \hookrightarrow \mathcal{D}_n^\times \to \ZZ \to 1
    \]
    in which the second map is the $S$-adic valuation map. Since conjugation by $S$ acts by $\varphi$ on $W(\FF_{p^n})$, after profinite completion, the above short exact sequence identifies with 
    \[
    1\to \mathcal{O}_n^\times \hookrightarrow\Morex_n \xrightarrow{\pi} \widehat{\ZZ} \to 1.
    \]
    By \cite{nakayama1943129} 
    the map 
    \[
        \mathcal{D}_n^{\times} \xrightarrow{\mathrm{det}} \QQ_p^{\times}
    \]
    exhibits $\QQ_p^{\times}$ as the abelization of $\mathcal{D}_n^\times$. Taking profinite completions, we obtain that 
    \[
        \Morex_n^\ab \simeq (\widehat{\mathcal{D}_n^\times})^\ab \xrightarrow[{}^\sim]{\det} \widehat{\QQ_p^\times} \simeq \ZZ_p^\times \times \widehat{\ZZ}. 
    \]
    as claimed.

\end{proof}
 
Consequently, every abelian Galois extension of $\Sph_{K(n)}$ is a sub-extension of a cyclotomic extension, obtained by adding an ordinary root of unity of some order prime to $p$ and a higher root of unity of some $p$-power order.

\begin{rem}
    For $\QQ_p \in \calg(\Sp_{\QQ})$, considered as the extrapolation to height $n=0$  of the sequence $\Sph_{K(n)}\in \calg(\Sp_{K(n)})$, we have a completely analogous picture. By the ($p$-local) Kronecker-Weber theorem, every abelian extension of $\QQ_p$ is contained in a cyclotomic extension. Moreover, we have 
    $
        \Gal(\QQ_p)^\ab \simeq 
        \widehat{\ZZ} \times \ZZ_p^\times,
    $
    where the $\widehat{\ZZ}$ component corresponds to the maximal \emph{unramified} cyclotomic extension $\QQ_p^{\mathrm{un}} = \bigcup_m \QQ_p(\omega_{p^m-1}),$
    and the $\ZZ_p^\times$ component corresponds to the maximal \emph{ramified} cyclotomic extension $\QQ_p(\omega_{p^{\infty}}).$
\end{rem}

\subsection{Picard Groups}

In this subsection we relate the higher cyclotomic extensions of $\Sph_{K(n)}$ to the Picard group of $\Sp_{K(n)}$.

\begin{defn}
    Let $\Pic_n:=\Pic(\Sp_{K(n)}),$ and let $\Pic_n^0 \le \Pic_n$ be the (index 2) subgroup of objects $X\in\Pic_n,$ such that $E_n \otimes X \simeq E_n$ as $E_n$-modules.
\end{defn}

We also denote by $\Pic_n^{\mathrm{alg},0}$ the Picard group of the category of even Morava modules. The functor $\pi_0(E_n \otimes -)$ induces a map $\Pic_n^0 \to \Pic_n^{\mathrm{alg},0}$ (whose kernel is known as the \emph{exotic} Picard group). Furthermore, there is a canonical isomorphism \cite[Proposition     2.5]{goerss2015Picard}
\[
    \Pic_n^{\mathrm{alg},0} \simeq 
    H_c^1(\Morex_n; (\pi_0E_n)^\times).
\]

\begin{rem}\label{Pic_Cocycle}
    Since it will play a role in the sequel, we recall briefly how this identification goes. Given $M\in\Picalg_n$, we have $M\simeq\pi_0E_n$ as $\pi_0E_n$-modules. By choosing a generator $x\in M$, we associate with $M$ the function 
    \(
        \alpha_M\colon\Morex_n \to \pi_0E_n^\times
    \)
    given by $\alpha_M(\sigma) := \sigma(x)/x$. This function is a 1-cocycle, whose cohomology class $[\alpha_M] \in H_c^1(\Morex_n; (\pi_0E_n)^\times)$ is independent of the generator $x\in M$.
\end{rem}

\subsubsection{Odd prime}
We begin by considering the case where the prime $p$ is odd. First, 

\begin{lem}\label{Detect_Pic0}
    If $p$ is odd, then
    $\Pic_n^0 = \Pic^\re(\Sp_{K(n)})$.
\end{lem}
\begin{proof}
    Since $(\pi_0\Sph_{K(n)})^\mathrm{red} \simeq \ZZ_p$ (e.g. \cite[Proposition 2.2.6]{AmbiHeight}), the commutative ring $\pi_0\Sph_{K(n)}$ is connected with $2$ invertible. Hence, by \Cref{Dim_Pic_pm}, 
    every $X\in\Pic_n$ satisfies $\dim(X)=\pm1$. Applying the symmetric monoidal functor 
    \[ 
        E_n\otimes(-)\colon 
        \Sp_{K(n)} \to
        \Mod_{E_n}(\Sp_{K(n)}),
    \]
    we can test whether $\dim(X)$ is $1$ or $-1$, by looking at $\dim(E_n\otimes X)$. 
    Finally, by \cite[Theorem 8.7]{BakerRichter_Picard}, we have $\Pic(E_n)\simeq \ZZ/2$, with representatives given by $E_n$ and $\Sigma E_n$, which have dimensions $1$ and $-1$ respectively.
\end{proof}

We can now apply the Kummer theory developed in \S\ref{kummer}, to relate the $p$-th cyclotomic extension to the $(p-1)$-torsion in the Picard group of $\Sp_{K(n)}$. Namely, since the $p$-th cyclotomic extension is Galois it provides us with a distinguished Picard object. 

\begin{defn}\label{Def_Z_n}
    For $p$ odd, let $Z_n\in\Pic_n^0[p-1]$ be the Picard object corresponding to the $\ZZ/(p-1)$-Galois extension $\cyc[\Sph_{K(n)}]{p}{n}$ in $\Sp_{K(n)}$, under the map of  \Cref{Kummer_Cyclic}.
\end{defn}
That is, $Z_n$ is a $(p-1)$-torsion Picard object of dimension 1 in $\Sp_{K(n)}$, such that 
\[
    \cyc[\Sph_{K(n)}]{p}{n} \simeq 
    \bigoplus_{k=0}^{p-2} Z_n^{\otimes k}
    \qin \Sp_{K(n)}.
\]

The Picard object $Z_n$ can be characterized in an intrinsic way to $\Pic_n$ as follows:
\begin{prop}\label{Pic_Kn}
    For $p$ odd, the group $\Pic^0_n[p-1]$ is isomorphic to $\ZZ/(p-1)$ and is generated by $Z_n$.
\end{prop}

\begin{proof} 
    Using \Cref{Detect_Pic0} and \Cref{Kummer_Cyclic} together with its naturality with respect to the symmetric monoidal the functor 
    \[
        L_{K(n)}\colon
        \Mod_{\Sph_p}(\Sp) \to
        \Sp_{K(n)},
    \]
    we obtain the following commutative diagram of abelian groups:
    \[
        \xymatrix@C=1.8em{
            0\ar[r] & (\pi_0\Sph_p^{\times})/(\pi_0\Sph_p^{\times})^{p-1}\ar[d]^f\ar[r] & \pi_0\GalExt{\Sph_p}{\ZZ/(p-1)}\ar[d]^g\ar[r] & \Pic^{\ev}(\Sph_p)[p-1]\ar[d]\ar[r] & 0\\
            0\ar[r] & (\pi_{0}\Sph_{K(n)}^{\times})/(\pi_{0}\bb S_{K(n)}^{\times})^{p-1}\ar[r] & \pi_0\GalExt{\Sp_{K(n)}}{\ZZ/(p-1)}\ar[r] & \Pic^0_n[p-1]\ar[r] & 0.
        }
    \]
    First, it is well known that $\Pic(\Sph_p)\simeq \ZZ$ (e.g., see \cite[Proposition 4.13]{carmeli2022strict}), so the upper right corner vanishes.  
    In the top left corner, we have 
    \[
         (\pi_{0}\Sph_{p}^{\times})/(\pi_{0}\Sph_{p}^{\times})^{p-1} \simeq
         (\ZZ_p^\times) / (\ZZ_p^\times)^{p-1} \simeq
         \ZZ/(p-1).
    \]
    Furthermore, the left vertical map $f$ is an isomorphism. Indeed, the map 
    \[
        \ZZ_p \simeq 
        \pi_0\widehat{\Sph}_p\to 
        \pi_0\Sph_{K(n)}
    \] 
    admits a retract 
    \(
        r\colon \pi_0\Sph_{K(n)} \to 
        \ZZ_p,
    \)
    whose kernel consists of nilpotent elements (e.g. see \cite[Proposition 2.2.6]{AmbiHeight}). 
    In particular, every element in the kernel of 
    \(
        r^\times \colon 
        \pi_0\Sph_{K(n)}^\times \to 
        \ZZ_p^\times,
    \)
    is of the form $x=(1+\varepsilon)$ for some nilpotent $\varepsilon \in \pi_0\Sph_{K(n)}$. Since $p-1$ is invertible in $\pi_0\Sph_{K(n)}$ and the power series expansion of $(1+t)^{\frac{1}{p-1}}$ belongs to $\ZZ[\inv{(p-1)}][[t]]$, every such element $x$ has a $(p-1)$-st root. Hence, $r^\times$ induces an isomorphism after modding out the $(p-1)$-st powers. Since this induced isomorphism is a left-inverse of $f$, it follows that $f$ is an isomorphism as well.
    
    Next, by \Cref{Morava_Abelianization}, the map 
    \[
        (\pi,\chi) \colon 
        \Morex_n \to \widehat{\ZZ} \times \ZZ_p^\times
    \]
    exhibits the target as the abelianization of the source. Hence, $g$ can be identified with the inclusion (see \Cref{Pi_Galois}):
    \[
        \hom(\widehat{\ZZ},\ZZ/(p-1)) \into 
        \hom(\widehat{\ZZ},\ZZ/(p-1)) \oplus 
        \hom(\ZZ_p^\times,\ZZ/(p-1)).
    \]
    Since 
    \(
        \hom(\widehat{\ZZ},\ZZ/(p-1)) \simeq
        \ZZ/(p-1),
    \)
    the entire diagram can be identified with
    \[
        \xymatrix@C=1.8em{
            0\ar[r] & \ZZ/(p-1)\ar[d]^=\ar[r]^= & \ZZ/(p-1)\ar@{^(->}[d]\ar[r] & 0\ar[d]\ar[r] & 0\\
            0\ar[r] & \ZZ/(p-1)\ar@{^(->}[r] & \ZZ/(p-1)\oplus \hom(\ZZ_p^\times,\ZZ/(p-1))\ar[r] & \Pic^0_n[p-1]\ar[r] & 0,
        }    
    \]
    where both inclusions of $\ZZ/(p-1)$ are as the first summand of the target. Thus, the bottom right map restricts to an isomorphism
    \[
        \ZZ/(p-1) \simeq
        \hom(\ZZ_p^\times, \ZZ/(p-1)) \iso
        \Pic^0_n[p-1].
    \]
    Chasing through the identifications, the generator $1 \in \ZZ/(p-1)$ corresponds to the $\ZZ/(p-1)$-Galois extension $\cyc[\Sph_{K(n)}]{p}{n}$, and thus its image, $Z_n$, generates $\Pic^0_n[p-1]$.
\end{proof}

\begin{rem}
    By \cite[\S3.3]{Westerland}, the image of $Z_n$ in $\Pic_n^{\mathrm{alg},0}$ is classified by the composition
    \[
        \Morex_n \oto{\chi} 
        \ZZ_p^\times \onto 
        \FF_p^\times \oto{\tau} 
        \ZZ_p^\times \sseq
        (\pi_0E_n)^\times,
    \]
    where $\tau$ is the Teichm{\"u}ller lift. 
\end{rem}

\subsubsection{Even prime}
In the case $p=2$, we can not rely on Kummer Theory to produce Picard objects in $\Sp_{K(n)}$. However, we can use instead the variant afforded by \Cref{R_minus_one}. Recall that given $R\in \GalExt{\Sp_{K(n)}}{\mu_2}$, where $\mu_2 = \{\pm1\}$, the cofiber of the unit map $\one \to R$, denoted by  $\cl{R}$, belongs to $\Pic_n$ (\Cref{R_minus_one_Pic}). In fact, we have a somewhat stronger statement:
\begin{lem}\label{Rbar_Pic0}
    For every $R\in\GalExt{\Sp_{K(n)}}{\mu_2}$, we have $\cl{R}\in\Pic_n^0$.
\end{lem}
\begin{proof}
    Let  $R\in\GalExt{\Sp_{K(n)}}{\mu_2}$. We need to show that $E_n\otimes \cl{R} \simeq E_n$ as an $E_n$-module. For this, we first observe that $R\otimes \cl{R}$ is isomorphic to the cofiber of the unit map $\one \to R$ tensored with $R$. Since this map can be identified with the diagonal $R\to R\times R$, whose cofiber is $R$, we get that $R \otimes \cl{R}\simeq R$ as $R$-modules. Since $R$ is a Galois extension of $\Sph_{K(n)}$, there exists a map of commutative algebras $R \to E_n$.
    Base-changing from $R$ to $E_n$ along this map, we get that $E_n\otimes \cl{R}\simeq E_n$. 
\end{proof}

Thus, we get a function
\[
    \Xi \colon 
    \hom_c(\Morex_n,\mu_2) \simeq
    \pi_0\GalExt{\Sp_{K(n)}}{\mu_2}\oto{\cl{(-)}} 
    \Pic_n^0.
\]
To analyse the image of $\Xi$, we shall consider its further image in $\Picalg_n$. For this, it will be convenient to identify $\hom_c(\Morex_n,\mu_2)$ with $H_c^1(\Morex_n;\mu_2)$ for the trivial $\Morex_n$-action on $\mu_2$.
\begin{prop}
\label{Morava_Module_Pic_2}
    The composition
    \[
        H_c^1(\Morex_n;\mu_2) \simeq
        \hom_c(\Morex_n,\mu_2) \oto{\Xi}
        \Pic_n^0 \to \Picalg_n \simeq
        H_c^1(\Morex_n;\pi_0E_n^\times)
    \]
    is induced by the inclusion $\mu_2 \sseq \pi_0E_n^\times$. 
\end{prop}

\begin{proof}
    Let $\Morex_n \oto{\rho} \mu_2$ be a homomorphism, and let $R \in \GalExt{\Sp_{K(n)}}{\mu_2}$ be the Galois extension classified by $\rho$ by the Galois correspondence (\Cref{Galois_corr_k(n)}). We have an isomorphism of $\Morex_n$-equivariant $E_n$-modules 
    $E_n \otimes R \simeq \prod_{\mu_2}E_n$,
    where $\Morex_n$ acts on the right hand side via the $\rho$-twisted action (\Cref{inv_Gal_K(n)}). Hence, we can identify $\pi_0(E_n \otimes \cl{R})$ with the cokernel of the diagonal map $\pi_0 E_n \to \prod_{\mu_2} \pi_0E_n$. This cokernel can be further identified with $\pi_0E_n$, via the difference map  
    $\prod_{\mu_2} \pi_0E_n \to \pi_0E_n$. Choosing the generator $x_0\in\pi_0(E_n\otimes \cl{R})$, that corresponds  via this identification to $1\in\pi_0E_n$, we get that the action of $\sigma\in\Morex_n$ on $x_0$ is given by 
    \(
        \sigma(x_0) = \rho(\sigma) x_0.
    \)
    This implies that the image of $\cl{R}$ in $H_c^1(\Morex_n,\pi_0E_n^\times)$ is the 1-cocycle 
    $\Morex_n\oto{\rho}\mu_2 \sseq \pi_0E_n^\times$ (see \Cref{Pic_Cocycle}). 
\end{proof}

\begin{rem}
    The above shows that the composition 
    \[
        \pi_0\GalExt{\cC}{\mu_2} \oto{\cl{(-)}}
        \Pic_n^0 \to
        \Pic_n^{0,\mathrm{alg}}
    \]
    is a group homomorphism. This is in contrast to the fact that $\cl{(-)}$ itself is not (see \Cref{ex_KO_KU}).
\end{rem}

 From \Cref{Morava_Module_Pic_2} we deduce the following:
\begin{prop}
\label{Theta_Injectivity}
    The composition
    \[
        \hom_c(\ZZ_2^\times,\mu_2) \oto{\chi^*}
        \hom_c(\Morex_n,\mu_2) \oto{\:\Xi\:}
        \Pic_n^0
    \] 
    is injective.
\end{prop}

\begin{proof}
    It suffices to show that composing further with $\Pic_n^0 \to \Picalg_n$ yields an injective map. By \Cref{Morava_Module_Pic_2}, this reduces to showing that the composition 
    \[
        H_c^1(\ZZ_2^\times;\mu_2) \oto{\chi^*}
        H_c^1(\Morex_n;\mu_2) \to 
        H_c^1(\Morex_n;\pi_0E_n^\times)
    \]
    is injective. Let $N\triangleleft\Morex_n$ denote the kernel of the cyclotomic character $\chi\colon \Morex_n\onto \ZZ_2^\times$. Since, by \Cref{Morex_Fixed_Points}, we have
    $\mu_2 \sseq \ZZ_2^\times = (\pi_0E_n^\times)^N$,
    so this composition fits into the following commutative diagram:

    \[\begin{tikzcd}
    	{ H_c^1(\ZZ_2^\times;\mu_2) } & {       H_c^1(\ZZ_2^\times;\ZZ_2^\times)} & {H_c^1(\ZZ_2^\times;(\pi_0E_n^\times)^N) } \\
    	{        H_c^1(\Morex_n;\mu_2)} && {H_c^1(\Morex_n;(\pi_0E_n^\times)^N) } \\
    	& {        H_c^1(\Morex_n;\pi_0E_n^\times).}
    	\arrow[from=1-1, to=1-2]
    	\arrow["\chi^*"',from=1-1, to=2-1]
    	\arrow[Rightarrow, no head, from=1-2, to=1-3]
    	\arrow["\chi^*",from=1-3, to=2-3]
    	\arrow[from=2-3, to=3-2]
    	\arrow[from=2-1, to=2-3]
    	\arrow[from=2-1, to=3-2]
    \end{tikzcd}\]
    The top left horizontal map is injective because the residual action of $\Morex_n/N = \ZZ_2^\times$ on $(\pi_0E_n^\times)^N =\ZZ_2^\times$ is trivial. The composition of the right vertical map $\chi^*$ with the right diagonal map is the inflation map 
    \[
        H_c^1(\ZZ_2^\times;(\pi_0E_n^\times)^N) \to
        H_c^1(\Morex_n;\pi_0E_n^\times).
    \]
    The injectivity of this map is part of the inflation-restriction exact sequence in (continuous) group cohomology. It follows that the composition of the left vertical map $\chi^*$ and the left diagonal map is injective as well.
\end{proof}

In concrete terms, we have 
\[
    \hom_c(\ZZ_2^\times,\mu_2) \simeq
    \hom_c((\ZZ/8)^\times,\mu_2) \simeq
    \ZZ/2 \times \ZZ/2.
\]
The three non-zero elements correspond to the $\ZZ/2$-Galois sub-extensions of the $(\ZZ/8)^\times$-Galois cyclotomic extension $\cyc[\Sph_{K(n)}]{8}{n}$, which we denote by $R_1,R_2,R_3$. The zero element corresponds of course to the split $\ZZ/2$-Galois extension $R_0 := \prod_{\mu_2}\Sph_{K(n)}$. 

\begin{defn}\label{Def_Wi}
    For $i=0,\dots,3$, we define the Picard objects $W_i := \cl{R}_i \in \Pic_n^0$.
\end{defn}

\Cref{Theta_Injectivity} implies that $W_0(=\Sph_{K(n)}),W_1,W_2$ and $W_3$ are all different. We shall now show further that all of their (de)suspensions are different as well. 

\begin{prop}\label{W_Different}
    The various (de)suspensions of $W_0(=\Sph_{K(n)}),W_1,W_2$ and $W_3$ are all different elements of $\Pic_n$.
\end{prop}

\begin{proof}
    We need to show that if $\Sigma^{k_i}W_i = \Sigma^{k_j}W_j$, then $i=j$ and $k_i = k_j$. By (de)suspending, we may assume that $k_j = 0$, and by \Cref{Theta_Injectivity}, it suffices to show that we must have $k_i=0$ as well. Let $k=k_i$ and let $R = R_i$ and $R' = R_j$.
    By \Cref{Rbar_Pic0}, we have 
    \[
        E_n \simeq 
        E_n \otimes \cl{R'} \simeq
        E_n \otimes \Sigma^k \cl{R} \simeq 
        \Sigma^{k}E_n,
    \]
    as $E_n$-modules. Thus, we get that $k=2m$ for some $m\in\ZZ$.
    To show that $m$ must be zero, we shall consider the image of $\Sigma^{2m}\cl{R}$ in $\Picalg_n$. More specifically, since the center $\ZZ_2^\times \le \Morex_n$ acts trivially on $\pi_0 E_n^\times$ (see \cite[\S3.2.2]{barthel2019chromatic}), restriction along its inclusion into $\Morex_n$ is a map of the form
    \[
        \theta_{(-)} \colon
        \Picalg_n \simeq 
        H_c^1(\Morex_n;\pi_0E_n^\times) \to
        H_c^1(\ZZ_2^\times;\pi_0E_n^\times) \simeq
        \hom_c(\ZZ_2^\times,\pi_0E_n^\times).
    \]
    Every element of the center $a\in\ZZ_2^\times \triangleleft \Morex_n$ acts on the polynomial generator $u\in\pi_2(E_n)$ by multiplication $u\mapsto au$ (see \cite[\S3.2.2]{barthel2019chromatic}). Thus, the object $\pi_{2m}E_n \in \Picalg_n$ is  mapped to
    \[
        \theta_{\pi_{2m}(E_n)} =
        (-)^{-m} \colon
        \ZZ_2^\times \to 
        \ZZ_2^\times \sseq \pi_0E_n^\times. 
    \]
    Since we have
    \[
        \pi_0(E_n \otimes \Sigma^{2m}\cl{R}) \simeq
        (\pi_{2m}E_n)\otimes_{\pi_0E_n}\pi_0(E_n\otimes \cl{R}), 
    \]
    we get 
    \[
        \theta_{\Sigma^{2m}\cl{R}}(a) = a^{-m}\theta_{\cl{R}}(a)\quad,\quad 
        \forall a\in \ZZ_2^\times.
    \]
    If $\theta_{\Sigma^{2m}\cl{R}}$ were to be equal to $\theta_{\cl{R'}}$, it would in particular have to factor through the finite group $\mu_2\sseq \ZZ_2^\times$. However, this can not happen unless $m=0$.
\end{proof}

\subsection{Telescopic Lifts}
We can now combine the results of the previous subsections to deduce the main results of the paper regarding the Galois extensions and Picard groups of the telescopic categories $\Sp_{T(n)}$. Recall that by \Cref{rem:Galois_semiadd_faithful} in higher semiadditive $\infty$-categories such as $\Sp_{K(n)}$ and $\Sp_{T(n)}$ all finite Galois extensions are automatically faithful.
First, we have 

\begin{thm}\label{Tele_Gal}
    Let $G$ be a finite abelian group. For every $G$-Galois extension $R$ in $\Sp_{K(n)}$, there exists a $G$-Galois extension $R^f$ in $\Sp_{T(n)}$, such that $L_{K(n)}R^f \simeq R$.
\end{thm}
\begin{proof}
By \Cref{Morava_Abelianization}, the abelian Galois extensions of $\Sp_{K(n)}$ are classified by the group 
$\Morex_n^\ab \simeq \widehat{\ZZ} \times \ZZ_p^\times,$
through the homomorphism
\[
    \chi_\tot \colon \Morex_n \onto 
    \widehat{\ZZ} \times \ZZ_p^\times.
\]
Thus, it suffices to show that the Galois extensions corresponding to the finite quotients 
\[
        \Morex_n \onto  \widehat{\ZZ} \onto \ZZ/m
\]
and 
\[
        \Morex_n \onto  \ZZ_p^\times \onto  (\ZZ/p^r)^\times 
\]
can be lifted to $\Sp_{T(n)}$. For the first kind, we can take $L_{T(n)}\Sph W(\FF_{p^m})$,
which is Galois by the nil-conservativity 
of $L_{K(n)}\colon\Sp_{T(n)}\to \Sp_{K(n)}$ (see \cite[Proposition 5.1.15]{TeleAmbi}) and Propositions \ref{K(n)_Local_Witt} and \ref{Galois_Nil}. For the second kind, it follows from \Cref{Cyclo_Galois}, that we can take $\cyc[\Sph_{T(n)}]{p^r}{n}$.
\end{proof}

The proof of \Cref{Tele_Gal} shows in fact a bit more. Namely, that the telescopic lifts of the abelian Galois extensions in $\Sp_{K(n)}$ can be chosen in a ``compatible way''. In the language of \cite{AkhilGalois}, the situation can be described as follows. The functor 
$L_{K(n)}\colon \Sp_{T(n)}\to \Sp_{K(n)}$
induces a continuous homomorphism on weak 
Galois groups (\cite[Definition 6.8]{AkhilGalois})
\[
    \pi_1^{\rm{weak}}(\Sp_{K(n)}) \to 
    \pi_1^{\rm{weak}}(\Sp_{T(n)})
\]
and after passing to abelianizations, this homomorphism admits a left-inverse. Hence, 
$ \pi_1^{\rm{weak}}(\Sp_{T(n)})^\ab$ contains
\[
    \pi_1^{\rm{weak}}(\Sp_{K(n)})^\ab \simeq 
    \widehat{\ZZ}\times \ZZ_p^\times
\]
as a direct summand. 


Consider now the telescopic Picard group 
$\Pic_n^{f}:=\Pic(\Sp_{T(n)})$ and its subgroup $\Pic_n^{f,0} \le \Pic_n^{f}$ of objects that map to $\Pic_n^0$ under $K(n)$-localization.
When $p$ is odd, the cyclotomic extension $\cyc[\Sph_{T(n)}]{p}{n}$ provides us with the following:
\begin{thm}
    \label{Tele_Pic_Odd}
    For every $n\ge1$ and an odd prime $p$,
    there exists $Z_n^f \in \Pic_n^{f,0}[p-1]$, such that $L_{K(n)}Z^f_n \simeq Z_n$ (see \Cref{Def_Z_n}). In particular, $\Pic_n^{f,0}[p-1]$
    contains $\Pic_n^0[p-1] \simeq \ZZ/(p-1)$ as a direct summand. 
\end{thm}
\begin{proof}
    We define $Z_n^f$ to be the image of the $\ZZ/(p-1)$-Galois extension $\cyc[\Sph_{T(n)}]{p}{n}$ under the map of \Cref{Kummer_Cyclic}. By the naturality with respect to the functor $L_{K(n)}\colon \Sp_{T(n)} \to \Sp_{K(n)}$, we have $L_{K(n)}Z_n^f \simeq Z_n$. In view of \Cref{Pic_Kn}, this provides a section to the map
    \[
        \Pic_n^{f,0}[p-1] \to 
        \Pic^0_n[p-1] \simeq
        \ZZ/(p-1),
    \]
which proves the last claim.
\end{proof}

In the case $p=2$, the cyclotomic extension $\cyc[\Sph_{T(n)}]{8}{n}$
provides the following:
\begin{thm}
    \label{Tele_Pic_Even}
    For every $n\ge1$ and $p=2$,
    there exist objects
    $W_1^f,W_2^f,W_3^f \in \Pic_n^{f,0},$ such that $L_{K(n)}W_i^f = W_i$ (see \Cref{Def_Wi}). In particular,
    all the (de)suspensions of the $W_i^f$-s are different and non-trivial. 
\end{thm}
\begin{proof}
    Let $R_1^f,R_2^f,R_3^f \in \Pic_n^f$ be the non-trivial $\ZZ/2$-Galois sub-extensions of the $(\ZZ/8)^\times$-Galois cyclotomic extension $\cyc[\Sph_{T(n)}]{8}{n}$, corresponding to the three order $2$ subgroups of
    \[
        (\ZZ/8)^\times \simeq
        \ZZ/2 \times \ZZ/2.
    \]
    We define $W_i^f = \cl{R^f_i} \in \Pic_n^f$ for $i=1,2,3.$ 
    Since $L_{K(n)}W_i^f \simeq W_i$, the last claim follows from \Cref{W_Different}.
\end{proof}

\bibliographystyle{alpha}
\phantomsection\addcontentsline{toc}{section}{\refname}
\bibliography{cycloref}

\newcommand{\etalchar}[1]{$^{#1}$}
\begin{thebibliography}{GHMR15}

\bibitem[ABG18]{ando2018parametrized}
Matthew Ando, Andrew Blumberg, and David Gepner.
\newblock Parametrized spectra, multiplicative {T}hom spectra and the twisted
  umkehr map.
\newblock {\em Geometry \& Topology}, 22(7):3761--3825, 2018.

\bibitem[AR08]{ausoni2008chromatic}
Christian Ausoni and John Rognes.
\newblock {The chromatic red-shift in algebraic K-theory}.
\newblock {\em L`Enseignement Math{\'e}matique}, 54(2):13--15, 2008.

\bibitem[Ax70]{ax1970zeros}
James Ax.
\newblock {Zeros of polynomials over local fields - The Galois action}.
\newblock {\em Journal of Algebra}, 15(3):417--428, 1970.

\bibitem[BB19]{barthel2019chromatic}
Tobias Barthel and Agnes Beaudry.
\newblock Chromatic structures in stable homotopy theory.
\newblock {\em Handbook of Homotopy Theory}, pages 163--220, 2019.

\bibitem[BBB{\etalchar{+}}19]{BBBCX2019}
Agnes Beaudry, Mark Behrens, Prasit Bhattacharya, Dominic Culver, and Zhouli
  Xu.
\newblock {On the tmf-resolution of Z}.
\newblock {\em arXiv preprint arXiv:1909.13379}, 2019.

\bibitem[BBGS22]{barthel2022constructing}
Tobias Barthel, Agnes Beaudry, Paul~G Goerss, and Vesna Stojanoska.
\newblock Constructing the determinant sphere using a tate twist.
\newblock {\em Mathematische Zeitschrift}, 301(1):255--274, 2022.

\bibitem[BCSY22]{Fourier}
Tobias Barthel, Shachar Carmeli, Tomer~M Schlank, and Lior Yanovski.
\newblock {The Chromatic Fourier Transform}.
\newblock {\em arXiv preprint arXiv:2210.12822}, 2022.

\bibitem[BG18]{bobkova2018topological}
Irina Bobkova and Paul~G Goerss.
\newblock {Topological resolutions in $K(2)$-local homotopy theory at the prime
  2}.
\newblock {\em Journal of Topology}, 11(4):918--957, 2018.

\bibitem[Bir67]{birch1967cyclotomic}
BJ~Birch.
\newblock {Cyclotomic fields and Kummer extensions, in ``Algebraic Number
  Theory'',(Cassels JWS and Frohlich A., eds.)}, 1967.

\bibitem[BR05]{BakerRichter_Picard}
Andrew Baker and Birgit Richter.
\newblock Invertible modules for commutative-algebras with residue fields.
\newblock {\em manuscripta mathematica}, 118(1):99--119, 2005.

\bibitem[BR08]{BakerRichterGalois}
Andrew Baker and Birgit Richter.
\newblock {Galois extensions of Lubin-Tate spectra}.
\newblock {\em Homology, Homotopy and Applications}, 10(3):27--43, 2008.

\bibitem[Car22]{carmeli2022strict}
Shachar Carmeli.
\newblock On the strict {P}icard spectrum of commutative ring spectra.
\newblock {\em arXiv preprint arXiv:2208.03073}, 2022.

\bibitem[CMNN20]{clausen2020descent}
Dustin Clausen, Akhil Mathew, Niko Naumann, and Justin Noel.
\newblock {Descent and vanishing in chromatic algebraic $K$-theory via group
  actions}.
\newblock {\em arXiv preprint arXiv:2011.08233}, 2020.

\bibitem[CSY20]{AmbiHeight}
Shachar Carmeli, Tomer~M Schlank, and Lior Yanovski.
\newblock {Ambidexterity and Height}.
\newblock {\em arXiv preprint arXiv:2007.13089}, 2020.

\bibitem[CSY22]{TeleAmbi}
Shachar Carmeli, Tomer~M Schlank, and Lior Yanovski.
\newblock Ambidexterity in chromatic homotopy theory.
\newblock {\em Inventiones mathematicae}, 228(3):1145--1254, 2022.

\bibitem[Dev20]{devalapurkar2020roots}
Sanath Devalapurkar.
\newblock {Roots of unity in $K(n)$-local rings}.
\newblock {\em Proceedings of the American Mathematical Society},
  148(7):3187--3194, 2020.

\bibitem[DH95]{DHAct}
Ethan~S Devinatz and Michael~J Hopkins.
\newblock {The action of the {M}orava stabilizer group on the {L}ubin-{T}ate
  moduli space of lifts}.
\newblock {\em American Journal of Mathematics}, 117(3):669--710, 1995.

\bibitem[DH04]{DH}
Ethan~S Devinatz and Michael~J Hopkins.
\newblock Homotopy fixed point spectra for closed subgroups of the {M}orava
  stabilizer groups.
\newblock {\em Topology}, 43(1):1--47, 2004.

\bibitem[GGN16]{GepUniv}
David Gepner, Moritz Groth, and Thomas Nikolaus.
\newblock Universality of multiplicative infinite loop space machines.
\newblock {\em Algebraic \& Geometric Topology}, 15(6):3107--3153, 2016.

\bibitem[GH04]{goerss2004moduli}
Paul~G Goerss and Michael~J Hopkins.
\newblock Moduli spaces of commutative ring spectra.
\newblock {\em Structured ring spectra}, 315(151-200):22, 2004.

\bibitem[GH05]{goerss2005moduli}
Paul Goerss and Michael Hopkins.
\newblock Moduli spaces of commutative ring spectra.
\newblock {\em London Mathematical Society Lecture Note Series}, 315:151, 2005.

\bibitem[GHMR05]{goerss2005resolution}
Paul Goerss, H-W Henn, Mark Mahowald, and Charles Rezk.
\newblock {A resolution of the $K(2)$-local sphere at the prime 3}.
\newblock {\em Annals of Mathematics}, pages 777--822, 2005.

\bibitem[GHMR15]{goerss2015Picard}
Paul Goerss, Hans-Werner Henn, Mark Mahowald, and Charles Rezk.
\newblock {On Hopkins' Picard groups for the prime 3 and chromatic level 2}.
\newblock {\em Journal of Topology}, 8(1):267--294, 2015.

\bibitem[Gla16]{glasman2016day}
Saul Glasman.
\newblock Day convolution for $\infty$-categories.
\newblock {\em Mathematical Research Letters}, 23(5):1369--1385, 2016.

\bibitem[Hea15]{heard2015morava}
Drew Heard.
\newblock {{M}orava modules and the local Picard group}.
\newblock {\em Bulletin of the Australian Mathematical Society},
  92(1):171--172, 2015.

\bibitem[Heu21]{heuts2021lie}
Gijs Heuts.
\newblock {Lie algebras and $v_n$-periodic spaces}.
\newblock {\em Annals of Mathematics}, 193(1):223--301, 2021.

\bibitem[HL13]{AmbiKn}
Michael Hopkins and Jacob Lurie.
\newblock {Ambidexterity in $K(n)$-local stable homotopy theory}.
\newblock {\em preprint}, 2013.

\bibitem[HM14]{hopkins2014elliptic}
Michael~J Hopkins and Haynes~R Miller.
\newblock Elliptic curves and stable homotopy i.
\newblock {\em Topological modular forms}, 201:209--260, 2014.

\bibitem[HMS94]{hopkins1994constructions}
Michael~J Hopkins, Mark Mahowald, and Hal Sadofsky.
\newblock {Constructions of elements in Picard groups}.
\newblock {\em Contemporary Mathematics}, 158:89--89, 1994.

\bibitem[HS98]{nilp2}
Michael~J. Hopkins and Jeffrey~H. Smith.
\newblock Nilpotence and stable homotopy theory {II}.
\newblock {\em Ann. of Math. (2)}, 148(1):1--49, 1998.

\bibitem[HW22]{hahn2022redshift}
Jeremy Hahn and Dylan Wilson.
\newblock {Redshift and multiplication for truncated Brown--Peterson spectra}.
\newblock {\em Annals of Mathematics}, 196(3):1277--1351, 2022.

\bibitem[Lad13]{lader2013resolution}
Olivier Lader.
\newblock {\em Une r{\'e}solution projective pour le second groupe de {M}orava
  pour {$p\ge5$} et applications}.
\newblock PhD thesis, Universit{\'e} de Strasbourg, 2013.

\bibitem[Law20]{lawson2020roots}
Tyler Lawson.
\newblock Adjoining roots in homotopy theory.
\newblock {\em arXiv preprint arXiv:2002.01997}, 2020.

\bibitem[LMMT20]{land2020purity}
Markus Land, Akhil Mathew, Lennart Meier, and Georg Tamme.
\newblock {Purity in chromatically localized algebraic $K$-theory}.
\newblock {\em arXiv preprint arXiv:2001.10425}, 2020.

\bibitem[Lur]{HA}
Jacob Lurie.
\newblock Higher algebra.
\newblock {http://www.math.harvard.edu/~lurie/}.

\bibitem[Lur09]{htt}
Jacob Lurie.
\newblock {\em Higher topos theory}, volume 170 of {\em Annals of Mathematics
  Studies}.
\newblock Princeton University Press, Princeton, NJ, 2009.

\bibitem[Lur18]{Lurie_Ell2}
Jacob Lurie.
\newblock Elliptic cohomology {II}: Orientations.
\newblock {\em https://www.math.ias.edu/~lurie/papers/Elliptic-II.pdf}, 2018.

\bibitem[Mah81]{mahowald1981bo}
Mark Mahowald.
\newblock {bo-Resolutions}.
\newblock {\em Pacific Journal of Mathematics}, 92(2):365--383, 1981.

\bibitem[Mat16]{AkhilGalois}
Akhil Mathew.
\newblock The {G}alois group of a stable homotopy theory.
\newblock {\em Advances in Mathematics}, 291:403--541, 2016.

\bibitem[May01]{may2001additivity}
J~Peter May.
\newblock The additivity of traces in triangulated categories.
\newblock {\em Advances in Mathematics}, 163(1):34--73, 2001.

\bibitem[Mil81]{miller1981telescope}
Haynes~R Miller.
\newblock {On relations between Adams spectral sequences, with an application
  to the stable homotopy of a Moore space}.
\newblock {\em Journal of Pure and Applied Algebra}, 20(3):287--312, 1981.

\bibitem[Mil97]{milne1997class}
James~S Milne.
\newblock Class field theory.
\newblock {\em lecture notes available at http://www. math. lsa. umich.
  edu/jmilne}, 1997.

\bibitem[MS16]{mathew2016picard}
Akhil Mathew and Vesna Stojanoska.
\newblock {The Picard group of topological modular forms via descent theory}.
\newblock {\em Geometry \& Topology}, 20(6):3133--3217, 2016.

\bibitem[NM43]{nakayama1943129}
Tadasi Nakayama and Yoz{\^o} Matsushima.
\newblock {{\"U}}ber die multiplikative {G}ruppe einer p-adischen
  divisionsalgebra.
\newblock {\em Proceedings of the Imperial Academy}, 19(10):622--628, 1943.

\bibitem[PS14a]{ponto2014linearity}
Kate Ponto and Michael Shulman.
\newblock The linearity of traces in monoidal categories and bicategories.
\newblock {\em arXiv preprint arXiv:1406.7854}, 2014.

\bibitem[PS14b]{PontoShulman}
Kate Ponto and Michael Shulman.
\newblock Traces in symmetric monoidal categories.
\newblock {\em Expositiones Mathematicae}, 32(3):248--273, 2014.

\bibitem[Rav92]{ravenel1992nilpotence}
Douglas~C Ravenel.
\newblock {\em Nilpotence and periodicity in stable homotopy theory}.
\newblock Princeton University Press, 1992.

\bibitem[Rog05]{rognes2005stably}
John Rognes.
\newblock Stably dualizable groups.
\newblock {\em arXiv preprint math/0502184}, 2005.

\bibitem[Rog08]{RognesGal}
John Rognes.
\newblock {\em {Galois Extensions of Structured Ring Spectra/Stably Dualizable
  Groups: Stably Dualizable Groups}}, volume 192.
\newblock American Mathematical Soc., 2008.

\bibitem[RW80]{ravenel1980morava}
Douglas~C Ravenel and W~Stephen Wilson.
\newblock The {M}orava {K}-theories of {E}ilenberg-{M}aclane spaces and the
  {C}onner-{F}loyd conjecture.
\newblock {\em American Journal of Mathematics}, 102(4):691--748, 1980.

\bibitem[SVW99]{schwanzl1999roots}
Roland Schw{\"a}nzl, RM~Vogt, and Friedhelm Waldhausen.
\newblock {Adjoining roots of unity to $E_{\infty}$-ring spectra in good cases:
  a remark}.
\newblock In {\em Homotopy invariant algebraic structures: a conference in
  honor of J. Michael Boardman, AMS Special Session on Homotopy Theory, January
  7-10, 1998, Baltimore, M}, 1999.

\bibitem[Wes17]{Westerland}
Craig Westerland.
\newblock {A higher chromatic analogue of the image of J}.
\newblock {\em Geometry \& Topology}, 21(2):1033--1093, 2017.

\bibitem[Yua22]{yuan2022sphere}
Allen Yuan.
\newblock The sphere of semiadditive height 1.
\newblock {\em arXiv preprint arXiv:2208.12844}, 2022.

\end{thebibliography}

\end{document}